\documentclass[11pt]{article}

\usepackage{array}
\usepackage{amsfonts}
\usepackage{amscd}
\usepackage{amssymb}
\usepackage{amsthm}
\usepackage{amsmath}
\usepackage{stmaryrd}
\usepackage{graphicx}
\usepackage{color}
\usepackage{verbatim}
\usepackage{mathrsfs}
\usepackage{tikz}
\usetikzlibrary{calc}
\usepackage{caption}
\usepackage[neveradjust]{paralist}

\usepackage{float}

\colorlet{ggrey}{black!50}

 \theoremstyle{plain}
 \newtheorem{thm1}{Theorem}
 \newtheorem{cor1}[thm1]{Corollary}
\newtheorem{thm}{Theorem}[section]
\newtheorem{lemma}[thm]{Lemma}
\newtheorem{prop}[thm]{Proposition}
\newtheorem{cor}[thm]{Corollary}
\newtheorem{conjecture}[thm]{Conjecture}

\theoremstyle{definition}
\newtheorem{defn}[thm]{Definition}
\newtheorem{remark}[thm]{Remark}
\newtheorem{example}[thm]{Example}
\numberwithin{equation}{section}

\def\sA{\mathsf{A}}
\def\sB{\mathsf{B}}
\def\sC{\mathsf{C}}
\def\sD{\mathsf{D}}
\def\sE{\mathsf{E}}
\def\sF{\mathsf{F}}

\def\PG{\mathsf{PG}}

\def\cC{\mathcal{C}}

\def\cT{\mathcal{T}}

\def\FF{\mathbb{F}}

\def\KK{\mathbb{K}}

\DeclareMathOperator\type{\tau}
\DeclareMathOperator\Type{\mathrm{Typ}}
\DeclareMathOperator\Res{\mathrm{Res}}
\DeclareMathOperator\proj{\mathrm{proj}}

\DeclareMathOperator\Opp{\mathrm{Opp}}
\DeclareMathOperator\disp{\mathrm{disp}}
\DeclareMathOperator\diam{\mathrm{diam}}

\def\<{\langle}
\def\>{\rangle}

\makeatletter
\renewcommand{\@makefnmark}{\mbox{\textsuperscript{}}}
\makeatother

\title{Opposition diagrams for automorphisms of small spherical buildings}
\author{James Parkinson 
\and
Hendrik Van Maldeghem}
\date{\today}
\begin{document}

\maketitle

\begin{abstract}
An automorphism $\theta$ of a spherical building $\Delta$ is called \textit{capped} if it satisfies the following property: if there exist both type $J_1$ and $J_2$ simplices of $\Delta$ mapped onto opposite simplices by $\theta$ then there exists a type $J_1\cup J_2$ simplex of $\Delta$ mapped onto an opposite simplex by~$\theta$. In previous work we showed that if $\Delta$ is a thick irreducible spherical building of rank at least $3$ with no Fano plane residues then every automorphism of $\Delta$ is capped. In the present work we consider the spherical buildings with Fano plane residues (the \textit{small buildings}). We show that uncapped automorphisms exist in these buildings and develop an enhanced notion of ``opposition diagrams'' to capture the structure of these automorphisms. Moreover we provide applications to the theory of ``domesticity'' in spherical buildings, including the complete classification of domestic automorphisms of small buildings of types $\sF_4$ and $\sE_6$. 
\end{abstract}

%\tableofcontents

\section*{Introduction}

Let $\theta$ be an automorphism of a thick irreducible spherical building~$\Delta$ of type $(W,S)$. The \textit{opposite geometry} of $\theta$ is the set $\Opp(\theta)$ of all simplices $\sigma$ of $\Delta$ such that $\sigma$ and $\sigma^{\theta}$ are opposite in~$\Delta$. This geometry forms a natural counterpart to the more familiar fixed element geometry $\mathrm{Fix}(\theta)$, however by comparison very little is known about $\Opp(\theta)$. 

This paper is the continuation of~\cite{PVM:17a}, where we initiated a systematic study of $\Opp(\theta)$ for automorphisms of spherical buildings. In particular in \cite{PVM:17a} we showed that if $\Delta$ is a thick irreducible spherical building of rank at least~$3$ containing no Fano plane residues then $\Opp(\theta)$ has the following weak closure property:  if there exist both type $J_1$ and $J_2$ simplices in $\Opp(\theta)$ then there exists a type $J_1\cup J_2$ simplex in $\Opp(\theta)$. Automorphisms with this property are called \textit{capped}, and the thick irreducible spherical buildings of rank at least~$3$ with no Fano plane residues are called \textit{large buildings}. Thus every automorphism of a large building is capped. 

In the present paper we investigate $\Opp(\theta)$ for the thick irreducible spherical buildings of rank at least~$3$ containing a Fano plane residue. These are called the \textit{small buildings}. In particular we show that, in contrast to the case of large buildings, uncapped automorphisms exist for all small buildings (with the possible exception of $\sE_8(2)$ where we provide conjectural examples). 

A key tool in \cite{PVM:17a} was the notion of the \textit{opposition diagram} of an automorphism $\theta$, consisting of the triple $(\Gamma,J,\pi)$, where $\Gamma$ is the Coxeter graph of $(W,S)$, $J$ is the union of all $J'\subseteq S$ such that there exists a type $J'$ simplex in $\Opp(\theta)$, and $\pi$ is the automorphism of $\Gamma$ induced by~$\theta$ (less formally, the opposition diagram is drawn by encircling the nodes $J$ of $\Gamma$). If $\theta$ is capped then this diagram turns out to encode a lot of information about the automorphism, essentially because it completely determines the partially ordered set $\mathcal{T}(\theta)$ of all types of simplices mapped onto opposite simplices by~$\theta$. However for an uncapped automorphism the opposition diagram does not necessarily determine~$\cT(\theta)$. For example in the polar space $\Delta=\sB_3(2)$ there are collineations $\theta_1$, $\theta_2$ and $\theta_3$ each with opposition diagram 
\begin{tikzpicture}[scale=0.5,baseline=-0.5ex]
\node at (0,0.3) {};
\node [inner sep=0.8pt,outer sep=0.8pt] at (-1.5,0) (1) {$\bullet$};
\node [inner sep=0.8pt,outer sep=0.8pt] at (-0.5,0) (2) {$\bullet$};
\node [inner sep=0.8pt,outer sep=0.8pt] at (0.5,0) (3) {$\bullet$};
\draw (-1.5,0)--(-0.5,0);
\draw (-0.5,0.07)--(0.5,0.07);
\draw (-0.5,-0.07)--(0.5,-0.07);
\draw [line width=0.5pt,line cap=round,rounded corners] (1.north west)  rectangle (1.south east);
\draw [line width=0.5pt,line cap=round,rounded corners] (2.north west)  rectangle (2.south east);
\draw [line width=0.5pt,line cap=round,rounded corners] (3.north west)  rectangle (3.south east);
\end{tikzpicture}
(that is, each $\theta_i$ maps a vertex of each type to an opposite vertex) whose partially ordered sets $\mathcal{T}(\theta_i)$, for $i=1,2,3$, are the following (see Theorem~\ref{thm:existenceBn(2)} for explicit examples):
\smallskip

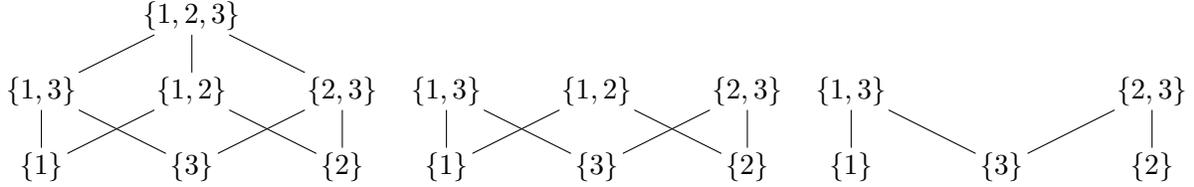
\begin{center}
\begin{tikzpicture}[scale=1,baseline=-0.5ex]
\node at (0,0.3) {};
\node [inner sep=0.8pt,outer sep=0.8pt] at (-2,-1) (1) {$\{1\}$};
\node [inner sep=0.8pt,outer sep=0.8pt] at (0,-1) (2) {$\{3\}$};
\node [inner sep=0.8pt,outer sep=0.8pt] at (2,-1) (3) {$\{2\}$};
\node [inner sep=0.8pt,outer sep=0.8pt] at (-2,0) (4) {$\{1,3\}$};
\node [inner sep=0.8pt,outer sep=0.8pt] at (0,0) (5) {$\{1,2\}$};
\node [inner sep=0.8pt,outer sep=0.8pt] at (2,0) (6) {$\{2,3\}$};
\node [inner sep=0.8pt,outer sep=0.8pt] at (0,1) (7) {$\{1,2,3\}$};
\draw (1)--(4);
\draw (1)--(5);
\draw (2)--(4);
\draw (2)--(6);
\draw (3)--(5);
\draw (3)--(6);
\draw (4)--(7);
\draw (5)--(7);
\draw (6)--(7);
\end{tikzpicture}\quad\begin{tikzpicture}[scale=1,baseline=-0.5ex]
\node at (0,0.3) {};
\node [inner sep=0.8pt,outer sep=0.8pt] at (-2,-1) (1) {$\{1\}$};
\node [inner sep=0.8pt,outer sep=0.8pt] at (0,-1) (2) {$\{3\}$};
\node [inner sep=0.8pt,outer sep=0.8pt] at (2,-1) (3) {$\{2\}$};
\node [inner sep=0.8pt,outer sep=0.8pt] at (-2,0) (4) {$\{1,3\}$};
\node [inner sep=0.8pt,outer sep=0.8pt] at (0,0) (5) {$\{1,2\}$};
\node [inner sep=0.8pt,outer sep=0.8pt] at (2,0) (6) {$\{2,3\}$};
\phantom{\node [inner sep=0.8pt,outer sep=0.8pt] at (0,1) (7) {$\{1,2,3\}$};}
\draw (1)--(4);
\draw (1)--(5);
\draw (2)--(4);
\draw (2)--(6);
\draw (3)--(5);
\draw (3)--(6);
\end{tikzpicture}\quad\begin{tikzpicture}[scale=1,baseline=-0.5ex]
\node at (0,0.3) {};
\node [inner sep=0.8pt,outer sep=0.8pt] at (-2,-1) (1) {$\{1\}$};
\node [inner sep=0.8pt,outer sep=0.8pt] at (0,-1) (2) {$\{3\}$};
\node [inner sep=0.8pt,outer sep=0.8pt] at (2,-1) (3) {$\{2\}$};
\node [inner sep=0.8pt,outer sep=0.8pt] at (-2,0) (4) {$\{1,3\}$};
\phantom{\node [inner sep=0.8pt,outer sep=0.8pt] at (0,0) (5) {$\{1,2\}$};}
\node [inner sep=0.8pt,outer sep=0.8pt] at (2,0) (6) {$\{2,3\}$};
\phantom{\node [inner sep=0.8pt,outer sep=0.8pt] at (0,1) (7) {$\{1,2,3\}$};}
\draw (1)--(4);
\draw (2)--(4);
\draw (2)--(6);
\draw (3)--(6);
\end{tikzpicture}
\end{center}
\smallskip

\noindent Note that only $\theta_1$ is capped (hence, in particular, analogues of $\theta_2$ and $\theta_3$ cannot exist for polar spaces $\sB_3(\mathbb{F})$ with $|\mathbb{F}|>2$ by the main result of \cite{PVM:17a}).

Thus the opposition diagram of an uncapped automorphism needs to be enhanced to properly understand these automorphisms. We achieve this by defining the \textit{decorated opposition diagram} of an uncapped automorphism. 

The full definition is given in Section~\ref{sec:1}, however for the purpose of this introduction consider the following simplified situation. Suppose that $\theta$ is an automorphism with the property that the induced automorphism $\pi$ of the Coxeter graph $\Gamma$ is the opposition automorphism~$w_0$. Then the \textit{decorated opposition diagram} of $\theta$ is the quadruple $(\Gamma,J,K,\pi)$ where $(\Gamma,J,\pi)$ is the opposition diagram, and 
\begin{center}
$K=\{j\in J\mid \text{there exists a type $J\backslash\{j\}$ simplex mapped onto an opposite simplex by $\theta$}\}$.
\end{center}
Less formally, the decorated opposition diagram is drawn by encircling the nodes of $J$, and then shading those nodes of $K$. Thus, for example, the decorated opposition diagrams of the two uncapped automorphisms of $\sB_3(2)$ given above are \begin{tikzpicture}[scale=0.5,baseline=-0.5ex]
\node at (0,0.3) {};
\node [inner sep=0.8pt,outer sep=0.8pt] at (-1.5,0) (1) {$\bullet$};
\node [inner sep=0.8pt,outer sep=0.8pt] at (-0.5,0) (2) {$\bullet$};
\node [inner sep=0.8pt,outer sep=0.8pt] at (0.5,0) (3) {$\bullet$};
\draw [line width=0.5pt,line cap=round,rounded corners,fill=ggrey] (1.north west)  rectangle (1.south east);
\draw [line width=0.5pt,line cap=round,rounded corners,fill=ggrey] (2.north west)  rectangle (2.south east);
\draw [line width=0.5pt,line cap=round,rounded corners,fill=ggrey] (3.north west)  rectangle (3.south east);
\draw (-1.5,0)--(-0.5,0);
\draw (-0.5,0.07)--(0.5,0.07);
\draw (-0.5,-0.07)--(0.5,-0.07);
\node [inner sep=0.8pt,outer sep=0.8pt] at (-1.5,0) (1) {$\bullet$};
\node [inner sep=0.8pt,outer sep=0.8pt] at (-0.5,0) (2) {$\bullet$};
\node [inner sep=0.8pt,outer sep=0.8pt] at (0.5,0) (3) {$\bullet$};
\end{tikzpicture} and \begin{tikzpicture}[scale=0.5,baseline=-0.5ex]
\node at (0,0.3) {};
\node [inner sep=0.8pt,outer sep=0.8pt] at (-1.5,0) (1) {$\bullet$};
\node [inner sep=0.8pt,outer sep=0.8pt] at (-0.5,0) (2) {$\bullet$};
\node [inner sep=0.8pt,outer sep=0.8pt] at (0.5,0) (3) {$\bullet$};
\draw [line width=0.5pt,line cap=round,rounded corners,fill=ggrey] (1.north west)  rectangle (1.south east);
\draw [line width=0.5pt,line cap=round,rounded corners,fill=ggrey] (2.north west)  rectangle (2.south east);
\draw [line width=0.5pt,line cap=round,rounded corners] (3.north west)  rectangle (3.south east);
\draw (-1.5,0)--(-0.5,0);
\draw (-0.5,0.07)--(0.5,0.07);
\draw (-0.5,-0.07)--(0.5,-0.07);
\node [inner sep=0.8pt,outer sep=0.8pt] at (-1.5,0) (1) {$\bullet$};
\node [inner sep=0.8pt,outer sep=0.8pt] at (-0.5,0) (2) {$\bullet$};
\node [inner sep=0.8pt,outer sep=0.8pt] at (0.5,0) (3) {$\bullet$};
\end{tikzpicture}. At an intuitive level, the more encircled nodes that are shaded on the decorated opposition diagram of an uncapped automorphism, the ``closer'' the automorphism is to being capped.

The main theorem of this paper is Theorem~\ref{thm:main*} below. Part~(a) of the theorem shows that the decorated opposition diagram of an uncapped automorphism lies in a small list of diagrams, hence severely restricting the structure of uncapped automorphisms. Part~(b) deals with the existence of uncapped automorphisms, showing that the list provided in part~(a) has no redundancies, with only the $\sE_8(2)$ case remaining open due to the size of the building rendering our computational techniques inadequate. We strongly believe that the two $\sE_8(2)$ diagrams are indeed realised as opposition diagrams; see Conjecture~\ref{conj:2} for details. 

\begin{thm1}\label{thm:main*}\leavevmode
\begin{compactenum}[$(a)$]
\item Let $\theta$ be an uncapped automorphism of a thick irreducible spherical building $\Delta$ of rank at least~$3$. Then the decorated opposition diagram of $\theta$ appears in Table~\ref{table:1} or Table~\ref{table:2}.
\item Let $\Delta$ be a small building. Each diagram appearing in the respective row of Table~\ref{table:1} or Table~\ref{table:2} can be realised as the decorated opposition diagram of some uncapped automorphism of $\Delta$, with the exception perhaps of the two $\sE_8(2)$ diagrams. 
\end{compactenum}
\end{thm1}

\begin{center}
\noindent\begin{tabular}{|c|l|}
\hline
\begin{tabular}{c}
$\Delta$
\end{tabular}&\emph{Diagrams}\\
\hline\hline
\begin{tabular}{l}
$\sA_n(2)$
\end{tabular}&
\begin{tabular}{l}
\begin{tikzpicture}[scale=0.5,baseline=-0.5ex]
\node at (0,0.5) {};
\node at (0,-0.5) {};
\node [inner sep=0.8pt,outer sep=0.8pt] at (-5,0) (-5) {$\bullet$};
\node [inner sep=0.8pt,outer sep=0.8pt] at (-4,0) (-4) {$\bullet$};
\node [inner sep=0.8pt,outer sep=0.8pt] at (-3,0) (-3) {$\bullet$};
\node [inner sep=0.8pt,outer sep=0.8pt] at (-2,0) (-2) {$\bullet$};
\node [inner sep=0.8pt,outer sep=0.8pt] at (-1,0) (-1) {$\bullet$};
\node [inner sep=0.8pt,outer sep=0.8pt] at (1,0) (1) {$\bullet$};
\node [inner sep=0.8pt,outer sep=0.8pt] at (2,0) (2) {$\bullet$};
\node [inner sep=0.8pt,outer sep=0.8pt] at (3,0) (3) {$\bullet$};
\node [inner sep=0.8pt,outer sep=0.8pt] at (4,0) (4) {$\bullet$};
\node [inner sep=0.8pt,outer sep=0.8pt] at (5,0) (5) {$\bullet$};
\draw [style=dashed] (-1,0)--(1,0);
\draw [line width=0.5pt,line cap=round,rounded corners,fill=ggrey] (-4.north west)  rectangle (-4.south east);
\draw [line width=0.5pt,line cap=round,rounded corners,fill=ggrey] (-2.north west)  rectangle (-2.south east);
\draw [line width=0.5pt,line cap=round,rounded corners,fill=ggrey] (2.north west)  rectangle (2.south east);
\draw [line width=0.5pt,line cap=round,rounded corners,fill=ggrey] (4.north west)  rectangle (4.south east);
\draw [line width=0.5pt,line cap=round,rounded corners,fill=ggrey] (-3.north west)  rectangle (-3.south east);
\draw [line width=0.5pt,line cap=round,rounded corners,fill=ggrey] (-1.north west)  rectangle (-1.south east);
\draw [line width=0.5pt,line cap=round,rounded corners,fill=ggrey] (1.north west)  rectangle (1.south east);
\draw [line width=0.5pt,line cap=round,rounded corners,fill=ggrey] (3.north west)  rectangle (3.south east);
\draw [line width=0.5pt,line cap=round,rounded corners,fill=ggrey] (-5.north west)  rectangle (-5.south east);
\draw [line width=0.5pt,line cap=round,rounded corners,fill=ggrey] (5.north west)  rectangle (5.south east);
\draw (-5,0)--(-1,0);
\draw (1,0)--(5,0);
\node [inner sep=0.8pt,outer sep=0.8pt] at (-5,0) (-5) {$\bullet$};
\node [inner sep=0.8pt,outer sep=0.8pt] at (-4,0) (-4) {$\bullet$};
\node [inner sep=0.8pt,outer sep=0.8pt] at (-3,0) (-3) {$\bullet$};
\node [inner sep=0.8pt,outer sep=0.8pt] at (-2,0) (-2) {$\bullet$};
\node [inner sep=0.8pt,outer sep=0.8pt] at (-1,0) (-1) {$\bullet$};
\node [inner sep=0.8pt,outer sep=0.8pt] at (1,0) (1) {$\bullet$};
\node [inner sep=0.8pt,outer sep=0.8pt] at (2,0) (2) {$\bullet$};
\node [inner sep=0.8pt,outer sep=0.8pt] at (3,0) (3) {$\bullet$};
\node [inner sep=0.8pt,outer sep=0.8pt] at (4,0) (4) {$\bullet$};
\node [inner sep=0.8pt,outer sep=0.8pt] at (5,0) (5) {$\bullet$};
\end{tikzpicture}
\end{tabular}\\
\hline
\begin{tabular}{c}
$\sB_n(2)$ \emph{or} $\sB_n(2,4)$,\\
$(3\leq j\leq n)$
\end{tabular}
&
\begin{tabular}{l}
\begin{tikzpicture}[scale=0.5,baseline=-0.5ex]
\node at (0,0.5) {};
\node [inner sep=0.8pt,outer sep=0.8pt] at (-5,0) (-5) {$\bullet$};
\node [inner sep=0.8pt,outer sep=0.8pt] at (-4,0) (-4) {$\bullet$};
\node [inner sep=0.8pt,outer sep=0.8pt] at (-3,0) (-3) {$\bullet$};
\node [inner sep=0.8pt,outer sep=0.8pt] at (-2,0) (-2) {$\bullet$};
\node [inner sep=0.8pt,outer sep=0.8pt] at (0,0) (-1) {$\bullet$};
\node [inner sep=0.8pt,outer sep=0.8pt] at (1,0) (1) {$\bullet$};
\node [inner sep=0.8pt,outer sep=0.8pt] at (2,0) (2) {$\bullet$};
\node [inner sep=0.8pt,outer sep=0.8pt] at (3,0) (3) {$\bullet$};
\node [inner sep=0.8pt,outer sep=0.8pt] at (4,0) (4) {$\bullet$};
\node [inner sep=0.8pt,outer sep=0.8pt] at (5,0) (5) {$\bullet$};
\node at (2,-0.7) {$j$};
\draw [line width=0.5pt,line cap=round,rounded corners,fill=ggrey] (-5.north west)  rectangle (-5.south east);
\draw [line width=0.5pt,line cap=round,rounded corners,fill=ggrey] (-4.north west)  rectangle (-4.south east);
\draw [line width=0.5pt,line cap=round,rounded corners,fill=ggrey] (-3.north west)  rectangle (-3.south east);
\draw [line width=0.5pt,line cap=round,rounded corners,fill=ggrey] (-2.north west)  rectangle (-2.south east);
\draw [line width=0.5pt,line cap=round,rounded corners,fill=ggrey] (-1.north west)  rectangle (-1.south east);
\draw [line width=0.5pt,line cap=round,rounded corners,fill=ggrey] (1.north west)  rectangle (1.south east);
\draw [line width=0.5pt,line cap=round,rounded corners] (2.north west)  rectangle (2.south east);
\draw (-5,0)--(-2,0);
\draw (0,0)--(4,0);
\draw (4,0.07)--(5,0.07);
\draw (4,-0.07)--(5,-0.07);
\draw [style=dashed] (-2,0)--(0,0);
\node [inner sep=0.8pt,outer sep=0.8pt] at (-5,0) (-5) {$\bullet$};
\node [inner sep=0.8pt,outer sep=0.8pt] at (-4,0) (-4) {$\bullet$};
\node [inner sep=0.8pt,outer sep=0.8pt] at (-3,0) (-3) {$\bullet$};
\node [inner sep=0.8pt,outer sep=0.8pt] at (-2,0) (-2) {$\bullet$};
\node [inner sep=0.8pt,outer sep=0.8pt] at (0,0) (-1) {$\bullet$};
\node [inner sep=0.8pt,outer sep=0.8pt] at (1,0) (1) {$\bullet$};
\node [inner sep=0.8pt,outer sep=0.8pt] at (2,0) (2) {$\bullet$};
\node [inner sep=0.8pt,outer sep=0.8pt] at (3,0) (3) {$\bullet$};
\node [inner sep=0.8pt,outer sep=0.8pt] at (4,0) (4) {$\bullet$};
\node [inner sep=0.8pt,outer sep=0.8pt] at (5,0) (5) {$\bullet$};
\end{tikzpicture}\quad%\quad  $3\leq j\leq n$\\
\begin{tikzpicture}[scale=0.5,baseline=-0.5ex]
\node at (0,0.3) {};
\node at (0,-0.5) {};
\node [inner sep=0.8pt,outer sep=0.8pt] at (-5,0) (-5) {$\bullet$};
\node [inner sep=0.8pt,outer sep=0.8pt] at (-4,0) (-4) {$\bullet$};
\node [inner sep=0.8pt,outer sep=0.8pt] at (-3,0) (-3) {$\bullet$};
\node [inner sep=0.8pt,outer sep=0.8pt] at (-2,0) (-2) {$\bullet$};
\node [inner sep=0.8pt,outer sep=0.8pt] at (0,0) (-1) {$\bullet$};
\node [inner sep=0.8pt,outer sep=0.8pt] at (1,0) (1) {$\bullet$};
\node [inner sep=0.8pt,outer sep=0.8pt] at (2,0) (2) {$\bullet$};
\node [inner sep=0.8pt,outer sep=0.8pt] at (3,0) (3) {$\bullet$};
\node [inner sep=0.8pt,outer sep=0.8pt] at (4,0) (4) {$\bullet$};
\node [inner sep=0.8pt,outer sep=0.8pt] at (5,0) (5) {$\bullet$};
%\node at (2,-0.7) {$j$};
\draw [line width=0.5pt,line cap=round,rounded corners,fill=ggrey] (-5.north west)  rectangle (-5.south east);
\draw [line width=0.5pt,line cap=round,rounded corners,fill=ggrey] (-4.north west)  rectangle (-4.south east);
\draw [line width=0.5pt,line cap=round,rounded corners,fill=ggrey] (-3.north west)  rectangle (-3.south east);
\draw [line width=0.5pt,line cap=round,rounded corners,fill=ggrey] (-2.north west)  rectangle (-2.south east);
\draw [line width=0.5pt,line cap=round,rounded corners,fill=ggrey] (-1.north west)  rectangle (-1.south east);
\draw [line width=0.5pt,line cap=round,rounded corners,fill=ggrey] (1.north west)  rectangle (1.south east);
\draw [line width=0.5pt,line cap=round,rounded corners,fill=ggrey] (2.north west)  rectangle (2.south east);
\draw [line width=0.5pt,line cap=round,rounded corners,fill=ggrey] (3.north west)  rectangle (3.south east);
\draw [line width=0.5pt,line cap=round,rounded corners,fill=ggrey] (4.north west)  rectangle (4.south east);
\draw [line width=0.5pt,line cap=round,rounded corners,fill=ggrey] (5.north west)  rectangle (5.south east);
\draw (-5,0)--(-2,0);
\draw (0,0)--(4,0);
\draw (4,0.07)--(5,0.07);
\draw (4,-0.07)--(5,-0.07);
\draw [style=dashed] (-2,0)--(0,0);
\node [inner sep=0.8pt,outer sep=0.8pt] at (-5,0) (-5) {$\bullet$};
\node [inner sep=0.8pt,outer sep=0.8pt] at (-4,0) (-4) {$\bullet$};
\node [inner sep=0.8pt,outer sep=0.8pt] at (-3,0) (-3) {$\bullet$};
\node [inner sep=0.8pt,outer sep=0.8pt] at (-2,0) (-2) {$\bullet$};
\node [inner sep=0.8pt,outer sep=0.8pt] at (0,0) (-1) {$\bullet$};
\node [inner sep=0.8pt,outer sep=0.8pt] at (1,0) (1) {$\bullet$};
\node [inner sep=0.8pt,outer sep=0.8pt] at (2,0) (2) {$\bullet$};
\node [inner sep=0.8pt,outer sep=0.8pt] at (3,0) (3) {$\bullet$};
\node [inner sep=0.8pt,outer sep=0.8pt] at (4,0) (4) {$\bullet$};
\node [inner sep=0.8pt,outer sep=0.8pt] at (5,0) (5) {$\bullet$};
%\node at (2,0) {$\bullet$};
\end{tikzpicture}
\end{tabular}
\\[4ex]
\hline
\begin{tabular}{c}
$\sD_{n}(2)$, $n\geq 4$ \emph{even}\\
$(4\leq 2j\leq n-2)$%\\
%\emph{collineation}
\end{tabular}&
\begin{tabular}{l}
\begin{tikzpicture}[scale=0.5,baseline=-0.5ex]
\node at (0,0.8) {};
\node [inner sep=0.8pt,outer sep=0.8pt] at (-5,0) (-5) {$\bullet$};
\node [inner sep=0.8pt,outer sep=0.8pt] at (-4,0) (-4) {$\bullet$};
\node [inner sep=0.8pt,outer sep=0.8pt] at (-3,0) (-3) {$\bullet$};
\node [inner sep=0.8pt,outer sep=0.8pt] at (-2,0) (-2) {$\bullet$};
\node [inner sep=0.8pt,outer sep=0.8pt] at (0,0) (-1) {$\bullet$};
%\node at (0,0) (0) {$\bullet$};
\node [inner sep=0.8pt,outer sep=0.8pt] at (1,0) (1) {$\bullet$};
\node [inner sep=0.8pt,outer sep=0.8pt] at (2,0) (2) {$\bullet$};
\node [inner sep=0.8pt,outer sep=0.8pt] at (3,0) (3) {$\bullet$};
\node [inner sep=0.8pt,outer sep=0.8pt] at (4,0) (4) {$\bullet$};
\node [inner sep=0.8pt,outer sep=0.8pt] at (5,0.5) (5a) {$\bullet$};
\node [inner sep=0.8pt,outer sep=0.8pt] at (5,-0.5) (5b) {$\bullet$};
%\draw [line width=0.5pt,line cap=round,rounded corners] (-5.north west)  rectangle (2.south east);
\draw [line width=0.5pt,line cap=round,rounded corners,fill=ggrey] (-5.north west)  rectangle (-5.south east);
\draw [line width=0.5pt,line cap=round,rounded corners,fill=ggrey] (-4.north west)  rectangle (-4.south east);
\draw [line width=0.5pt,line cap=round,rounded corners,fill=ggrey] (-3.north west)  rectangle (-3.south east);
\draw [line width=0.5pt,line cap=round,rounded corners,fill=ggrey] (-2.north west)  rectangle (-2.south east);
\draw [line width=0.5pt,line cap=round,rounded corners,fill=ggrey] (-1.north west)  rectangle (-1.south east);
\draw [line width=0.5pt,line cap=round,rounded corners,fill=ggrey] (1.north west)  rectangle (1.south east);
\draw [line width=0.5pt,line cap=round,rounded corners] (2.north west)  rectangle (2.south east);
\node [below] at (2,-0.25) {$2j$};
\draw (-5,0)--(-2,0);
\draw (0,0)--(4,0);
\draw (4,0) to (5,0.5);
\draw (4,0) to   (5,-0.5);
\draw [style=dashed] (-2,0)--(0,0);
\node at (2,0) {$\bullet$};
\node [inner sep=0.8pt,outer sep=0.8pt] at (-5,0) (-5) {$\bullet$};
\node [inner sep=0.8pt,outer sep=0.8pt] at (-4,0) (-4) {$\bullet$};
\node [inner sep=0.8pt,outer sep=0.8pt] at (-3,0) (-3) {$\bullet$};
\node [inner sep=0.8pt,outer sep=0.8pt] at (-2,0) (-2) {$\bullet$};
\node [inner sep=0.8pt,outer sep=0.8pt] at (0,0) (-1) {$\bullet$};
%\node at (0,0) (0) {$\bullet$};
\node [inner sep=0.8pt,outer sep=0.8pt] at (1,0) (1) {$\bullet$};
\node [inner sep=0.8pt,outer sep=0.8pt] at (2,0) (2) {$\bullet$};
\node [inner sep=0.8pt,outer sep=0.8pt] at (3,0) (3) {$\bullet$};
\node [inner sep=0.8pt,outer sep=0.8pt] at (4,0) (4) {$\bullet$};
\node [inner sep=0.8pt,outer sep=0.8pt] at (5,0.5) (5a) {$\bullet$};
\node [inner sep=0.8pt,outer sep=0.8pt] at (5,-0.5) (5b) {$\bullet$};
\end{tikzpicture}\quad%\quad $4\leq 2j\leq n-2$\\
\begin{tikzpicture}[scale=0.5,baseline=-0.5ex]
\node at (0,0.8) {};
\node at (0,-0.8) {};
\node [inner sep=0.8pt,outer sep=0.8pt] at (-5,0) (-5) {$\bullet$};
\node [inner sep=0.8pt,outer sep=0.8pt] at (-4,0) (-4) {$\bullet$};
\node [inner sep=0.8pt,outer sep=0.8pt] at (-3,0) (-3) {$\bullet$};
\node [inner sep=0.8pt,outer sep=0.8pt] at (-2,0) (-2) {$\bullet$};
\node [inner sep=0.8pt,outer sep=0.8pt] at (0,0) (-1) {$\bullet$};
%\node at (0,0) (0) {$\bullet$};
\node [inner sep=0.8pt,outer sep=0.8pt] at (1,0) (1) {$\bullet$};
\node [inner sep=0.8pt,outer sep=0.8pt] at (2,0) (2) {$\bullet$};
\node [inner sep=0.8pt,outer sep=0.8pt] at (3,0) (3) {$\bullet$};
\node [inner sep=0.8pt,outer sep=0.8pt] at (4,0) (4) {$\bullet$};
\node [inner sep=0.8pt,outer sep=0.8pt] at (5,0.5) (5a) {$\bullet$};
\node [inner sep=0.8pt,outer sep=0.8pt] at (5,-0.5) (5b) {$\bullet$};
\draw [line width=0.5pt,line cap=round,rounded corners,fill=ggrey] (-4.north west)  rectangle (-4.south east);
\draw [line width=0.5pt,line cap=round,rounded corners,fill=ggrey] (-2.north west)  rectangle (-2.south east);
\draw [line width=0.5pt,line cap=round,rounded corners,fill=ggrey] (2.north west)  rectangle (2.south east);
\draw [line width=0.5pt,line cap=round,rounded corners,fill=ggrey] (4.north west)  rectangle (4.south east);
\draw [line width=0.5pt,line cap=round,rounded corners,fill=ggrey] (5a.north west)  rectangle (5a.south east);
\draw [line width=0.5pt,line cap=round,rounded corners,fill=ggrey] (-5.north west)  rectangle (-5.south east);
\draw [line width=0.5pt,line cap=round,rounded corners,fill=ggrey] (-3.north west)  rectangle (-3.south east);
\draw [line width=0.5pt,line cap=round,rounded corners,fill=ggrey] (-1.north west)  rectangle (-1.south east);
\draw [line width=0.5pt,line cap=round,rounded corners,fill=ggrey] (1.north west)  rectangle (1.south east);
\draw [line width=0.5pt,line cap=round,rounded corners,fill=ggrey] (3.north west)  rectangle (3.south east);
\draw [line width=0.5pt,line cap=round,rounded corners,fill=ggrey] (5b.north west)  rectangle (5b.south east);
%\node [below] at (2,-0.25) {$2i$};
\node [inner sep=0.8pt,outer sep=0.8pt] at (-5,0) (-5) {$\bullet$};
\node [inner sep=0.8pt,outer sep=0.8pt] at (-4,0) (-4) {$\bullet$};
\node [inner sep=0.8pt,outer sep=0.8pt] at (-3,0) (-3) {$\bullet$};
\node [inner sep=0.8pt,outer sep=0.8pt] at (-2,0) (-2) {$\bullet$};
\node [inner sep=0.8pt,outer sep=0.8pt] at (0,0) (-1) {$\bullet$};
%\node at (0,0) (0) {$\bullet$};
\node [inner sep=0.8pt,outer sep=0.8pt] at (1,0) (1) {$\bullet$};
\node [inner sep=0.8pt,outer sep=0.8pt] at (2,0) (2) {$\bullet$};
\node [inner sep=0.8pt,outer sep=0.8pt] at (3,0) (3) {$\bullet$};
\node [inner sep=0.8pt,outer sep=0.8pt] at (4,0) (4) {$\bullet$};
\node [inner sep=0.8pt,outer sep=0.8pt] at (5,0.5) (5a) {$\bullet$};
\node [inner sep=0.8pt,outer sep=0.8pt] at (5,-0.5) (5b) {$\bullet$};
\draw (-5,0)--(-2,0);
\draw (0,0)--(4,0);
\draw (4,0) to  (5,0.5);
\draw (4,0) to  (5,-0.5);
\draw [style=dashed] (-2,0)--(0,0);
\end{tikzpicture}\\
\end{tabular}\\ \hline
\begin{tabular}{c}
$\sD_{n}(2)$, $n\geq 4$ \emph{odd}\\
$(4\leq 2j\leq n-3)$%,\\
%\emph{collineation}
\end{tabular}&
\begin{tabular}{l}
\begin{tikzpicture}[scale=0.5,baseline=-0.5ex]
\node at (0,0.8) {};
\node [inner sep=0.8pt,outer sep=0.8pt] at (-5,0) (-5) {$\bullet$};
\node [inner sep=0.8pt,outer sep=0.8pt] at (-4,0) (-4) {$\bullet$};
\node [inner sep=0.8pt,outer sep=0.8pt] at (-3,0) (-3) {$\bullet$};
\node [inner sep=0.8pt,outer sep=0.8pt] at (-2,0) (-2) {$\bullet$};
\node [inner sep=0.8pt,outer sep=0.8pt] at (0,0) (-1) {$\bullet$};
%\node at (0,0) (0) {$\bullet$};
\node [inner sep=0.8pt,outer sep=0.8pt] at (1,0) (1) {$\bullet$};
\node [inner sep=0.8pt,outer sep=0.8pt] at (2,0) (2) {$\bullet$};
\node [inner sep=0.8pt,outer sep=0.8pt] at (3,0) (3) {$\bullet$};
\node [inner sep=0.8pt,outer sep=0.8pt] at (4,0) (4) {$\bullet$};
\node [inner sep=0.8pt,outer sep=0.8pt] at (5,0.5) (5a) {$\bullet$};
\node [inner sep=0.8pt,outer sep=0.8pt] at (5,-0.5) (5b) {$\bullet$};
%\draw [line width=0.5pt,line cap=round,rounded corners] (-5.north west)  rectangle (2.south east);
\draw [line width=0.5pt,line cap=round,rounded corners,fill=ggrey] (-5.north west)  rectangle (-5.south east);
\draw [line width=0.5pt,line cap=round,rounded corners,fill=ggrey] (-4.north west)  rectangle (-4.south east);
\draw [line width=0.5pt,line cap=round,rounded corners,fill=ggrey] (-3.north west)  rectangle (-3.south east);
\draw [line width=0.5pt,line cap=round,rounded corners,fill=ggrey] (-2.north west)  rectangle (-2.south east);
\draw [line width=0.5pt,line cap=round,rounded corners,fill=ggrey] (-1.north west)  rectangle (-1.south east);
\draw [line width=0.5pt,line cap=round,rounded corners] (1.north west)  rectangle (1.south east);
%\draw [line width=0.5pt,line cap=round,rounded corners] (2.north west)  rectangle (2.south east);
\node [below] at (1,-0.25) {$2j$};
\draw (-5,0)--(-2,0);
\draw (0,0)--(4,0);
\draw (4,0) to [bend left] (5,0.5);
\draw (4,0) to [bend right=45] (5,-0.5);
\draw [style=dashed] (-2,0)--(0,0);
\node at (1,0) {$\bullet$};
\node [inner sep=0.8pt,outer sep=0.8pt] at (-5,0) (-5) {$\bullet$};
\node [inner sep=0.8pt,outer sep=0.8pt] at (-4,0) (-4) {$\bullet$};
\node [inner sep=0.8pt,outer sep=0.8pt] at (-3,0) (-3) {$\bullet$};
\node [inner sep=0.8pt,outer sep=0.8pt] at (-2,0) (-2) {$\bullet$};
\node [inner sep=0.8pt,outer sep=0.8pt] at (0,0) (-1) {$\bullet$};
%\node at (0,0) (0) {$\bullet$};
\node [inner sep=0.8pt,outer sep=0.8pt] at (1,0) (1) {$\bullet$};
\node [inner sep=0.8pt,outer sep=0.8pt] at (2,0) (2) {$\bullet$};
\node [inner sep=0.8pt,outer sep=0.8pt] at (3,0) (3) {$\bullet$};
\node [inner sep=0.8pt,outer sep=0.8pt] at (4,0) (4) {$\bullet$};
\node [inner sep=0.8pt,outer sep=0.8pt] at (5,0.5) (5a) {$\bullet$};
\node [inner sep=0.8pt,outer sep=0.8pt] at (5,-0.5) (5b) {$\bullet$};
\end{tikzpicture}\quad% $4\leq 2j\leq n-3$\\
\begin{tikzpicture}[scale=0.5,baseline=-0.5ex]
\node at (0,0.8) {};
\node at (0,-0.8) {};
\node [inner sep=0.8pt,outer sep=0.8pt] at (-5,0) (-5) {$\bullet$};
\node [inner sep=0.8pt,outer sep=0.8pt] at (-4,0) (-4) {$\bullet$};
\node [inner sep=0.8pt,outer sep=0.8pt] at (-3,0) (-3) {$\bullet$};
\node [inner sep=0.8pt,outer sep=0.8pt] at (-2,0) (-2) {$\bullet$};
\node [inner sep=0.8pt,outer sep=0.8pt] at (0,0) (-1) {$\bullet$};
%\node at (0,0) (0) {$\bullet$};
\node [inner sep=0.8pt,outer sep=0.8pt] at (1,0) (1) {$\bullet$};
\node [inner sep=0.8pt,outer sep=0.8pt] at (2,0) (2) {$\bullet$};
\node [inner sep=0.8pt,outer sep=0.8pt] at (3,0) (3) {$\bullet$};
\node [inner sep=0.8pt,outer sep=0.8pt] at (4,0) (4) {$\bullet$};
\node [inner sep=0.8pt,outer sep=0.8pt] at (5,0.5) (5a) {$\bullet$};
\node [inner sep=0.8pt,outer sep=0.8pt] at (5,-0.5) (5b) {$\bullet$};
\draw [line width=0.5pt,line cap=round,rounded corners,fill=ggrey] (-4.north west)  rectangle (-4.south east);
\draw [line width=0.5pt,line cap=round,rounded corners,fill=ggrey] (-2.north west)  rectangle (-2.south east);
\draw [line width=0.5pt,line cap=round,rounded corners,fill=ggrey] (2.north west)  rectangle (2.south east);
\draw [line width=0.5pt,line cap=round,rounded corners,fill=ggrey] (4.north west)  rectangle (4.south east);
\draw [line width=0.5pt,line cap=round,rounded corners,fill=ggrey] (-5.north west)  rectangle (-5.south east);
\draw [line width=0.5pt,line cap=round,rounded corners,fill=ggrey] (-3.north west)  rectangle (-3.south east);
\draw [line width=0.5pt,line cap=round,rounded corners,fill=ggrey] (-1.north west)  rectangle (-1.south east);
\draw [line width=0.5pt,line cap=round,rounded corners,fill=ggrey] (1.north west)  rectangle (1.south east);
\draw [line width=0.5pt,line cap=round,rounded corners,fill=ggrey] (3.north west)  rectangle (3.south east);
\draw [line width=0.5pt,line cap=round,rounded corners] (5a.north west)  rectangle (5b.south east);
%\node [below] at (2,-0.25) {$2i$};
\draw (-5,0)--(-2,0);
\draw (0,0)--(4,0);
\draw (4,0) to [bend left] (5,0.5);
\draw (4,0) to [bend right=45] (5,-0.5);
\draw [style=dashed] (-2,0)--(0,0);
\node at (5,0.5) {$\bullet$};
\node at (5,-0.5) {$\bullet$};
\node [inner sep=0.8pt,outer sep=0.8pt] at (-5,0) (-5) {$\bullet$};
\node [inner sep=0.8pt,outer sep=0.8pt] at (-4,0) (-4) {$\bullet$};
\node [inner sep=0.8pt,outer sep=0.8pt] at (-3,0) (-3) {$\bullet$};
\node [inner sep=0.8pt,outer sep=0.8pt] at (-2,0) (-2) {$\bullet$};
\node [inner sep=0.8pt,outer sep=0.8pt] at (0,0) (-1) {$\bullet$};
%\node at (0,0) (0) {$\bullet$};
\node [inner sep=0.8pt,outer sep=0.8pt] at (1,0) (1) {$\bullet$};
\node [inner sep=0.8pt,outer sep=0.8pt] at (2,0) (2) {$\bullet$};
\node [inner sep=0.8pt,outer sep=0.8pt] at (3,0) (3) {$\bullet$};
\node [inner sep=0.8pt,outer sep=0.8pt] at (4,0) (4) {$\bullet$};
\node [inner sep=0.8pt,outer sep=0.8pt] at (5,0.5) (5a) {$\bullet$};
\node [inner sep=0.8pt,outer sep=0.8pt] at (5,-0.5) (5b) {$\bullet$};
\end{tikzpicture}
\end{tabular}\\
\hline
\begin{tabular}{c}
$\sD_{n}(2)$, $n\geq 4$ \emph{even}\\
$(3\leq 2j+1\leq n-3)$%\\ 
%\emph{duality}
\end{tabular}&
\begin{tabular}{l}
\begin{tikzpicture}[scale=0.5,baseline=-0.5ex]
\node at (0,0.8) {};
\node [inner sep=0.8pt,outer sep=0.8pt] at (-5,0) (-5) {$\bullet$};
\node [inner sep=0.8pt,outer sep=0.8pt] at (-4,0) (-4) {$\bullet$};
\node [inner sep=0.8pt,outer sep=0.8pt] at (-3,0) (-3) {$\bullet$};
\node [inner sep=0.8pt,outer sep=0.8pt] at (-2,0) (-2) {$\bullet$};
\node [inner sep=0.8pt,outer sep=0.8pt] at (0,0) (-1) {$\bullet$};
%\node at (0,0) (0) {$\bullet$};
\node [inner sep=0.8pt,outer sep=0.8pt] at (1,0) (1) {$\bullet$};
\node [inner sep=0.8pt,outer sep=0.8pt] at (2,0) (2) {$\bullet$};
\node [inner sep=0.8pt,outer sep=0.8pt] at (3,0) (3) {$\bullet$};
\node [inner sep=0.8pt,outer sep=0.8pt] at (4,0) (4) {$\bullet$};
\node [inner sep=0.8pt,outer sep=0.8pt] at (5,0.5) (5a) {$\bullet$};
\node [inner sep=0.8pt,outer sep=0.8pt] at (5,-0.5) (5b) {$\bullet$};
%\draw [line width=0.5pt,line cap=round,rounded corners] (-5.north west)  rectangle (2.south east);
\draw [line width=0.5pt,line cap=round,rounded corners,fill=ggrey] (-5.north west)  rectangle (-5.south east);
\draw [line width=0.5pt,line cap=round,rounded corners,fill=ggrey] (-4.north west)  rectangle (-4.south east);
\draw [line width=0.5pt,line cap=round,rounded corners,fill=ggrey] (-3.north west)  rectangle (-3.south east);
\draw [line width=0.5pt,line cap=round,rounded corners,fill=ggrey] (-2.north west)  rectangle (-2.south east);
\draw [line width=0.5pt,line cap=round,rounded corners,fill=ggrey] (-1.north west)  rectangle (-1.south east);
\draw [line width=0.5pt,line cap=round,rounded corners] (1.north west)  rectangle (1.south east);
%\draw [line width=0.5pt,line cap=round,rounded corners] (2.north west)  rectangle (2.south east);
\node [below] at (1,-0.25) {$2j+1$};
\draw (-5,0)--(-2,0);
\draw (0,0)--(4,0);
\draw (4,0) to [bend left] (5,0.5);
\draw (4,0) to [bend right=45] (5,-0.5);
\draw [style=dashed] (-2,0)--(0,0);
\node at (1,0) {$\bullet$};
\node [inner sep=0.8pt,outer sep=0.8pt] at (-5,0) (-5) {$\bullet$};
\node [inner sep=0.8pt,outer sep=0.8pt] at (-4,0) (-4) {$\bullet$};
\node [inner sep=0.8pt,outer sep=0.8pt] at (-3,0) (-3) {$\bullet$};
\node [inner sep=0.8pt,outer sep=0.8pt] at (-2,0) (-2) {$\bullet$};
\node [inner sep=0.8pt,outer sep=0.8pt] at (0,0) (-1) {$\bullet$};
%\node at (0,0) (0) {$\bullet$};
\node [inner sep=0.8pt,outer sep=0.8pt] at (1,0) (1) {$\bullet$};
\node [inner sep=0.8pt,outer sep=0.8pt] at (2,0) (2) {$\bullet$};
\node [inner sep=0.8pt,outer sep=0.8pt] at (3,0) (3) {$\bullet$};
\node [inner sep=0.8pt,outer sep=0.8pt] at (4,0) (4) {$\bullet$};
\node [inner sep=0.8pt,outer sep=0.8pt] at (5,0.5) (5a) {$\bullet$};
\node [inner sep=0.8pt,outer sep=0.8pt] at (5,-0.5) (5b) {$\bullet$};
\end{tikzpicture}\quad% $3\leq 2j+1\leq n-3$\\
\begin{tikzpicture}[scale=0.5,baseline=-0.5ex]
\node [inner sep=0.8pt,outer sep=0.8pt] at (-5,0) (-5) {$\bullet$};
\node [inner sep=0.8pt,outer sep=0.8pt] at (-4,0) (-4) {$\bullet$};
\node [inner sep=0.8pt,outer sep=0.8pt] at (-3,0) (-3) {$\bullet$};
\node [inner sep=0.8pt,outer sep=0.8pt] at (-2,0) (-2) {$\bullet$};
\node [inner sep=0.8pt,outer sep=0.8pt] at (0,0) (-1) {$\bullet$};
%\node at (0,0) (0) {$\bullet$};
\node [inner sep=0.8pt,outer sep=0.8pt] at (1,0) (1) {$\bullet$};
\node [inner sep=0.8pt,outer sep=0.8pt] at (2,0) (2) {$\bullet$};
\node [inner sep=0.8pt,outer sep=0.8pt] at (3,0) (3) {$\bullet$};
\node [inner sep=0.8pt,outer sep=0.8pt] at (4,0) (4) {$\bullet$};
\node [inner sep=0.8pt,outer sep=0.8pt] at (5,0.5) (5a) {$\bullet$};
\node [inner sep=0.8pt,outer sep=0.8pt] at (5,-0.5) (5b) {$\bullet$};
%\draw [line width=0.5pt,line cap=round,rounded corners] (-5.north west)  rectangle (2.south east);
\draw [line width=0.5pt,line cap=round,rounded corners,fill=ggrey] (-5.north west)  rectangle (-5.south east);
\draw [line width=0.5pt,line cap=round,rounded corners,fill=ggrey] (-4.north west)  rectangle (-4.south east);
\draw [line width=0.5pt,line cap=round,rounded corners,fill=ggrey] (-3.north west)  rectangle (-3.south east);
\draw [line width=0.5pt,line cap=round,rounded corners,fill=ggrey] (-2.north west)  rectangle (-2.south east);
\draw [line width=0.5pt,line cap=round,rounded corners,fill=ggrey] (-1.north west)  rectangle (-1.south east);
\draw [line width=0.5pt,line cap=round,rounded corners,fill=ggrey] (1.north west)  rectangle (1.south east);
\draw [line width=0.5pt,line cap=round,rounded corners,fill=ggrey] (2.north west)  rectangle (2.south east);
\draw [line width=0.5pt,line cap=round,rounded corners,fill=ggrey] (3.north west)  rectangle (3.south east);
\draw [line width=0.5pt,line cap=round,rounded corners,fill=ggrey] (4.north west)  rectangle (4.south east);
%\draw [line width=0.5pt,line cap=round,rounded corners] (5a.north west)  rectangle (5a.south east);
\draw [line width=0.5pt,line cap=round,rounded corners] (5a.north west)  rectangle (5b.south east);
%\node [below] at (2,-0.25) {$2i+1$};
\node at (0,-0.8) {};
\draw (-5,0)--(-2,0);
\draw (0,0)--(4,0);
\draw (4,0) to [bend left] (5,0.5);
\draw (4,0) to [bend right=45] (5,-0.5);
\draw [style=dashed] (-2,0)--(0,0);
\node at (5,0.5) {$\bullet$};
\node at (5,-0.5) {$\bullet$};
\node [inner sep=0.8pt,outer sep=0.8pt] at (-5,0) (-5) {$\bullet$};
\node [inner sep=0.8pt,outer sep=0.8pt] at (-4,0) (-4) {$\bullet$};
\node [inner sep=0.8pt,outer sep=0.8pt] at (-3,0) (-3) {$\bullet$};
\node [inner sep=0.8pt,outer sep=0.8pt] at (-2,0) (-2) {$\bullet$};
\node [inner sep=0.8pt,outer sep=0.8pt] at (0,0) (-1) {$\bullet$};
%\node at (0,0) (0) {$\bullet$};
\node [inner sep=0.8pt,outer sep=0.8pt] at (1,0) (1) {$\bullet$};
\node [inner sep=0.8pt,outer sep=0.8pt] at (2,0) (2) {$\bullet$};
\node [inner sep=0.8pt,outer sep=0.8pt] at (3,0) (3) {$\bullet$};
\node [inner sep=0.8pt,outer sep=0.8pt] at (4,0) (4) {$\bullet$};
\node [inner sep=0.8pt,outer sep=0.8pt] at (5,0.5) (5a) {$\bullet$};
\node [inner sep=0.8pt,outer sep=0.8pt] at (5,-0.5) (5b) {$\bullet$};
\end{tikzpicture}
\end{tabular}\\ 
\hline
\begin{tabular}{c}
$\sD_{n}(2)$, $n\geq 4$ \emph{odd}\\
$(3\leq 2j+1\leq n-2)$%,\\ 
%\emph{duality}
\end{tabular}&
\begin{tabular}{l}
\begin{tikzpicture}[scale=0.5,baseline=-0.5ex]
\node at (0,0.8) {};
\node [inner sep=0.8pt,outer sep=0.8pt] at (-5,0) (-5) {$\bullet$};
\node [inner sep=0.8pt,outer sep=0.8pt] at (-4,0) (-4) {$\bullet$};
\node [inner sep=0.8pt,outer sep=0.8pt] at (-3,0) (-3) {$\bullet$};
\node [inner sep=0.8pt,outer sep=0.8pt] at (-2,0) (-2) {$\bullet$};
\node [inner sep=0.8pt,outer sep=0.8pt] at (0,0) (-1) {$\bullet$};
%\node at (0,0) (0) {$\bullet$};
\node [inner sep=0.8pt,outer sep=0.8pt] at (1,0) (1) {$\bullet$};
\node [inner sep=0.8pt,outer sep=0.8pt] at (2,0) (2) {$\bullet$};
\node [inner sep=0.8pt,outer sep=0.8pt] at (3,0) (3) {$\bullet$};
\node [inner sep=0.8pt,outer sep=0.8pt] at (4,0) (4) {$\bullet$};
\node [inner sep=0.8pt,outer sep=0.8pt] at (5,0.5) (5a) {$\bullet$};
\node [inner sep=0.8pt,outer sep=0.8pt] at (5,-0.5) (5b) {$\bullet$};
%\draw [line width=0.5pt,line cap=round,rounded corners] (-5.north west)  rectangle (2.south east);
\draw [line width=0.5pt,line cap=round,rounded corners,fill=ggrey] (-5.north west)  rectangle (-5.south east);
\draw [line width=0.5pt,line cap=round,rounded corners,fill=ggrey] (-4.north west)  rectangle (-4.south east);
\draw [line width=0.5pt,line cap=round,rounded corners,fill=ggrey] (-3.north west)  rectangle (-3.south east);
\draw [line width=0.5pt,line cap=round,rounded corners,fill=ggrey] (-2.north west)  rectangle (-2.south east);
\draw [line width=0.5pt,line cap=round,rounded corners,fill=ggrey] (-1.north west)  rectangle (-1.south east);
\draw [line width=0.5pt,line cap=round,rounded corners,fill=ggrey] (1.north west)  rectangle (1.south east);
\draw [line width=0.5pt,line cap=round,rounded corners] (2.north west)  rectangle (2.south east);
\node [below] at (2,-0.25) {$2j+1$};
\draw (-5,0)--(-2,0);
\draw (0,0)--(4,0);
\draw (4,0) to  (5,0.5);
\draw (4,0) to (5,-0.5);
\draw [style=dashed] (-2,0)--(0,0);
\node at (2,0) {$\bullet$};
\node [inner sep=0.8pt,outer sep=0.8pt] at (-5,0) (-5) {$\bullet$};
\node [inner sep=0.8pt,outer sep=0.8pt] at (-4,0) (-4) {$\bullet$};
\node [inner sep=0.8pt,outer sep=0.8pt] at (-3,0) (-3) {$\bullet$};
\node [inner sep=0.8pt,outer sep=0.8pt] at (-2,0) (-2) {$\bullet$};
\node [inner sep=0.8pt,outer sep=0.8pt] at (0,0) (-1) {$\bullet$};
%\node at (0,0) (0) {$\bullet$};
\node [inner sep=0.8pt,outer sep=0.8pt] at (1,0) (1) {$\bullet$};
\node [inner sep=0.8pt,outer sep=0.8pt] at (2,0) (2) {$\bullet$};
\node [inner sep=0.8pt,outer sep=0.8pt] at (3,0) (3) {$\bullet$};
\node [inner sep=0.8pt,outer sep=0.8pt] at (4,0) (4) {$\bullet$};
\node [inner sep=0.8pt,outer sep=0.8pt] at (5,0.5) (5a) {$\bullet$};
\node [inner sep=0.8pt,outer sep=0.8pt] at (5,-0.5) (5b) {$\bullet$};
\end{tikzpicture}\quad %$3\leq 2j+1\leq n-2$\\
\begin{tikzpicture}[scale=0.5,baseline=-0.5ex]
\node [inner sep=0.8pt,outer sep=0.8pt] at (-5,0) (-5) {$\bullet$};
\node [inner sep=0.8pt,outer sep=0.8pt] at (-4,0) (-4) {$\bullet$};
\node [inner sep=0.8pt,outer sep=0.8pt] at (-3,0) (-3) {$\bullet$};
\node [inner sep=0.8pt,outer sep=0.8pt] at (-2,0) (-2) {$\bullet$};
\node [inner sep=0.8pt,outer sep=0.8pt] at (0,0) (-1) {$\bullet$};
%\node at (0,0) (0) {$\bullet$};
\node [inner sep=0.8pt,outer sep=0.8pt] at (1,0) (1) {$\bullet$};
\node [inner sep=0.8pt,outer sep=0.8pt] at (2,0) (2) {$\bullet$};
\node [inner sep=0.8pt,outer sep=0.8pt] at (3,0) (3) {$\bullet$};
\node [inner sep=0.8pt,outer sep=0.8pt] at (4,0) (4) {$\bullet$};
\node [inner sep=0.8pt,outer sep=0.8pt] at (5,0.5) (5a) {$\bullet$};
\node [inner sep=0.8pt,outer sep=0.8pt] at (5,-0.5) (5b) {$\bullet$};
%\draw [line width=0.5pt,line cap=round,rounded corners] (-5.north west)  rectangle (2.south east);
\draw [line width=0.5pt,line cap=round,rounded corners,fill=ggrey] (-5.north west)  rectangle (-5.south east);
\draw [line width=0.5pt,line cap=round,rounded corners,fill=ggrey] (-4.north west)  rectangle (-4.south east);
\draw [line width=0.5pt,line cap=round,rounded corners,fill=ggrey] (-3.north west)  rectangle (-3.south east);
\draw [line width=0.5pt,line cap=round,rounded corners,fill=ggrey] (-2.north west)  rectangle (-2.south east);
\draw [line width=0.5pt,line cap=round,rounded corners,fill=ggrey] (-1.north west)  rectangle (-1.south east);
\draw [line width=0.5pt,line cap=round,rounded corners,fill=ggrey] (1.north west)  rectangle (1.south east);
\draw [line width=0.5pt,line cap=round,rounded corners,fill=ggrey] (2.north west)  rectangle (2.south east);
\draw [line width=0.5pt,line cap=round,rounded corners,fill=ggrey] (3.north west)  rectangle (3.south east);
\draw [line width=0.5pt,line cap=round,rounded corners,fill=ggrey] (4.north west)  rectangle (4.south east);
\draw [line width=0.5pt,line cap=round,rounded corners,fill=ggrey] (5a.north west)  rectangle (5a.south east);
\draw [line width=0.5pt,line cap=round,rounded corners,fill=ggrey] (5b.north west)  rectangle (5b.south east);
%\node [below] at (2,-0.25) {$2i+1$};
\node at (0,-0.8) {};
\node [inner sep=0.8pt,outer sep=0.8pt] at (-5,0) (-5) {$\bullet$};
\node [inner sep=0.8pt,outer sep=0.8pt] at (-4,0) (-4) {$\bullet$};
\node [inner sep=0.8pt,outer sep=0.8pt] at (-3,0) (-3) {$\bullet$};
\node [inner sep=0.8pt,outer sep=0.8pt] at (-2,0) (-2) {$\bullet$};
\node [inner sep=0.8pt,outer sep=0.8pt] at (0,0) (-1) {$\bullet$};
%\node at (0,0) (0) {$\bullet$};
\node [inner sep=0.8pt,outer sep=0.8pt] at (1,0) (1) {$\bullet$};
\node [inner sep=0.8pt,outer sep=0.8pt] at (2,0) (2) {$\bullet$};
\node [inner sep=0.8pt,outer sep=0.8pt] at (3,0) (3) {$\bullet$};
\node [inner sep=0.8pt,outer sep=0.8pt] at (4,0) (4) {$\bullet$};
\node [inner sep=0.8pt,outer sep=0.8pt] at (5,0.5) (5a) {$\bullet$};
\node [inner sep=0.8pt,outer sep=0.8pt] at (5,-0.5) (5b) {$\bullet$};
\draw (-5,0)--(-2,0);
\draw (0,0)--(4,0);
\draw (4,0) to  (5,0.5);
\draw (4,0) to (5,-0.5);
\draw [style=dashed] (-2,0)--(0,0);
\end{tikzpicture}
\end{tabular}\\
\hline
\end{tabular}%\vspace{-0.3cm}
\captionof{table}{Decorated opposition diagrams of uncapped automorphisms (classical types)}\label{table:1}
\end{center}

\begin{center}
\noindent\begin{tabular}{|c|l|}
\hline
\begin{tabular}{l}
$\Delta$
\end{tabular}&\emph{Diagrams}\\
\hline\hline
\begin{tabular}{l}
$\sE_6(2)$%, \emph{collineation}
\end{tabular}&
\begin{tabular}{l}
\begin{tikzpicture}[scale=0.5,baseline=-0.5ex]
\node at (0,0.3) {};
\node [inner sep=0.8pt,outer sep=0.8pt] at (-2,0) (2) {$\bullet$};
\node [inner sep=0.8pt,outer sep=0.8pt] at (-1,0) (4) {$\bullet$};
\node [inner sep=0.8pt,outer sep=0.8pt] at (0,-0.5) (5) {$\bullet$};
\node [inner sep=0.8pt,outer sep=0.8pt] at (0,0.5) (3) {$\bullet$};
\node [inner sep=0.8pt,outer sep=0.8pt] at (1,-0.5) (6) {$\bullet$};
\node [inner sep=0.8pt,outer sep=0.8pt] at (1,0.5) (1) {$\bullet$};
\draw [line width=0.5pt,line cap=round,rounded corners,fill=ggrey] (2.north west)  rectangle (2.south east);
\draw [line width=0.5pt,line cap=round,rounded corners,fill=ggrey] (4.north west)  rectangle (4.south east);
\draw [line width=0.5pt,line cap=round,rounded corners] (3.north west)  rectangle (5.south east);
\draw [line width=0.5pt,line cap=round,rounded corners] (1.north west)  rectangle (6.south east);
\draw (-2,0)--(-1,0);
\draw (-1,0) to [bend left=45] (0,0.5);
\draw (-1,0) to [bend right=45] (0,-0.5);
\draw (0,0.5)--(1,0.5);
\draw (0,-0.5)--(1,-0.5);
\node at (0,-0.5) {$\bullet$};
\node at (0,0.5) {$\bullet$};
\node at (1,-0.5) {$\bullet$};
\node at (1,0.5) {$\bullet$};
\node [inner sep=0.8pt,outer sep=0.8pt] at (-2,0) (2) {$\bullet$};
\node [inner sep=0.8pt,outer sep=0.8pt] at (-1,0) (4) {$\bullet$};
\node [inner sep=0.8pt,outer sep=0.8pt] at (0,-0.5) (5) {$\bullet$};
\node [inner sep=0.8pt,outer sep=0.8pt] at (0,0.5) (3) {$\bullet$};
\node [inner sep=0.8pt,outer sep=0.8pt] at (1,-0.5) (6) {$\bullet$};
\node [inner sep=0.8pt,outer sep=0.8pt] at (1,0.5) (1) {$\bullet$};
\end{tikzpicture}\qquad \begin{tikzpicture}[scale=0.5,baseline=-2ex]
\node at (0,0.3) {};
\node at (0,-1.3) {};
\node [inner sep=0.8pt,outer sep=0.8pt] at (-2,0) (1) {$\bullet$};
\node [inner sep=0.8pt,outer sep=0.8pt] at (-1,0) (3) {$\bullet$};
\node [inner sep=0.8pt,outer sep=0.8pt] at (0,0) (4) {$\bullet$};
\node [inner sep=0.8pt,outer sep=0.8pt] at (1,0) (5) {$\bullet$};
\node [inner sep=0.8pt,outer sep=0.8pt] at (2,0) (6) {$\bullet$};
\node [inner sep=0.8pt,outer sep=0.8pt] at (0,-1) (2) {$\bullet$};
\draw [line width=0.5pt,line cap=round,rounded corners,fill=ggrey] (1.north west)  rectangle (1.south east);
\draw [line width=0.5pt,line cap=round,rounded corners,fill=ggrey] (2.north west)  rectangle (2.south east);
\draw [line width=0.5pt,line cap=round,rounded corners,fill=ggrey] (3.north west)  rectangle (3.south east);
\draw [line width=0.5pt,line cap=round,rounded corners,fill=ggrey] (4.north west)  rectangle (4.south east);
\draw [line width=0.5pt,line cap=round,rounded corners,fill=ggrey] (5.north west)  rectangle (5.south east);
\draw [line width=0.5pt,line cap=round,rounded corners,fill=ggrey] (6.north west)  rectangle (6.south east);
\draw (-2,0)--(2,0);
\draw (0,0)--(0,-1);
\node [inner sep=0.8pt,outer sep=0.8pt] at (-2,0) (1) {$\bullet$};
\node [inner sep=0.8pt,outer sep=0.8pt] at (-1,0) (3) {$\bullet$};
\node [inner sep=0.8pt,outer sep=0.8pt] at (0,0) (4) {$\bullet$};
\node [inner sep=0.8pt,outer sep=0.8pt] at (1,0) (5) {$\bullet$};
\node [inner sep=0.8pt,outer sep=0.8pt] at (2,0) (6) {$\bullet$};
\node [inner sep=0.8pt,outer sep=0.8pt] at (0,-1) (2) {$\bullet$};
\end{tikzpicture} 
\end{tabular}\\
\hline
\begin{tabular}{l}
$\sE_7(2)$%, \emph{collineation}
\end{tabular}&
\begin{tabular}{l}
\begin{tikzpicture}[scale=0.5,baseline=-1.5ex]
\node [inner sep=0.8pt,outer sep=0.8pt] at (-2,0) (1) {$\bullet$};
\node [inner sep=0.8pt,outer sep=0.8pt] at (-1,0) (3) {$\bullet$};
\node [inner sep=0.8pt,outer sep=0.8pt] at (0,0) (4) {$\bullet$};
\node [inner sep=0.8pt,outer sep=0.8pt] at (1,0) (5) {$\bullet$};
\node [inner sep=0.8pt,outer sep=0.8pt] at (2,0) (6) {$\bullet$};
\node [inner sep=0.8pt,outer sep=0.8pt] at (3,0) (7) {$\bullet$};
\node [inner sep=0.8pt,outer sep=0.8pt] at (0,-1) (2) {$\bullet$};
\draw [line width=0.5pt,line cap=round,rounded corners,fill=ggrey] (1.north west)  rectangle (1.south east);
\draw [line width=0.5pt,line cap=round,rounded corners,fill=ggrey] (3.north west)  rectangle (3.south east);
\draw [line width=0.5pt,line cap=round,rounded corners] (4.north west)  rectangle (4.south east);
\draw [line width=0.5pt,line cap=round,rounded corners] (6.north west)  rectangle (6.south east);
\node at (0,0) {$\bullet$};
\node at (2,0) {$\bullet$};
\node at (0,-1.3) {};
\node at (0,0.3) {};
\node [inner sep=0.8pt,outer sep=0.8pt] at (-2,0) (1) {$\bullet$};
\node [inner sep=0.8pt,outer sep=0.8pt] at (-1,0) (3) {$\bullet$};
\node [inner sep=0.8pt,outer sep=0.8pt] at (0,0) (4) {$\bullet$};
\node [inner sep=0.8pt,outer sep=0.8pt] at (1,0) (5) {$\bullet$};
\node [inner sep=0.8pt,outer sep=0.8pt] at (2,0) (6) {$\bullet$};
\node [inner sep=0.8pt,outer sep=0.8pt] at (3,0) (7) {$\bullet$};
\node [inner sep=0.8pt,outer sep=0.8pt] at (0,-1) (2) {$\bullet$};
\draw (-2,0)--(3,0);
\draw (0,0)--(0,-1);
\end{tikzpicture}
\qquad 
\begin{tikzpicture}[scale=0.5,baseline=-1.5ex]
\node [inner sep=0.8pt,outer sep=0.8pt] at (-2,0) (1) {$\bullet$};
\node [inner sep=0.8pt,outer sep=0.8pt] at (-1,0) (3) {$\bullet$};
\node [inner sep=0.8pt,outer sep=0.8pt] at (0,0) (4) {$\bullet$};
\node [inner sep=0.8pt,outer sep=0.8pt] at (1,0) (5) {$\bullet$};
\node [inner sep=0.8pt,outer sep=0.8pt] at (2,0) (6) {$\bullet$};
\node [inner sep=0.8pt,outer sep=0.8pt] at (3,0) (7) {$\bullet$};
\node [inner sep=0.8pt,outer sep=0.8pt] at (0,-1) (2) {$\bullet$};
\draw [line width=0.5pt,line cap=round,rounded corners,fill=ggrey] (1.north west)  rectangle (1.south east);
\draw [line width=0.5pt,line cap=round,rounded corners,fill=ggrey] (3.north west)  rectangle (3.south east);
\draw [line width=0.5pt,line cap=round,rounded corners,fill=ggrey] (4.north west)  rectangle (4.south east);
\draw [line width=0.5pt,line cap=round,rounded corners,fill=ggrey] (6.north west)  rectangle (6.south east);
\draw [line width=0.5pt,line cap=round,rounded corners,fill=ggrey] (2.north west)  rectangle (2.south east);
\draw [line width=0.5pt,line cap=round,rounded corners,fill=ggrey] (5.north west)  rectangle (5.south east);
\draw [line width=0.5pt,line cap=round,rounded corners,fill=ggrey] (7.north west)  rectangle (7.south east);
\node [inner sep=0.8pt,outer sep=0.8pt] at (-2,0) (1) {$\bullet$};
\node [inner sep=0.8pt,outer sep=0.8pt] at (-1,0) (3) {$\bullet$};
\node [inner sep=0.8pt,outer sep=0.8pt] at (0,0) (4) {$\bullet$};
\node [inner sep=0.8pt,outer sep=0.8pt] at (1,0) (5) {$\bullet$};
\node [inner sep=0.8pt,outer sep=0.8pt] at (2,0) (6) {$\bullet$};
\node [inner sep=0.8pt,outer sep=0.8pt] at (3,0) (7) {$\bullet$};
\node [inner sep=0.8pt,outer sep=0.8pt] at (0,-1) (2) {$\bullet$};
\draw (-2,0)--(3,0);
\draw (0,0)--(0,-1);
\end{tikzpicture} 
\end{tabular}\\
\hline
\begin{tabular}{l}
$\sE_8(2)$%, \emph{collineation}
\end{tabular}&
\begin{tabular}{l}
\begin{tikzpicture}[scale=0.5,baseline=-1.5ex]
\node at (0,0.3) {};
\node at (0,-1.3) {};
\node [inner sep=0.8pt,outer sep=0.8pt] at (-2,0) (1) {$\bullet$};
\node [inner sep=0.8pt,outer sep=0.8pt] at (-1,0) (3) {$\bullet$};
\node [inner sep=0.8pt,outer sep=0.8pt] at (0,0) (4) {$\bullet$};
\node [inner sep=0.8pt,outer sep=0.8pt] at (1,0) (5) {$\bullet$};
\node [inner sep=0.8pt,outer sep=0.8pt] at (2,0) (6) {$\bullet$};
\node [inner sep=0.8pt,outer sep=0.8pt] at (3,0) (7) {$\bullet$};
\node [inner sep=0.8pt,outer sep=0.8pt] at (4,0) (8) {$\bullet$};
\node [inner sep=0.8pt,outer sep=0.8pt] at (0,-1) (2) {$\bullet$};
\draw [line width=0.5pt,line cap=round,rounded corners] (1.north west)  rectangle (1.south east);
\draw [line width=0.5pt,line cap=round,rounded corners] (6.north west)  rectangle (6.south east);
\draw [line width=0.5pt,line cap=round,rounded corners,fill=ggrey] (7.north west)  rectangle (7.south east);
\draw [line width=0.5pt,line cap=round,rounded corners,fill=ggrey] (8.north west)  rectangle (8.south east);
\draw (-2,0)--(4,0);
\draw (0,0)--(0,-1);
\node at (-2,0) {$\bullet$};
\node at (2,0) {$\bullet$};
\node [inner sep=0.8pt,outer sep=0.8pt] at (-2,0) (1) {$\bullet$};
\node [inner sep=0.8pt,outer sep=0.8pt] at (-1,0) (3) {$\bullet$};
\node [inner sep=0.8pt,outer sep=0.8pt] at (0,0) (4) {$\bullet$};
\node [inner sep=0.8pt,outer sep=0.8pt] at (1,0) (5) {$\bullet$};
\node [inner sep=0.8pt,outer sep=0.8pt] at (2,0) (6) {$\bullet$};
\node [inner sep=0.8pt,outer sep=0.8pt] at (3,0) (7) {$\bullet$};
\node [inner sep=0.8pt,outer sep=0.8pt] at (4,0) (8) {$\bullet$};
\node [inner sep=0.8pt,outer sep=0.8pt] at (0,-1) (2) {$\bullet$};
\end{tikzpicture}\qquad
\begin{tikzpicture}[scale=0.5,baseline=-1.5ex]
\node [inner sep=0.8pt,outer sep=0.8pt] at (-2,0) (1) {$\bullet$};
\node [inner sep=0.8pt,outer sep=0.8pt] at (-1,0) (3) {$\bullet$};
\node [inner sep=0.8pt,outer sep=0.8pt] at (0,0) (4) {$\bullet$};
\node [inner sep=0.8pt,outer sep=0.8pt] at (1,0) (5) {$\bullet$};
\node [inner sep=0.8pt,outer sep=0.8pt] at (2,0) (6) {$\bullet$};
\node [inner sep=0.8pt,outer sep=0.8pt] at (3,0) (7) {$\bullet$};
\node [inner sep=0.8pt,outer sep=0.8pt] at (4,0) (8) {$\bullet$};
\node [inner sep=0.8pt,outer sep=0.8pt] at (0,-1) (2) {$\bullet$};
\draw [line width=0.5pt,line cap=round,rounded corners,fill=ggrey] (8.north west)  rectangle (8.south east);
\draw [line width=0.5pt,line cap=round,rounded corners,fill=ggrey] (7.north west)  rectangle (7.south east);
\draw [line width=0.5pt,line cap=round,rounded corners,fill=ggrey] (6.north west)  rectangle (6.south east);
\draw [line width=0.5pt,line cap=round,rounded corners,fill=ggrey] (5.north west)  rectangle (5.south east);
\draw [line width=0.5pt,line cap=round,rounded corners,fill=ggrey] (4.north west)  rectangle (4.south east);
\draw [line width=0.5pt,line cap=round,rounded corners,fill=ggrey] (3.north west)  rectangle (3.south east);
\draw [line width=0.5pt,line cap=round,rounded corners,fill=ggrey] (2.north west)  rectangle (2.south east);
\draw [line width=0.5pt,line cap=round,rounded corners,fill=ggrey] (1.north west)  rectangle (1.south east);
\draw (-2,0)--(4,0);
\draw (0,0)--(0,-1);
\node [inner sep=0.8pt,outer sep=0.8pt] at (-2,0) (1) {$\bullet$};
\node [inner sep=0.8pt,outer sep=0.8pt] at (-1,0) (3) {$\bullet$};
\node [inner sep=0.8pt,outer sep=0.8pt] at (0,0) (4) {$\bullet$};
\node [inner sep=0.8pt,outer sep=0.8pt] at (1,0) (5) {$\bullet$};
\node [inner sep=0.8pt,outer sep=0.8pt] at (2,0) (6) {$\bullet$};
\node [inner sep=0.8pt,outer sep=0.8pt] at (3,0) (7) {$\bullet$};
\node [inner sep=0.8pt,outer sep=0.8pt] at (4,0) (8) {$\bullet$};
\node [inner sep=0.8pt,outer sep=0.8pt] at (0,-1) (2) {$\bullet$};
\node at (0,0.3) {};
\node at (0,-1.3) {};
\end{tikzpicture} 
\end{tabular}\\
\hline
\begin{tabular}{l}
$\sF_4(2)$%, \emph{collineation}
\end{tabular}&
\begin{tabular}{l}
\begin{tikzpicture}[scale=0.5]
\node at (0,0.5) {};
\node at (0,-0.3) {};
\node [inner sep=0.8pt,outer sep=0.8pt] at (-1.5,0) (1) {$\bullet$};
\node [inner sep=0.8pt,outer sep=0.8pt] at (-0.5,0) (2) {$\bullet$};
\node [inner sep=0.8pt,outer sep=0.8pt] at (0.5,0) (3) {$\bullet$};
\node [inner sep=0.8pt,outer sep=0.8pt] at (1.5,0) (4) {$\bullet$};
\draw [line width=0.5pt,line cap=round,rounded corners,fill=ggrey] (1.north west)  rectangle (1.south east);
\draw [line width=0.5pt,line cap=round,rounded corners,fill=ggrey] (2.north west)  rectangle (2.south east);
\draw [line width=0.5pt,line cap=round,rounded corners] (3.north west)  rectangle (3.south east);
\draw [line width=0.5pt,line cap=round,rounded corners] (4.north west)  rectangle (4.south east);
\draw (-1.5,0)--(-0.5,0);
\draw (0.5,0)--(1.5,0);
\draw (-0.5,0.07)--(0.5,0.07);
\draw (-0.5,-0.07)--(0.5,-0.07);
\node [inner sep=0.8pt,outer sep=0.8pt] at (-1.5,0) (1) {$\bullet$};
\node [inner sep=0.8pt,outer sep=0.8pt] at (-0.5,0) (2) {$\bullet$};
\node [inner sep=0.8pt,outer sep=0.8pt] at (0.5,0) (3) {$\bullet$};
\node [inner sep=0.8pt,outer sep=0.8pt] at (1.5,0) (4) {$\bullet$};
\end{tikzpicture}\qquad\begin{tikzpicture}[scale=0.5]
\node at (0,0.5) {};
\node at (0,-0.3) {};
\node [inner sep=0.8pt,outer sep=0.8pt] at (-1.5,0) (1) {$\bullet$};
\node [inner sep=0.8pt,outer sep=0.8pt] at (-0.5,0) (2) {$\bullet$};
\node [inner sep=0.8pt,outer sep=0.8pt] at (0.5,0) (3) {$\bullet$};
\node [inner sep=0.8pt,outer sep=0.8pt] at (1.5,0) (4) {$\bullet$};
\draw [line width=0.5pt,line cap=round,rounded corners] (1.north west)  rectangle (1.south east);
\draw [line width=0.5pt,line cap=round,rounded corners] (2.north west)  rectangle (2.south east);
\draw [line width=0.5pt,line cap=round,rounded corners,fill=ggrey] (3.north west)  rectangle (3.south east);
\draw [line width=0.5pt,line cap=round,rounded corners,fill=ggrey] (4.north west)  rectangle (4.south east);
\draw (-1.5,0)--(-0.5,0);
\draw (0.5,0)--(1.5,0);
\draw (-0.5,0.07)--(0.5,0.07);
\draw (-0.5,-0.07)--(0.5,-0.07);
\node [inner sep=0.8pt,outer sep=0.8pt] at (-1.5,0) (1) {$\bullet$};
\node [inner sep=0.8pt,outer sep=0.8pt] at (-0.5,0) (2) {$\bullet$};
\node [inner sep=0.8pt,outer sep=0.8pt] at (0.5,0) (3) {$\bullet$};
\node [inner sep=0.8pt,outer sep=0.8pt] at (1.5,0) (4) {$\bullet$};
\end{tikzpicture}\qquad
\begin{tikzpicture}[scale=0.5]
\node at (0,0.5) {};
\node at (0,-0.3) {};
\node [inner sep=0.8pt,outer sep=0.8pt] at (-1.5,0) (1) {$\bullet$};
\node [inner sep=0.8pt,outer sep=0.8pt] at (-0.5,0) (2) {$\bullet$};
\node [inner sep=0.8pt,outer sep=0.8pt] at (0.5,0) (3) {$\bullet$};
\node [inner sep=0.8pt,outer sep=0.8pt] at (1.5,0) (4) {$\bullet$};
\draw [line width=0.5pt,line cap=round,rounded corners,fill=ggrey] (1.north west)  rectangle (1.south east);
\draw [line width=0.5pt,line cap=round,rounded corners,fill=ggrey] (2.north west)  rectangle (2.south east);
\draw [line width=0.5pt,line cap=round,rounded corners,fill=ggrey] (3.north west)  rectangle (3.south east);
\draw [line width=0.5pt,line cap=round,rounded corners,fill=ggrey] (4.north west)  rectangle (4.south east);
\draw (-1.5,0)--(-0.5,0);
\draw (0.5,0)--(1.5,0);
\draw (-0.5,0.07)--(0.5,0.07);
\draw (-0.5,-0.07)--(0.5,-0.07);
\node [inner sep=0.8pt,outer sep=0.8pt] at (-1.5,0) (1) {$\bullet$};
\node [inner sep=0.8pt,outer sep=0.8pt] at (-0.5,0) (2) {$\bullet$};
\node [inner sep=0.8pt,outer sep=0.8pt] at (0.5,0) (3) {$\bullet$};
\node [inner sep=0.8pt,outer sep=0.8pt] at (1.5,0) (4) {$\bullet$};
\end{tikzpicture}
\end{tabular}\\
\hline
\begin{tabular}{l}
$\sF_4(2,4)$%, \emph{collineation}
\end{tabular}&
\begin{tabular}{l}
\begin{tikzpicture}[scale=0.5]
\node at (0,0.5) {};
\node at (0,-0.3) {};
\node [inner sep=0.8pt,outer sep=0.8pt] at (-1.5,0) (1) {$\bullet$};
\node [inner sep=0.8pt,outer sep=0.8pt] at (-0.5,0) (2) {$\bullet$};
\node [inner sep=0.8pt,outer sep=0.8pt] at (0.5,0) (3) {$\bullet$};
\node [inner sep=0.8pt,outer sep=0.8pt] at (1.5,0) (4) {$\bullet$};
\draw [line width=0.5pt,line cap=round,rounded corners,fill=ggrey] (1.north west)  rectangle (1.south east);
\draw [line width=0.5pt,line cap=round,rounded corners,fill=ggrey] (2.north west)  rectangle (2.south east);
\draw [line width=0.5pt,line cap=round,rounded corners] (3.north west)  rectangle (3.south east);
\draw [line width=0.5pt,line cap=round,rounded corners] (4.north west)  rectangle (4.south east);
\draw (-1.5,0)--(-0.5,0);
\draw (0.5,0)--(1.5,0);
\draw (-0.5,0.07)--(0.5,0.07);
\draw (-0.5,-0.07)--(0.5,-0.07);
\draw (-0.15,0.3) -- (0.08,0) -- (-0.15,-0.3);%arrow
\node [inner sep=0.8pt,outer sep=0.8pt] at (-1.5,0) (1) {$\bullet$};
\node [inner sep=0.8pt,outer sep=0.8pt] at (-0.5,0) (2) {$\bullet$};
\node [inner sep=0.8pt,outer sep=0.8pt] at (0.5,0) (3) {$\bullet$};
\node [inner sep=0.8pt,outer sep=0.8pt] at (1.5,0) (4) {$\bullet$};
\end{tikzpicture}
\end{tabular}\\
\hline
\end{tabular}%\vspace{-0.3cm}
\captionof{table}{Decorated opposition diagrams of uncapped automorphisms (exceptional types). The arrow in the $\sF_4(2,4)$ diagram indicates that the residues of type $\{1,2\}$ are projective planes of order $2$.
}\label{table:2}
\end{center}
%The arrow in the $\sF_4(2,4)$ diagram in Table~\ref{table:2} indicates that the residues of type $\{1,2\}$ are projective planes of order $2$, and those of type $\{3,4\}$ are projective planes of order~$4$.
%

Let us briefly describe corollaries to Theorem~\ref{thm:main*}(a) (see Section~\ref{sec:applications} for details and precise statements). Recall that the \textit{displacement} $\disp(\theta)$ of an automorphism $\theta$ is the maximum length of $\delta(C,C^{\theta})$, with $C$ a chamber. 

\begin{cor1}\label{cor:cor1} Let $\theta$ be an automorphism of a thick irreducible spherical building~$\Delta$. 
\begin{compactenum}[$(a)$]
\item If $\theta$ is an involution then $\theta$ is capped. 
\item If $\theta$ is uncapped then $\mathcal{T}(\theta)$ is determined by the decorated opposition diagram of~$\theta$. 
\item If $\theta$ is uncapped then $\disp(\theta)$ is determined by the decorated opposition diagram of~$\theta$. 
\end{compactenum}
\end{cor1}

In particular, if $\Delta$ has type $(W,S)$ and $J=\Type(\theta)$ then Corollary~\ref{cor:cor1}(c) implies that (see Corollary~\ref{cor:disp})
$$
\disp(\theta)=\begin{cases}
\mathrm{diam}(W)-\mathrm{diam}(W_{S\backslash J})&\text{if $\theta$ is capped}\\
\mathrm{diam}(W)-\mathrm{diam}(W_{S\backslash J})-1&\text{if $\theta$ is uncapped.}
\end{cases}
$$ 
To illustrate this in an example, it follows that if $\theta$ is a nontrivial automorphism of a thick $\sE_8$ building then $\mathrm{disp}(\theta)\in\{57,90,107,108,119,120\}$, which is a surprisingly restricted list of possibilities (see Remark~\ref{rem:disp}). Moreover, displacements of $107$ or $119$ can only occur for uncapped automorphisms of the small building~$\sE_8(2)$.
%
%All known examples of uncapped automorphisms have order at least~$4$, and we conjecture that all automorphisms of order~$3$ are capped. 
%
%

We also provide applications of Theorem~\ref{thm:main*}(a) to the study of \textit{domesticity} in spherical buildings (recall that an automorphism is called \textit{domestic} if it maps no chamber to an opposite chamber). These automorphisms have recently enjoyed extensive investigation, including the series \cite{TTM:11,TTM:12,TTM:12b} where domesticity in projective spaces, polar spaces, and generalised quadrangles is studied, \cite{HVM:12} where symplectic polarities of large $\sE_6$ buildings are classified in terms of domesticity, \cite{HVM:13} where domestic trialities of $\sD_4$ buildings are classified, and \cite{PTM:15} where domesticity in generalised polygons is studied. 

To give one example of our applications to domesticity, suppose that $\Delta$ is a simply laced spherical building, and that $\theta$ is a domestic automorphism inducing opposition on the type set with the property that $\theta$ maps at least one vertex of each type onto an opposite vertex (such automorphisms are called ``exceptional domestic''). Then we show that in fact $\theta$ maps simplices of each type $J\subsetneq S$ onto opposite simplices (such automorphisms are called ``strongly exceptional domestic''). In particular, this implies that $\disp(\theta)=\mathrm{diam}(\Delta)-1$ for exceptional domestic automorphisms.
%
%While Theorem~\ref{thm:main*} restricts the structure of an uncapped automorphism, it says nothing about the existence of such automorphisms. Thus our next main task is to prove that each of the decorated opposition diagrams listed in Tables~\ref{table:1} and~\ref{table:2} do indeed arise from automorphisms of buildings. We very nearly succeed in this task, with only the $\sE_8(2)$ diagrams remaining open due to the size of the building rendering our computational techniques inadequate. More precisely, we prove:
%
%\begin{thm1}\label{thm:main2}
%Let $\Delta$ be a small building. Each diagram appearing in the respective row of Table~\ref{table:1} or Table~\ref{table:2} can be realised as the decorated opposition diagram of some uncapped automorphism of $\Delta$, with the exception perhaps of the two $\sE_8(2)$ diagrams. 
%\end{thm1}
%
%We strongly believe that the two omitted $\sE_8(2)$ cases are indeed realised as opposition diagrams; see Conjecture~\ref{conj:2} for details. 

Theorem~\ref{thm:main*}(b) provides the first known examples of exceptional domestic automorphisms of spherical buildings of rank at least $3$ (examples were previously only known for generalised polygons; see \cite{PTM:15}). In fact Theorem~\ref{thm:main*}(b) shows that, with the possible exception of $\sE_8(2)$, every small building admits a  strongly exceptional domestic automorphism. 

%\begin{sloppypar}
The proof of Theorem~\ref{thm:main*}(b) for the small buildings of exceptional type involves computations using $\mathsf{MAGMA}$~\cite{MAGMA}, and in particular the Groups of Lie Type Package~\cite{CMT:04}.  In fact for the small buildings of type $\sF_4$ and $\sE_6$ we are able to prove a much stronger result and completely classify the domestic automorphisms of these buildings. To perform these calculations we implemented the minimal faithful permutation representations of the $\mathbb{ATLAS}$ groups $\sF_4(2)$, $\sF_4(2).2$, $\sE_6(2)$, $\sE_6(2).2$, ${^2}\sE_6(2^2)$, and ${^2}\sE_6(2^2).2$ (respective permutation degrees $69615$, $139230$, $139503$, $279006$, $3968055$ and $3968055$) into the $\mathsf{MAGMA}$ system. At the time of writing these representations were not readily available in either $\mathsf{MAGMA}$ or $\mathsf{GAP}$, and therefore they are provided on the first author's webpage. 
%\end{sloppypar} 
 
We conclude this introduction with an outline of the structure of the paper. In Section~\ref{sec:1} we provide definitions and background. The proofs of Theorem~\ref{thm:main*}(a) and its corollaries are contained in Section~\ref{sec:2}. The proof of Theorem~\ref{thm:main*}(b) is divided across Section~\ref{sec:classical} for the classical types and Section~\ref{sec:exceptional} for the exceptional types. Moreover, Section~\ref{sec:exceptional} contains the complete classification of domestic automorphisms of the small buildings of types $\sF_4$ and $\sE_6$.

\section{Definitions and background}\label{sec:1}

We refer to \cite{AB:08} for the general theory of buildings. In this section we will briefly recall some notation, mainly from \cite[Section~1]{PVM:17a}. Let $\Delta$ be a spherical building of type $(W,S)$, typically considered as a simplicial complex with type map $\type:\Delta\to 2^S$. Let $\cC$ be the set of chambers (maximal simplices) of $\Delta$, and let $\delta:\cC\times \cC\to W$ be the Weyl distance function.

Chambers $C$ and $D$ of $\Delta$ are \textit{opposite} if and only if they are at maximal distance in the chamber graph (with adjacency given by the union of the $s$-adjacency relations: $C\sim_s D$ if and only $\delta(C,D)=s$). Equivalently, chambers $C,D\in\cC$ are opposite if and only if
$
\delta(C,D)=w_0
$ where $w_0$ is the longest element of~$W$.

If $J\subseteq S$ we write $J^{\mathrm{op}}=J^{w_0}=w_0^{-1}Jw_0$ (the `opposite type' to $J$). The definition of opposition for chambers extends naturally to arbitrary simplices as follows (see~\cite[Lemma~5.107]{AB:08}). 

\begin{defn} Simplices $\alpha,\beta$ of $\Delta$ are \textit{opposite} if $\tau(\beta)=\tau(\alpha)^{\mathrm{op}}$ and there exists a chamber $A$ containing $\alpha$ and a chamber $B$ containing $\beta$ such that $A$ and $B$ are opposite.
\end{defn}

An \textit{automorphism} of $\Delta$ is a simplicial complex automorphism $\theta:\Delta\to\Delta$. Note that $\theta$ does not necessarily preserve types. Indeed each automorphism $\theta:\Delta\to\Delta$ induces a permutation $\pi_{\theta}$ of the type set $S$, given by $\delta(C,D)=s$ if and only if $\delta(C^{\theta},D^{\theta})=s^{\pi_{\theta}}$, and this permutation is a diagram automorphism of the Coxeter graph~$\Gamma$ of $(W,S)$. If $\Delta$ is irreducible, then from the classification of irreducible spherical Coxeter systems we see that $\pi_{\theta}:S\to S$ is either:

\begin{compactenum}[$(1)$]
\item the identity, in which case $\theta$ is called a \textit{collineation} (or \textit{type preserving}),
\item has order $2$, in which case $\theta$ is called a \textit{duality}, or
\item has order $3$, in which case $\theta$ is called a \textit{triality}; this case only occurs in type $\sD_4$.
\end{compactenum}
Automorphisms $\theta:\Delta\to\Delta$ that induce opposition on the type set (that is, $\pi_{\theta}=w_0$, where $w_0$ is the diagram automorphism  given by $s^{w_0}=w_0^{-1}sw_0$) are called \textit{oppomorphisms}. For example, oppomorphisms of an $\sE_6$ building are dualities, and oppomorphisms of an $\sE_7$ building are collineations (see, for example, \cite[Section~5.7.4]{AB:08}).

Let $\theta$ be an automorphism of $\Delta$. The \textit{opposite geometry} of $\theta$ is 
$$
\mathrm{Opp}(\theta)=\{\sigma\in\Delta\mid \sigma\text{ is opposite }\sigma^{\theta}\}.
$$ 
A fundamental result of Leeb~\cite[Section~5]{Lee:00} and Abramenko and Brown~\cite[Proposition~4.2]{AB:09} states that if $\theta$ is a nontrivial automorphism of a thick spherical building then $\mathrm{Opp}(\theta)$ is necessarily nonempty (this result has been generalised to the setting of twin buildings; see~\cite{DPM:13}).

The \textit{type} $\Type(\theta)$ of an automorphism $\theta$ is the union of all subsets $J\subseteq S$ such that there exists a type $J$ simplex in $\Opp(\theta)$. The \textit{opposition diagram} of $\theta$ is the triple $(\Gamma,\Type(\theta),\pi_{\theta})$. Less formally, the opposition diagram of $\theta$ is depicted by drawing $\Gamma$ and encircling the nodes of $\Type(\theta)$, where we encircle nodes in minimal subsets invariant under $w_0\circ \pi_{\theta}$. We draw the diagram `bent' (in the standard way) if $w_0\circ\pi_{\theta}\neq 1$. For example, consider the diagrams
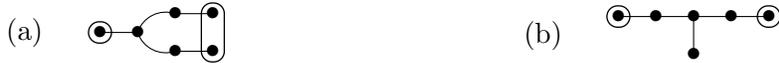
\begin{center}
\begin{tikzpicture}[scale=0.5]
\node at (-4,0) {(a)};
\node at (0,0.3) {};
\node [inner sep=0.8pt,outer sep=0.8pt] at (-2,0) (2) {$\bullet$};
\node [inner sep=0.8pt,outer sep=0.8pt] at (-1,0) (4) {$\bullet$};
\node [inner sep=0.8pt,outer sep=0.8pt] at (0,-0.5) (5) {$\bullet$};
\node [inner sep=0.8pt,outer sep=0.8pt] at (0,0.5) (3) {$\bullet$};
\node [inner sep=0.8pt,outer sep=0.8pt] at (1,-0.5) (6) {$\bullet$};
\node [inner sep=0.8pt,outer sep=0.8pt] at (1,0.5) (1) {$\bullet$};
\draw (-2,0)--(-1,0);
\draw (-1,0) to [bend left=45] (0,0.5);
\draw (-1,0) to [bend right=45] (0,-0.5);
\draw (0,0.5)--(1,0.5);
\draw (0,-0.5)--(1,-0.5);
\draw [line width=0.5pt,line cap=round,rounded corners] (2.north west)  rectangle (2.south east);
\draw [line width=0.5pt,line cap=round,rounded corners] (1.north west)  rectangle (6.south east);
\end{tikzpicture}\qquad\qquad\qquad\qquad\qquad
\begin{tikzpicture}[scale=0.5]
\node at (-4,-0.5) {(b)};
\node at (0,0.3) {};
\node [inner sep=0.8pt,outer sep=0.8pt] at (-2,0) (1) {$\bullet$};
\node [inner sep=0.8pt,outer sep=0.8pt] at (-1,0) (3) {$\bullet$};
\node [inner sep=0.8pt,outer sep=0.8pt] at (0,0) (4) {$\bullet$};
\node [inner sep=0.8pt,outer sep=0.8pt] at (1,0) (5) {$\bullet$};
\node [inner sep=0.8pt,outer sep=0.8pt] at (2,0) (6) {$\bullet$};
\node [inner sep=0.8pt,outer sep=0.8pt] at (0,-1) (2) {$\bullet$};
\draw (-2,0)--(2,0);
\draw (0,0)--(0,-1);
\draw [line width=0.5pt,line cap=round,rounded corners] (1.north west)  rectangle (1.south east);
\draw [line width=0.5pt,line cap=round,rounded corners] (6.north west)  rectangle (6.south east);
\end{tikzpicture} 
\end{center}
Diagram (a) represents a collineation $\theta$ of an $\sE_6$ building with $\Type(\theta)=\{1,2,6\}$, and diagram (b) represents a duality $\theta$ of an $\sE_6$ building with $\Type(\theta)=\{1,6\}$.

We call an opposition diagram \textit{empty} if no nodes are encircled (that is, $\Type(\theta)=\emptyset$), and \textit{full} if all nodes are encircled (that is, $\Type(\theta)=S$).

\begin{defn}
Let $\Delta$ be a spherical building of type $(W,S)$. Let $\theta$ be a nontrivial automorphism of $\Delta$, and let $J\subseteq S$. Then $\theta$ is called:
\begin{compactenum}[$(a)$]
\item \textit{capped} if there exists a type $\Type(\theta)$ simplex in $\Opp(\theta)$, and \textit{uncapped} otherwise.
\item \textit{domestic} if $\Opp(\theta)$ contains no chamber. 
\item \textit{$J$-domestic} if $\Opp(\theta)$ contains no type $J$ simplex (this terminology is reserved for subsets $J$ which are stable under $w_0\circ\pi_{\theta}$). 
\item \textit{exceptional domestic} if $\theta$ is domestic with full opposition diagram.
\item \textit{strongly exceptional domestic} if $\theta$ is domestic, but not $J$-domestic for any strict subset $J$ of $S$ invariant under $w_0\circ\pi_{\theta}$. 
\end{compactenum}
Note that if $\theta$ is a domestic automorphism with $w_0\circ\pi_{\theta}=1$ then $\theta$ is exceptional domestic if and only if there exists a vertex of each type mapped to an opposite vertex, and $\theta$ is strongly exceptional domestic if and only if there exists a panel of each cotype mapped to an opposite panel (recall that a \textit{panel} is a codimension~$1$ simplex). 
\end{defn}

To study uncapped automorphisms $\theta$ we introduce the decorated opposition diagram. Let $\mathcal{J}_{\theta}$ denote the set of subsets $I\subseteq S$ which are minimal with respect to the condition $I^{\pi_{\theta} w_0}=I$. For example, if $\theta$ induces opposition on $\Gamma$ then $\mathcal{J}_{\theta}=\{\{s\}\mid s\in S\}$ is the set of all singleton subsets of $S$.

\begin{defn}
The \textit{decorated opposition diagram} of an uncapped automorphism $\theta$ is the quadruple $(\Gamma,J,K_{\theta},\pi_{\theta})$ where $J=\Type(\theta)$ and $K_{\theta}\subseteq J$ is the union of all $J'\in\mathcal{J}_{\theta}$ such that there exists a type $J\backslash J'$ simplex mapped onto an opposite simplex. 
\end{defn}

Less formally, the decorated opposition diagram is drawn by shading the nodes of $K_{\theta}$ on the opposition diagram. For example, consider the following.
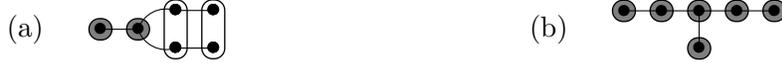
\begin{center}
\begin{tikzpicture}[scale=0.5,baseline=-0.5ex]
\node at (-4,0) {(a)};
\node at (0,0.3) {};
\node [inner sep=0.8pt,outer sep=0.8pt] at (-2,0) (2) {$\bullet$};
\node [inner sep=0.8pt,outer sep=0.8pt] at (-1,0) (4) {$\bullet$};
\node [inner sep=0.8pt,outer sep=0.8pt] at (0,-0.5) (5) {$\bullet$};
\node [inner sep=0.8pt,outer sep=0.8pt] at (0,0.5) (3) {$\bullet$};
\node [inner sep=0.8pt,outer sep=0.8pt] at (1,-0.5) (6) {$\bullet$};
\node [inner sep=0.8pt,outer sep=0.8pt] at (1,0.5) (1) {$\bullet$};
\draw [line width=0.5pt,line cap=round,rounded corners,fill=ggrey] (2.north west)  rectangle (2.south east);
\draw [line width=0.5pt,line cap=round,rounded corners,fill=ggrey] (4.north west)  rectangle (4.south east);
\draw [line width=0.5pt,line cap=round,rounded corners] (3.north west)  rectangle (5.south east);
\draw [line width=0.5pt,line cap=round,rounded corners] (1.north west)  rectangle (6.south east);
\draw (-2,0)--(-1,0);
\draw (-1,0) to [bend left=45] (0,0.5);
\draw (-1,0) to [bend right=45] (0,-0.5);
\draw (0,0.5)--(1,0.5);
\draw (0,-0.5)--(1,-0.5);
\node at (0,-0.5) {$\bullet$};
\node at (0,0.5) {$\bullet$};
\node at (1,-0.5) {$\bullet$};
\node at (1,0.5) {$\bullet$};
\node [inner sep=0.8pt,outer sep=0.8pt] at (-2,0) (2) {$\bullet$};
\node [inner sep=0.8pt,outer sep=0.8pt] at (-1,0) (4) {$\bullet$};
\node [inner sep=0.8pt,outer sep=0.8pt] at (0,-0.5) (5) {$\bullet$};
\node [inner sep=0.8pt,outer sep=0.8pt] at (0,0.5) (3) {$\bullet$};
\node [inner sep=0.8pt,outer sep=0.8pt] at (1,-0.5) (6) {$\bullet$};
\node [inner sep=0.8pt,outer sep=0.8pt] at (1,0.5) (1) {$\bullet$};
\end{tikzpicture}\qquad\qquad\qquad\qquad\qquad \begin{tikzpicture}[scale=0.5,baseline=-2ex]
\node at (-4,-0.5) {(b)};
\node at (0,0.3) {};
\node at (0,-1.3) {};
\node [inner sep=0.8pt,outer sep=0.8pt] at (-2,0) (1) {$\bullet$};
\node [inner sep=0.8pt,outer sep=0.8pt] at (-1,0) (3) {$\bullet$};
\node [inner sep=0.8pt,outer sep=0.8pt] at (0,0) (4) {$\bullet$};
\node [inner sep=0.8pt,outer sep=0.8pt] at (1,0) (5) {$\bullet$};
\node [inner sep=0.8pt,outer sep=0.8pt] at (2,0) (6) {$\bullet$};
\node [inner sep=0.8pt,outer sep=0.8pt] at (0,-1) (2) {$\bullet$};
\draw [line width=0.5pt,line cap=round,rounded corners,fill=ggrey] (1.north west)  rectangle (1.south east);
\draw [line width=0.5pt,line cap=round,rounded corners,fill=ggrey] (2.north west)  rectangle (2.south east);
\draw [line width=0.5pt,line cap=round,rounded corners,fill=ggrey] (3.north west)  rectangle (3.south east);
\draw [line width=0.5pt,line cap=round,rounded corners,fill=ggrey] (4.north west)  rectangle (4.south east);
\draw [line width=0.5pt,line cap=round,rounded corners,fill=ggrey] (5.north west)  rectangle (5.south east);
\draw [line width=0.5pt,line cap=round,rounded corners,fill=ggrey] (6.north west)  rectangle (6.south east);
\draw (-2,0)--(2,0);
\draw (0,0)--(0,-1);
\node [inner sep=0.8pt,outer sep=0.8pt] at (-2,0) (1) {$\bullet$};
\node [inner sep=0.8pt,outer sep=0.8pt] at (-1,0) (3) {$\bullet$};
\node [inner sep=0.8pt,outer sep=0.8pt] at (0,0) (4) {$\bullet$};
\node [inner sep=0.8pt,outer sep=0.8pt] at (1,0) (5) {$\bullet$};
\node [inner sep=0.8pt,outer sep=0.8pt] at (2,0) (6) {$\bullet$};
\node [inner sep=0.8pt,outer sep=0.8pt] at (0,-1) (2) {$\bullet$};
\end{tikzpicture} 
\end{center}
The decorated opposition diagram (a) represents an uncapped collineation of $\sE_6(2)$ with the property that there are simplices of types $S\backslash\{2\}$ and $S\backslash\{4\}$ mapped onto opposite simplices, and no simplices of types $S\backslash\{3,5\}$ nor $S\backslash\{1,6\}$ mapped onto opposite simplices -- this automorphism is exceptional domestic, but not strongly exceptional domestic. The diagram (b) represents an uncapped duality of $\sE_6(2)$ with the property that there are panels of each cotype mapped onto opposite panels -- this automorphism is strongly exceptional domestic. 

Residue arguments are used extensively in the proof of Theorem~\ref{thm:main*}(a), and so we conclude this section with a summary of the techniques. We first briefly define residues and projections (see \cite{AB:08} for details). The \textit{residue} $\Res(\alpha)$ of a simplex $\alpha\in\Delta$ is the set of all simplices of $\Delta$ which contain~$\alpha$, together with the order relation induced by that of~$\Delta$. Then $\Res(\alpha)$ is a building whose diagram is obtained from the diagram of $\Delta$ by removing all nodes which belong to $\tau(\alpha)$. The \textit{projection onto $\alpha$} is the map $\proj_{\alpha}:\Delta\to\Res(\alpha)$ defined as follows. Firstly, if $B$ is a chamber of $\Delta$ then there is a unique chamber $A\in \Res(\alpha)$ such that $\ell(\delta(A,B))<\ell(\delta(A',B))$ for all chambers $A'\in \Res(\alpha)$ with $A'\neq A$, and we define $\proj_{\alpha}(B)=A$. In other words, $\proj_{\alpha}(B)$ is the unique chamber $A$ of $\Res(\alpha)$ with the property that every minimal length gallery from $B$ to $\Res(\alpha)$ ends with the chamber~$A$. Now, if $\beta$ is an arbitrary simplex we define
$$
\proj_{\alpha}(\beta)=\bigcap_{B}\,\proj_{\alpha}(B)
$$
where the intersection is over all chambers $B$ in $\Res(\beta)$. In other words, $\proj_{\alpha}(\beta)$ is the unique simplex $\gamma$ of $\Res(\alpha)$ which is maximal subject to the property that every minimal length gallery from a chamber of $\Res(\beta)$ to $\Res(\alpha)$ ends in a chamber containing~$\gamma$.

Let $\theta$ be an automorphism of $\Delta$, and suppose that $\sigma\in\Opp(\theta)$. It follows from \cite[Theorem~3.28]{Tit:74} that the projection map $\proj_{\sigma}:\Res(\sigma^{\theta})\to\Res(\sigma)$ is an isomorphism. Define
$$
\theta_{\sigma}:\Res(\sigma)\xrightarrow{\sim} \Res(\sigma)\quad \text{by}\quad \theta_{\sigma}=\proj_{\sigma}\circ\,\theta.
$$
The type map induced by $\theta_{\sigma}$ is as follows.

\begin{prop}\label{prop:typemap}
Let $\theta$ be an automorphism of a spherical building $\Delta$ of type $(W,S)$. Suppose that $\sigma\in\Opp(\theta)$ and let $J=\tau(\sigma)$. Then the type map on $S\backslash J$ induced by $\theta_{\sigma}$ is $w_{S\backslash J}\circ w_0\circ \pi_{\theta}$. 
\end{prop}

\begin{proof}
This follows easily from \cite[Corollary~5.116]{AB:08}.
\end{proof}

\begin{example}
We will use Proposition~\ref{prop:typemap}  many times in our residue arguments. For example, consider a duality $\theta$ of an $\sD_n$ building, and suppose that $v\in\Opp(\theta)$ is a type $i$ vertex, with $i\leq n-2$. The residue of $v$ is a building of type $\sA_{i-1}\times \sD_{n-i}$, and the induced automorphism $\theta_v$ of $\Res(v)$ is a duality on the $\sA_{i-1}$ component, and a duality (respectively collineation) on the $\sD_{n-i}$ component if $i$ is even (respectively odd). 
\end{example}

\newpage

From \cite[Proposition~3.29]{Tit:74} we have:

\begin{prop}\label{prop:proj} Let $\theta$ be an automorphism of a spherical building $\Delta$ and let $\alpha\in\Opp(\theta)$. If $\beta\in\Res(\alpha)$ then $\beta$ is opposite $\beta^{\theta}$ in the building $\Delta$ if and only if $\beta$ is opposite $\beta^{\theta_{\alpha}}$ in the building $\Res(\alpha)$. 
\end{prop}

The following corollary facilitates inductive residue arguments. 

\begin{cor}\label{cor:proj}
Let $\theta:\Delta\to\Delta$ be a domestic automorphism and let $\sigma\in\Opp(\theta)$. Then $\theta_{\sigma}:\Res(\sigma)\to \Res(\sigma)$ is a domestic automorphism of the building~$\Res(\sigma)$.
\end{cor}

\begin{proof}
Let $J=\tau(\sigma)$. If $\theta_{\sigma}$ is not domestic then there is a chamber $\sigma'$ of $\Res(\sigma)$ mapped onto an opposite chamber by $\theta_{\sigma}$. Then $\sigma\cup\sigma'$ is a chamber of $\Delta$, and from Proposition~\ref{prop:proj} this chamber is mapped onto an opposite chamber, a contradiction. 
\end{proof}

\section{Theorem~\ref{thm:main*}(a) and its corollaries}\label{sec:2}

In this section we prove Theorem~\ref{thm:main*}(a) and give applications to determining the partially ordered set $\mathcal{T}(\theta)$, domesticity, cappedness of involutions, and calculating displacement.

\subsection{Proof of Theorem~\ref{thm:main*}(a)}

By \cite[Theorem~1]{PVM:17a} if $\theta$ is an uncapped automorphism of a thick irreducible spherical building $\Delta$ of rank at least $3$ then $\Delta$ is a small building. These are precisely the buildings listed in the first column of Tables~\ref{table:1} and~\ref{table:2}. Moreover, the following proposition from~\cite{PVM:17a} explains why collineations of $\sA_n$, trialities of $\sD_4$, and dualities of $\sF_4$ do not appear in Tables~\ref{table:1} and~\ref{table:2}. 

\begin{prop}\label{prop:1.1}
Every collineation of a thick $\sA_n$ building is capped, every triality of a thick $\sD_4$ building is capped, and every duality of a thick $\sF_4$ building is capped.
\end{prop}

\begin{proof}
See \cite[Corollary~3.9, Theorem~3.17, Lemma~4.1]{PVM:17a}.
\end{proof}

Buildings of type $\sA_n$ play an important role in our proof techniques owing to their prevalence as residues of spherical buildings of arbitrary type. Every thick building of type $\sA_n$ with $n>2$ is a projective space $\mathsf{PG}(n,\KK)$ over a division ring~$\KK$, where the type $i$ vertices of the building are the $(i-1)$-spaces of the projective space. Thus points have type~$1$, lines have type $2$, and so on.

\begin{defn}
Let $\mathbb{F}$ be a field. A duality of $\sA_{2n-1}(\mathbb{F})$ with
$
U^{\theta}=\{v\mid (u,v)=0\text{ for all $u\in U$}\}
$
for some nondegenerate symplectic form $(\cdot,\cdot)$ on $\FF^{2n}$ is called a \textit{symplectic polarity}. 
\end{defn}

Let us recall some useful facts concerning dualities of type $\sA$ buildings. 

\begin{lemma}[{\cite[Lemma~3.2]{TTM:11}}] \label{lem:An} If the projective space $\Delta=\mathsf{PG}(n,\KK)$ admits a duality $\theta$ for which all points are absolute (equivalently no type~$1$ vertex is mapped to an opposite), then $n$ is odd, $\KK$ is a field, and $\theta$ is a symplectic polarity.  
\end{lemma}

\begin{lemma}[{\cite[Lemma~3.4]{PVM:17a}}]\label{lem:sp}
If $\theta$ is a symplectic polarity of an $\sA_{2n-1}$ building $\Delta$ then $\theta$ is $\{i\}$-domestic for each odd~$i$, and each vertex mapped to an opposite vertex is contained in a type $\{2,4,\ldots,2n-2\}$ simplex mapped to an opposite simplex. In particular, symplectic polarities are capped.
\end{lemma}

\begin{thm}[{\cite[Theorems~3.10 and~3.11]{PVM:17a}}]\label{thm:Asmall}
Let $\theta$ be a domestic duality of the small building $\Delta=\sA_n(2)$ with $n\geq 2$. Then either $\theta$ is a strongly exceptional domestic duality or $n$ is odd and $\theta$ is a symplectic polarity. 
\end{thm}

The following proposition shows that the diagrams for uncapped dualities of $\sA_n$ buildings are as claimed in the first row of Table~\ref{table:1}. 

\begin{prop}\label{prop:1.2} Every uncapped duality of $\sA_n(2)$ is a strongly exceptional domestic duality.
\end{prop}

\begin{proof}
If $\theta$ is uncapped then necessarily $\theta$ is domestic, and so by Theorem~\ref{thm:Asmall} $\theta$ is either a symplectic polarity or is strongly exceptional domestic. The first case is eliminated by Lemma~\ref{lem:sp}.
\end{proof}

We now consider the small buildings of types $\sB_n$ and $\sD_n$. We first require some preliminary results. It is convenient at times to use terminology like ``$x$ is domestic for $\theta$'' and ``$x$ is non-domestic for $\theta$'' as short hand for ``$\theta$ does not map $x$ to an opposite'' and ``$x$ is mapped to an opposite by $\theta$''. If the automorphism $\theta$ is clear from context we will simply say ``$x$ is domestic'' or ``$x$ is non-domestic''.

\begin{lemma}\label{lem:type1}
Let $n\geq 4$ and let $\Delta$ be a building of type $\sB_n$ or $\sD_{n+2}$ with thick projective plane residues. Let $\theta$ be an automorphism and let $J=\Type(\theta)$. If there exists $j\in J$ odd with $j\leq n$, then $\{1,2,\ldots,j\}\subseteq J$. 
\end{lemma}

\begin{proof}
Let $v$ be a non-domestic type $j$ vertex. Then $\theta_v$ acts as a duality on the $\sA_{j-1}$ component of the residue of $v$ (by Proposition~\ref{prop:typemap}). Since $j$ is odd, this duality is either non-domestic or is exceptional domestic (see Theorem~\ref{thm:Asmall}), and in either case $1,2,\ldots,j-1\in J$, and hence the result.
\end{proof}

\begin{lemma}\label{lem:1} Let $\Delta$ be a building of type $\sB_n$ or $\sD_{n+2}$ with $n\geq 4$ and thick projective plane residues, and let $\theta$ be a collineation. Let $J=\Type(\theta)$. Suppose that $3\leq j<n$, and that $\{j-1,j\}\subseteq J$ and $j+1\notin J$. Then there exists a type $\{1,j\}$-simplex mapped onto an opposite simplex by $\theta$.
\end{lemma}

\begin{proof}
We first show that $\theta$ is not $\{j-1,j\}$-domestic. For if $\theta$ is $\{j-1,j\}$-domestic, then since $\theta$ is also $\{j-1,j+1\}$-domestic it follows from \cite[Lemma~3.25]{PVM:17a} that either $\theta$ is $\{j-1\}$-domestic or $\{j\}$-domestic, a contradiction. Thus there exists a type $\{j-1,j\}$ simplex $\sigma$ mapped onto an opposite. If $v$ is the type $j$ vertex of this simplex then $\theta_v$ acts as a duality on the $\sA_{j-1}$ component (Proposition~\ref{prop:typemap}) mapping a hyperplane to an opposite (by Proposition~\ref{prop:proj}). Thus $\theta_v$ is either non-domestic or strongly exceptional domestic on the $\sA_{j-1}$ component, and in either case there exists a non-domestic type $\{1,j\}$ simplex (note that $j-1\geq 2$). 
\end{proof}

\begin{lemma}\label{lem:2} Let $\Delta$ be a small building of type $\sB_n$ or $\sD_{n+1}$, and let $j<n$. Suppose that $\theta$ is an uncapped collineation of type $J=\{1,2,3,\ldots,j\}$. Then $\theta$ is $\{1,2,3,\ldots,j-1\}$-domestic.
\end{lemma}

\begin{proof}
Suppose that there is a non-domestic type $\{1,2,\ldots,j-1\}$ simplex, and let $v$ be the type $j-1$ vertex this simplex. If $\theta$ is uncapped then necessarily $\theta_v$ acts as the identity on the ``upper'' residue of type $\sB_{n-j+1}$ or $\sD_{n-j+2}$ (by Proposition~\ref{prop:proj}). Thus \cite[Lemma~3.28]{PVM:17a} with $i=j-2$ and $\ell=j-3$ (note the index shift due to the fact that we used projective dimension in~\cite{PVM:17a}) implies that every $(j-1)$-space in the polar space of $\Delta$ has a fixed point. Thus no type $j$ vertex of $\Delta$ is mapped onto an opposite vertex, contradicting the fact that $j\in J$.
\end{proof}

We can now complete the proof of Theorem~\ref{thm:main*}(a) for buildings of type $\sB_n$. We allow the additional generality of thin cotype~$n$ panels in the following proposition in order to facilitate our later arguments for type~$\sD_n$.

\begin{prop}\label{prop:B}
Let $\Delta$ be a (possibly non-thick) building of type $\sB_n$ with Fano plane residues and $n\geq 3$, and let $\theta$ be a collineation of $\Delta$. If $\theta$ is uncapped, then the decorated opposition diagram of $\theta$ is one of the diagrams in Table~\ref{table:1}.
\end{prop}

\begin{proof}
Suppose that $\theta$ is uncapped. Let $J=\Type(\theta)$, and let $j=\max J$. Then $j\geq 3$, for if $j=1$ then $\theta$ is capped, and if $j=2$ then either $J=\{2\}$ and $\theta$ is capped, or $J=\{1,2\}$ in which case \cite[Fact 3.21]{PVM:17a} implies that $\theta$ is capped.

We claim that $J$ contains an odd element. For if every element of $J$ is even then for each non-domestic type $j$-vertex $v$ the induced automorphism $\theta_v$ is a point domestic duality of an $\sA_{j-1}$ building (by Propositions~\ref{prop:typemap} and~\ref{prop:proj}). Thus $\theta_v$ is a symplectic polarity (Lemma~\ref{lem:An}), and so there exists a type $\{2,4,\ldots,j-2\}$ simplex of the residue mapped to an opposite (Lemma~\ref{lem:sp}). Hence by Proposition~\ref{prop:proj} there is a type $\{2,4,\ldots,j-2,j\}=J$ simplex of $\Delta$ mapped onto an opposite and so $\theta$ is capped, a contradiction. 

Let $k\in J$ be the maximal odd node. By Lemma~\ref{lem:type1} we have $\{1,2,\ldots,k\}\subseteq J$. Consider the following cases. 
\begin{compactenum}[$(1)$]
\item If $j=n$ then by \cite[Proposition~3.12(2)]{PVM:17a} there is a non-domestic type $\{1,n\}$ simplex. In the $\sA_{n-1}$ residue of the type $n$ vertex of this simplex we have a strongly exceptional domestic duality of $\sA_{n-1}$ (since it is domestic and maps a point to an opposite), and hence there are panels of each cotype $1,2,\ldots,n-1$ mapped onto opposites in $\Delta$. Thus $\theta$ has either the first diagram listed in Table~\ref{table:1} (with $j=n$) or the second diagram listed in Table~\ref{table:1} (strongly exceptional domestic). 
\item If $k=j<n$ then $J=\{1,2,\ldots,j\}$, and by Lemma~\ref{lem:1} there exists a non-domestic type $\{1,j\}$~simplex. Considering the type $\sA_{j-1}$ residue of the type $j$ vertex of this simplex, and noting that $j-1$ is even, we see that in $\Delta$ there are non-domestic simplices of each type $J\backslash\{j'\}$ with $j'=1,2,\ldots,j-1$ (using Theorem~\ref{thm:Asmall}), and hence the diagram of $\theta$ is either 
\begin{align}\label{eq:diags}
\begin{tikzpicture}[scale=0.5,baseline=-0.5ex]
\node at (0,0.3) {};
\node [inner sep=0.8pt,outer sep=0.8pt] at (-5,0) (-5) {$\bullet$};
\node [inner sep=0.8pt,outer sep=0.8pt] at (-4,0) (-4) {$\bullet$};
\node [inner sep=0.8pt,outer sep=0.8pt] at (-3,0) (-3) {$\bullet$};
\node [inner sep=0.8pt,outer sep=0.8pt] at (-2,0) (-2) {$\bullet$};
\node [inner sep=0.8pt,outer sep=0.8pt] at (0,0) (-1) {$\bullet$};
\node [inner sep=0.8pt,outer sep=0.8pt] at (1,0) (1) {$\bullet$};
\node [inner sep=0.8pt,outer sep=0.8pt] at (2,0) (2) {$\bullet$};
\node [inner sep=0.8pt,outer sep=0.8pt] at (3,0) (3) {$\bullet$};
\node [inner sep=0.8pt,outer sep=0.8pt] at (4,0) (4) {$\bullet$};
\node [inner sep=0.8pt,outer sep=0.8pt] at (5,0) (5) {$\bullet$};
\node at (2,-0.7) {$j$};
\draw [line width=0.5pt,line cap=round,rounded corners,fill=ggrey] (-5.north west)  rectangle (-5.south east);
\draw [line width=0.5pt,line cap=round,rounded corners,fill=ggrey] (-4.north west)  rectangle (-4.south east);
\draw [line width=0.5pt,line cap=round,rounded corners,fill=ggrey] (-3.north west)  rectangle (-3.south east);
\draw [line width=0.5pt,line cap=round,rounded corners,fill=ggrey] (-2.north west)  rectangle (-2.south east);
\draw [line width=0.5pt,line cap=round,rounded corners,fill=ggrey] (-1.north west)  rectangle (-1.south east);
\draw [line width=0.5pt,line cap=round,rounded corners,fill=ggrey] (1.north west)  rectangle (1.south east);
\draw [line width=0.5pt,line cap=round,rounded corners,fill=ggrey] (2.north west)  rectangle (2.south east);
\draw (-5,0)--(-2,0);
\draw (0,0)--(4,0);
\draw (4,0.07)--(5,0.07);
\draw (4,-0.07)--(5,-0.07);
\draw [style=dashed] (-2,0)--(0,0);
\node [inner sep=0.8pt,outer sep=0.8pt] at (-5,0) (-5) {$\bullet$};
\node [inner sep=0.8pt,outer sep=0.8pt] at (-4,0) (-4) {$\bullet$};
\node [inner sep=0.8pt,outer sep=0.8pt] at (-3,0) (-3) {$\bullet$};
\node [inner sep=0.8pt,outer sep=0.8pt] at (-2,0) (-2) {$\bullet$};
\node [inner sep=0.8pt,outer sep=0.8pt] at (0,0) (-1) {$\bullet$};
\node [inner sep=0.8pt,outer sep=0.8pt] at (1,0) (1) {$\bullet$};
\node [inner sep=0.8pt,outer sep=0.8pt] at (2,0) (2) {$\bullet$};
\node [inner sep=0.8pt,outer sep=0.8pt] at (3,0) (3) {$\bullet$};
\node [inner sep=0.8pt,outer sep=0.8pt] at (4,0) (4) {$\bullet$};
\node [inner sep=0.8pt,outer sep=0.8pt] at (5,0) (5) {$\bullet$};
\end{tikzpicture}\quad\text{or}\quad 
\begin{tikzpicture}[scale=0.5,baseline=-0.5ex]
\node at (0,0.3) {};
\node [inner sep=0.8pt,outer sep=0.8pt] at (-5,0) (-5) {$\bullet$};
\node [inner sep=0.8pt,outer sep=0.8pt] at (-4,0) (-4) {$\bullet$};
\node [inner sep=0.8pt,outer sep=0.8pt] at (-3,0) (-3) {$\bullet$};
\node [inner sep=0.8pt,outer sep=0.8pt] at (-2,0) (-2) {$\bullet$};
\node [inner sep=0.8pt,outer sep=0.8pt] at (0,0) (-1) {$\bullet$};
\node [inner sep=0.8pt,outer sep=0.8pt] at (1,0) (1) {$\bullet$};
\node [inner sep=0.8pt,outer sep=0.8pt] at (2,0) (2) {$\bullet$};
\node [inner sep=0.8pt,outer sep=0.8pt] at (3,0) (3) {$\bullet$};
\node [inner sep=0.8pt,outer sep=0.8pt] at (4,0) (4) {$\bullet$};
\node [inner sep=0.8pt,outer sep=0.8pt] at (5,0) (5) {$\bullet$};
\node at (2,-0.7) {$j$};
\draw [line width=0.5pt,line cap=round,rounded corners,fill=ggrey] (-5.north west)  rectangle (-5.south east);
\draw [line width=0.5pt,line cap=round,rounded corners,fill=ggrey] (-4.north west)  rectangle (-4.south east);
\draw [line width=0.5pt,line cap=round,rounded corners,fill=ggrey] (-3.north west)  rectangle (-3.south east);
\draw [line width=0.5pt,line cap=round,rounded corners,fill=ggrey] (-2.north west)  rectangle (-2.south east);
\draw [line width=0.5pt,line cap=round,rounded corners,fill=ggrey] (-1.north west)  rectangle (-1.south east);
\draw [line width=0.5pt,line cap=round,rounded corners,fill=ggrey] (1.north west)  rectangle (1.south east);
\draw [line width=0.5pt,line cap=round,rounded corners] (2.north west)  rectangle (2.south east);
\draw (-5,0)--(-2,0);
\draw (0,0)--(4,0);
\draw (4,0.07)--(5,0.07);
\draw (4,-0.07)--(5,-0.07);
\draw [style=dashed] (-2,0)--(0,0);
\node [inner sep=0.8pt,outer sep=0.8pt] at (-5,0) (-5) {$\bullet$};
\node [inner sep=0.8pt,outer sep=0.8pt] at (-4,0) (-4) {$\bullet$};
\node [inner sep=0.8pt,outer sep=0.8pt] at (-3,0) (-3) {$\bullet$};
\node [inner sep=0.8pt,outer sep=0.8pt] at (-2,0) (-2) {$\bullet$};
\node [inner sep=0.8pt,outer sep=0.8pt] at (0,0) (-1) {$\bullet$};
\node [inner sep=0.8pt,outer sep=0.8pt] at (1,0) (1) {$\bullet$};
\node [inner sep=0.8pt,outer sep=0.8pt] at (2,0) (2) {$\bullet$};
\node [inner sep=0.8pt,outer sep=0.8pt] at (3,0) (3) {$\bullet$};
\node [inner sep=0.8pt,outer sep=0.8pt] at (4,0) (4) {$\bullet$};
\node [inner sep=0.8pt,outer sep=0.8pt] at (5,0) (5) {$\bullet$};
\end{tikzpicture}
\end{align}
The first digram is eliminated by Lemma~\ref{lem:2}. 
\item If $k<j<n$ then $j$ is even, and as above we have $\{2,4,\ldots,j-2,j\}\subseteq J$. In particular $\{k,k+1\}\subseteq J$ and $k+2\notin J$ (as $k$ is maximum odd node of $J$, and note that $k+2\leq n$). Lemma~\ref{lem:1} implies that there is a non-domestic type $\{1,k+1\}$ simplex. If $k+1=j$ then as above we have the diagrams~(\ref{eq:diags}) and Lemma~\ref{lem:2} eliminates the first of the diagrams. If $k+1<j$ then $k+3\leq j<n$. If $\theta$ is $\{1,k+3\}$-domestic, then since $\theta$ is not $\{k+3\}$-domestic, \cite[Lemma 3.29]{PVM:17a} implies that $\theta$ is $\{1,k+1\}$-domestic, a contradiction. Hence there exists a type $\{1,k+3\}$ simplex mapped onto an opposite. However, considering the $\sA_{k+2}$ residue of the type $k+3$ vertex of this simplex we see that $\theta$ is not $\{k+2\}$-domestic, contradicting the maximality of $k$. 
\end{compactenum}
Hence the result.
\end{proof}

\begin{cor}\label{cor:B1i}
Let $\Delta$ be a building of type $\sB_n$ with thick projective spaces, and let $\theta$ be a collineation and $n\geq i\geq 3$. If $\theta$ is $\{1,i\}$-domestic then $\theta$ is either $\{1\}$-domestic or $\{i\}$-domestic. 
\end{cor}

\begin{proof}
If $\theta$ is capped then the result is true by definition. If $\theta$ is uncapped then the result follows directly from the classification of uncapped diagrams given above.
\end{proof}

\begin{remark} The assumption $i\geq 3$ cannot be removed from Corollary~\ref{cor:B1i}. For example, consider the exceptional domestic collineation of the generalised quadrangle $\sB_2(2)$ (see \cite[Section~4]{TTM:12b}) . More generally, for each $n\geq 2$ there exists an uncapped collineation of $\sB_n(2)$ with $\Type(\theta)=\{1,2\}$ (see Theorem~\ref{thm:existenceBn(2)}). 
\end{remark}

We now continue with the analysis of buildings of type $\sD_n$. Recall that each building of type $\sD_n$ can be realised as the oriflamme geometry of the space $\mathbb{F}^{2n}$ equipped with an orthogonal form of Witt index $n$, for some field $\mathbb{F}$. The vertices of type $j$ for $j\in\{1,\ldots,n-2\}$ are the totally isotropic spaces of dimension~$j$, and the vertices of type $n-1$ and $n$ are the totally isotropic subspaces of dimension~$n$ (corresponding to the orbits of the action of the associated simple orthogonal group). To each such building $\Delta$ of type $\sD_n$ there is an associated (non-thick) building $\Delta'$ of type~$\sB_n$. The type $j$ vertices of $\Delta'$, for $1\leq j\leq n$, are the totally isotropic subspaces of dimension~$j$. Each type $n-1$ vertex of $\Delta'$ determines a type $\{n-1,n\}$ simplex of $\Delta$, and vice versa, as follows. A type $n-1$ vertex of $\Delta'$ is an $(n-1)$-dimensional totally isotropic space~$W$, and there are precisely two totally isotropic $n$-dimensional subspaces $U,V$ containing $W$ and $(U,V)$ is an $\{n-1,n\}$-simplex of $\Delta$. Conversely, if $(U,V)$ is a type $\{n-1,n\}$ simplex of $\Delta$ then $W=U\cap V$ is a type $n-1$ vertex of $\Delta'$. 

We first recall two facts from \cite{PVM:17a}.

\begin{lemma}[{\cite[Lemma~3.32]{PVM:17a}}]\label{lem:containedB}
Let $\Delta$ be a thick building of type $\sD_n$ with $n$ odd, and let $\Delta'$ be the associated non-thick $\sB_n$ building. A collineation $\theta$ maps a type $\{n-1,n\}$ simplex of $\Delta$ to an opposite simplex if and only if it maps the associated type $n-1$ vertex of $\Delta'$ to an opposite vertex. 
\end{lemma}

\begin{lemma}[{\cite[Proposition~3.16]{PVM:17a}}]\label{lem:Ddual}
No duality of a thick building of type $\sD_n$ is $\{1\}$-domestic.
\end{lemma}

\begin{lemma}\label{lem:1nDodd}
Let $\Delta$ be a thick building of type $\sD_n$ with $n\geq 5$ odd, and let $\theta$ be a collineation. If $\theta$ is $\{1,n-1,n\}$-domestic then $\theta$ is either $\{1\}$-domestic or $\{n-1,n\}$-domestic.  
\end{lemma}

\begin{proof} 
Suppose that $\theta$ is neither $\{1\}$-domestic nor $\{n-1,n\}$-domestic. Since $\theta$ maps a type $\{n-1,n\}$-simplex to an opposite, by familiar residue arguments there are vertices of types $2,4,\ldots,n-3$ mapped onto opposite vertices. These vertex types are therefore also mapped onto opposites in the associated non-thick $\sB_n$ building~$\Delta'$. If there are no type $n-2$ or $n-1$ vertices of $\Delta'$ mapped onto opposite vertices, then $\theta$ is $\{n-3,n-2\}$-domestic and $\{n-3,n-1\}$-domestic (on $\Delta'$) and thus since $\theta$ is not $\{n-3\}$-domestic it follows from \cite[Lemma~3.25]{PVM:17a} that every space of vector space dimension at least $n-2$ contains a fixed point. However by Lemma~\ref{lem:containedB} there are $n-1$ dimensional spaces mapped onto opposites, a contradiction. Thus either (i) $\theta$ is not $\{n-3,n-2\}$-domestic, or (ii) $\theta$ is not $\{n-3,n-1\}$-domestic (on $\Delta'$). 

Consider case (i). Let $v$ be the type $n-2$ vertex of a non-domestic type $\{n-3,n-2\}$ simplex. Then $\theta_v$ acts on the upper type $\sA_1\times\sA_1$ residue by permuting the components, and thus $\theta_v$ is non-domestic on this upper residue (see \cite[Lemma~3.7]{PVM:17a}). Moreover $\theta_v$ is a duality on the lower type $\sA_{n-3}$ residue mapping a hyperplane (a type $n-3$ vertex) of this residue onto an opposite, and thus $\theta_v$ also maps a point (a type $1$ vertex) to an opposite. Thus $\theta$ maps a type $\{1,n-1,n\}$ simplex to an opposite, a contradiction. 

Consider case (ii). Since $\theta$ is neither $\{1\}$-domestic nor $\{n-1\}$-domestic on $\Delta'$, and since $n-1\leq 4$, Corollary~\ref{cor:B1i} implies that there exists a type $\{1,n-1\}$ simplex of $\Delta'$ mapped to an opposite. Now Lemma~\ref{lem:containedB} implies that $\theta$ is not $\{1,n-1,n\}$-domestic on $\Delta$. This contradiction establishes the result.
\end{proof}

\begin{prop}\label{prop:D1}
Let $\Delta$ be the building $\sD_n(2)$, $n\geq 4$, and let $\theta$ be a collineation of $\Delta$. If $\theta$ is uncapped then the decorated opposition diagram of $\theta$ is contained in Table~\ref{table:1}.
\end{prop}

\begin{proof}
Let $\theta$ be an uncapped collineation of $\sD_n(2)$, and let $J=\Type(\theta)$. Let $j=\max J$. 
\smallskip

\noindent\textit{Case 1: $j\in\{n-1,n\}$ with $n$ odd}. Then necessarily $\{n-1,n\}\subseteq J$. If $J\backslash\{n-1,n\}$ contains no odd types, then the induced automorphism in every residue of a non-domestic $\{n-1,n\}$-simplex is a symplectic polarity, and hence $\theta$ is capped, a contradiction. Thus $J\backslash \{n-1,n\}$ contains an odd node, and so by Lemma~\ref{lem:type1} we have $1\in J$. Thus by Lemma~\ref{lem:1nDodd} there exists a type $\{1,n-1,n\}$ simplex mapped onto an opposite simplex, and it easily follows that $\theta$ maps simplices of each type $S\backslash\{i\}$ with $i=1,2,\ldots,n-2$ to opposite. Hence the claimed diagram. 
\smallskip

\noindent\textit{Case 2: $j\in\{n-1,n\}$ with $n$ even}. By duality symmetry we may assume that $j=n$. If $n-1\in J$, then by \cite[Proposition~3.12(3)(b)]{PVM:17a} there is a type $\{n-1,n\}$-simplex mapped onto an opposite, and then considering the type $\sA_{n-2}$ residue we easily deduce that there are simplices of each cotype $S\backslash \{i\}$ with $i=1,2,\ldots,n-2$ mapped onto opposites. It then easily follows that there are also simplices of each type $S\backslash \{n-1\}$ and $S\backslash \{n\}$ mapped onto opposite. So suppose that $n-1\notin J$. If $J\backslash\{n-1,n\}$ contains no odd indices, then as above we deduce that $\theta$ is capped. Thus $J\backslash\{n-1,n\}$ contains an odd node, and so $1\in J$ by Lemma~\ref{lem:type1}, and by \cite[Proposition~3.12(3)(a)]{PVM:17a} there is a type $\{1,n\}$ simplex mapped onto an opposite. It now easily follows that $\theta$ is strongly exceptional domestic. 
\smallskip

\noindent\textit{Case 3: $j\notin\{n-1,n\}$.} If $j$ is odd, then considering the upper residue of a type $j$ non-domestic we obtain a duality of a $\sD_{n-j}$, and since every duality of a $\sD_{n-j}$ maps a point to an opposite point (Lemma~\ref{lem:Ddual}) we have $j+1\in J$, a contradiction. Thus $j$ is even. If $j=2$ then $\theta$ is capped (see \cite[Fact 3.22]{PVM:17a}). So $j\geq 4$ (and hence $n\geq 6$). If $J$ has only even types then clearly $\theta$ is capped. Thus $J$ contains an odd node, and hence by Lemma~\ref{lem:type1} we have $1\in J$. Applying Corollary~\ref{cor:B1i} in the non-thick $\sB_n$ building it follows that there is a type $\{1,j\}$-simplex mapped onto an opposite, and the result easily follows, using Lemma~\ref{lem:2} to show that the last node is not shaded. 
\end{proof}

\begin{prop}\label{prop:D2}
Let $\theta$ be a duality of the $\sD_n(2)$ building. If $\theta$ is uncapped then the decorated opposition diagram of $\theta$ is contained in Table~\ref{table:1}. 
\end{prop}

\begin{proof}
Let $\theta$ be an uncapped duality of $\sD_n(2)$, and let $J=\Type(\theta)$. Let $j=\max J$. 
\smallskip

\noindent\textit{Case 1: $j\in\{n-1,n\}$ with $n$ even}. Then necessarily $\{n-1,n\}\subseteq J$. In the residue of such a simplex we have an exceptional domestic duality of $\sA_{n-2}(2)$, and and the result easily follows.
\smallskip

\noindent\textit{Case 2: $j\in\{n-1,n\}$ with $n$ odd}. In the residue of a non-domestic type $j$ vertex we obtain an exceptional domestic duality of $\sA_{n-1}(2)$, and again the result easily follows. 
\smallskip

\noindent\textit{Case 3: $j\notin\{n-1,n\}$.} If $j$ is even, then considering the upper residue of a non-domestic type $j$ vertex we obtain a duality of $\sD_{n-j}(2)$, and since every duality of $\sD_{n-j}(2)$ maps a point to an opposite point we have $j+1\in J$, a contradiction. Thus $j$ is odd. If $j=1$ then $\theta$ is obviously capped. So $j\geq 3$ (and hence $n\geq 5$). In the lower residue of a non-domestic type $j$ vertex we obtain an exceptional domestic duality of $\sA_{j-1}(2)$, and hence the result, using Lemma~\ref{lem:2} to see that the last node is not shaded. 
\end{proof}

Propositions~\ref{prop:D1} and~\ref{prop:D2} establish Theorem~\ref{thm:main*}(a) for  buildings of type $\sD_n$. We now consider the exceptional types.

\begin{lemma}\label{lem:F41234}
Let $\Delta$ be the building $\sF_4(2)$, and let $\theta$ be a collineation. If $\Type(\theta)=\{1,2,3,4\}$ then there exists either a non-domestic type $\{1,2\}$ simplex, or a non-domestic type $\{3,4\}$ simplex. 
\end{lemma}

\begin{proof}
This follows from the classification given in Theorem~\ref{thm:F4}. We note that no circular logic is introduced by postponing the proof until Section~\ref{sec:exceptional}. 
\end{proof}

We are now ready to prove Theorem~\ref{thm:main*}(a) for the small exceptional buildings. Before doing so we would like to correct \cite[Main Result 2.2]{HVM:12}, where it is asserted that every domestic duality of an $\sE_6$ building is a symplectic polarity. In fact this result only holds for large $\sE_6$ buildings. The oversight in the proof of \cite[Main Result 2.2]{HVM:12} is in the proof of \cite[Lemma~5.2]{HVM:12}, where the existence of exceptional domestic automorphisms of $\sA_4(2)$ is overlooked.

\begin{prop}\label{prop:exceptional}
If $\theta$ is an uncapped automorphism of a building of exceptional type then the decorated opposition diagram of $\theta$ is contained in Table~\ref{table:2}. 
\end{prop}

\begin{proof}
(1) Let $\theta$ be an uncapped collineation of $\sE_6(2)$ and let $J=\Type(\theta)$. Suppose that $J=S$, and so the opposition diagram has the subsets $\{2\}$, $\{4\}$, $\{3,5\}$ and $\{1,6\}$ encircled. Let $\sigma$ be a non-domestic type $\{3,5\}$ simplex. Then $\theta_{\sigma}$ is an automorphism of an $\sA_2\times\sA_1\times\sA_1$ building acting as a duality on the $\sA_2$ component and interchanging the two $\sA_1$ components (by Proposition~\ref{prop:typemap}). Thus $\theta_{\sigma}$ is not domestic on the $\sA_1\times\sA_1$ component (see \cite[Lemma~3.7]{PVM:17a}) and must be exceptional domestic on the $\sA_2$ component (for otherwise $\theta$ is capped). Hence there are non-domestic simplices of types $S\backslash \{2\}$ and $S\backslash \{4\}$, and so the encircled nodes $2$ and $4$ are shaded. Suppose that there is a non-domestic simplex $\sigma'$ either of type $S\backslash\{3,5\}$ or $S\backslash\{1,6\}$. Then $\theta_{\sigma'}$ is an automorphism of an $\sA_1\times\sA_1$ building interchanging the two components (again by Proposition~\ref{prop:typemap}), and hence is not domestic, and hence $\theta$ is capped, a contradiction. Thus the encircled subsets $\{3,5\}$ and $\{1,6\}$ are not shaded.

Suppose that $J\neq S$. Then the first argument of the previous paragraph shows that $\{3,5\}\cap J=\emptyset$. A similar argument shows that $4\notin J$. Thus if $J\neq S$ we have $\{3,4,5\}\cap J=\emptyset$. If $\{1,6\}\subseteq J$ then $2\in J$ (for in the residue of a non-domestic type $\{1,6\}$ simplex we obtain a duality of $\sD_4$, and no duality of $\sD_n$ is point domestic; see \cite[Proposition~3.16]{PVM:17a}), and $\theta$ is capped. If $J=\{2\}$ then $\theta$ is obviously capped. Thus there are no uncapped collineations of $\sE_6$ with $\Type(\theta)\neq S$.

(2) Let $\theta$ be an uncapped duality of an $\sE_6$ building and let $J=\Type(\theta)$. We claim that $J=S$. If $1\in J$ then $6\in J$, and vice versa (since no duality of $\sD_n$ is point domestic), and this argument shows that if $J=\{1,6\}$ then $\theta$ is capped, a contradiction. So $\{2,3,4,5\}\cap J\neq \emptyset $. If $3\in J$ then $\{2,3,4,5,6\}\subseteq J$ (considering the $\sA_4$ component of the residue of a non-domestic type $3$ vertex) and similarly if $5\in J$ then $\{1,2,3,4,5\}\subseteq J$. Thus if either $3\in J$ or $5\in J$ then $J=S$. If $2\in J$ then $\{2,3,5\}\subseteq J$ (considering the $\sA_5$ residue of a non-domestic type $2$ vertex), and thus again $J=S$. If $4\in J$ then $\{1,3,4,5,6\}\subseteq J$ (considering the $\sA_2\times\sA_2$ component of the residue of a non-domestic type $4$ vertex), and so once more $J=S$. 

Thus all nodes are encircled. We claim that $\theta$ is strongly exceptional domestic, and so all nodes are shaded. To prove that there exists a cotype $j$ panels mapped onto opposite panels for each $j\in\{1,3,4,5,6\}$, note first that there exists a non-domestic type $\{2,4\}$ simplex (by considering the $\sA_4$ component of the residue of a non-domestic type $3$ vertex). If $v$ is the type $2$ vertex of such a simplex, then $\theta_v$ is a domestic duality of $\sA_5$ mapping a plane of this projective space onto an opposite, and thus $\theta_v$ is strongly exceptional domestic, and hence the result. Finally, to see that there is a non-domestic cotype $2$ panel, let $v$ be the type $1$ vertex of a non-domestic cotype $4$ panel. Using the classification of uncapped $\sD_5$ diagrams we see that $\theta_v$ is strongly exceptional domestic, and it follows that there exists a cotype $2$ panel of $\sE_6$ mapped onto an opposite. 

(3) Let $\theta$ be an uncapped collineation of an $\sE_7$ building and let $J=\Type(\theta)$. If $J=S$ then $\theta$ is strongly exceptional domestic (considering the $\sA_6$ residue of a non-domestic type $2$ vertex shows that $\theta$ maps simplices of each type $S\backslash\{j\}$ onto opposites for $j=1,3,4,5,6,7$, and considering the $\sE_6$ residue of the type $7$ vertex of a non-domestic type $\{2,7\}$ simplex, and using (2), shows that there is a simplex of type $S\backslash\{2\}$ mapped onto an opposite). 

Suppose that $J\neq S$. Then $2\notin J$ (for otherwise the induced duality of the $\sA_6$ residue is strongly exceptional domestic) and $5\notin J$ (for otherwise the induced dualities of the $\sA_4$ and $\sA_2$ residues are both strongly exceptional domestic). We note the following: If $3\in J$ then $\{3,4,6\}\subseteq J$ (considering the $\sA_5$ component of the residue) and if $4\in J$ then $\{1,3,4,6\}\subseteq J$ (considering the $\sA_2$ and $\sA_3$ components of the residue). Thus if either $3\in J$ or $4\in J$ then $\{1,3,4,6\}\subseteq J$. If $6\in J$ then $\{1,6\}\subseteq J$ (since no duality of the $\sD_5$ component of the residue is point domestic). If $7\in J$ then $\{1,6,7\}\subseteq J$ (since every duality of $\sE_6$ maps both type $1$ and type $6$ vertices to opposites). It follows that either $J=\{1\}$, $J=\{1,6\}$, $J=\{1,6,7\}$, $J=\{1,3,4,6\}$, or $J=\{1,3,4,6,7\}$. In the first, second, and third cases it is clear using the above arguments that $\theta$ is capped, a contradiction. We claim that $J=\{1,3,4,6,7\}$ is impossible (for any collineation, capped or uncapped). For if $J=\{1,3,4,6,7\}$ then by \cite[Proposition~4.3(2)]{PVM:17a} there exists a type $\{3,7\}$ simplex $\sigma$ mapped to an opposite simplex, and if $v$ is the type $7$ vertex of $\sigma$ then $\theta_v$ is a duality of an $\sE_6$ building mapping a type~$3$ vertex to an opposite, thus forcing $2,5\in J$, a contradiction.

The previous paragraph shows that if $\theta$ is uncapped and $J\neq S$ then $J=\{1,3,4,6\}$. Considering the $\sA_2\times\sA_3$ component of the residue of a non-domestic type $4$ vertex shows that there are simplices of types $\{3,4,6\}$ and $\{1,4,6\}$ mapped onto opposites, thus the nodes $1$ and $3$ are shaded. If there exist either type $\{1,3,6\}$ or $\{1,3,4\}$ simplices mapped onto opposite simplices then considering the residue of the type $1$ vertex of such a simplex we deduce that $\theta$ is capped, a contradiction. Thus the nodes $4$ and $6$ are not shaded. 

(4) Let $\theta$ be an uncapped (hence nontrivial) collineation of an $\sE_8$ building and let $J=\Type(\theta)$. If $J=S$ then easy residue arguments show that $\theta$ is strongly exceptional domestic. 

We claim that if $J\neq S$ then $J\subseteq \{1,6,7,8\}$. To see this, note that if $2\in J$ then $\{3,5,7\}\in J$ (considering an $\sA_7$ residue), if $3\in J$ then $\{2,4,5,6,7,8\}\subseteq J$ (considering the $\sA_6$ component of the residue), if $4\in J$ then $\{1,3,5,6,7,8\}\subseteq J$ (considering the $\sA_2\times\sA_4$ component of the residue), and if $5\in J$ then $\{1,2,3,4,7\}\subseteq J$ (considering the $\sA_4\times\sA_3$ residue). Combining these statements it follows that if $\{2,3,4,5\}\cap J\neq \emptyset$ then $J=S$, and hence the claim. 

Suppose that $J\neq S$, and so $J\subseteq \{1,6,7,8\}$. We claim that $J=\{1,6,7,8\}$. For if $1\in J$ then $8\in J$ (since no duality of $\sD_7$ is point domestic), if $6\in J$ then $J=\{1,6,7,8\}$ (considering the $\sD_5\times \sA_2$ residue and recalling that no duality of $\sD_5$ is point domestic), and if $7\in J$ then $6\in J$ (considering the duality of $\sE_6$ and using (2) above) and so again $J=\{1,6,7,8\}$. Thus $J=\{8\},\{1,8\}$ or $\{1,6,7,8\}$. The first two cases are clearly capped, hence the claim. Now considering the residue of a type $6$ non-domestic vertex we see that there are simplices of types $\{1,6,7\}$ and $\{1,6,8\}$ mapped onto opposite simplices (hence the nodes $7$ and $8$ are shaded). If there exists a simplex of type $\{6,7,8\}$ or $\{1,7,8\}$ mapped onto an opposite then considering the $\sD_5$ residue we deduce that $\theta$ is capped, and so the nodes $1$ and $6$ are not shaded.

(5) Let $\theta$ be an uncapped collineation of an $\sF_4$ building and let $J=\Type(\theta)$. If $2\in J$ then $3,4\in J$ (by the duality in the $\sA_2$ component of the residue) and similarly if $3\in J$ then $1,2\in J$. Thus either $J=\{1\}$,  $J=\{4\}$, $J=\{1,4\}$, or $J=\{1,2,3,4\}$. The first and second cases are trivially capped. The third case is capped by \cite[Lemma~4.5]{PVM:17a}. Thus $J=\{1,2,3,4\}$. 

If $\Delta=\sF_4(2)$ then by Lemma~\ref{lem:F41234} there is either a type $\{1,2\}$ or $\{3,4\}$ simplex mapped onto an opposite simplex. In the first case, by considering the residue of the type $2$ vertex, we see that there are panels of cotype $3$ and $4$ mapped onto opposites, and hence the nodes $3$ and $4$ are shaded. The second case is symmetric, with the nodes $1$ and $2$ shaded. Of course both cases may occur simultaneously, and then all nodes are shaded. Finally, note that if either nodes $1$ or $2$ are shaded then both are shaded (if the $i$ node is shaded and $i\in\{1,2\}$ then consider the residue of the type $3$ vertex of a non-domestic cotype $i$ panel). Similarly, if either nodes $3$ or $4$ are shaded then both are shaded. Hence the result for $\sF_4(2)$. 

If $\Delta=\sF_4(2,4)$ then considering the $\sA_2(4)$ component of a type $2$ non-domestic vertex we deduce that there are simplices of type $\{2,3,4\}$ mapped onto opposites. Then considering the $\sA_2(2)$ residue of a type $\{3,4\}$ non-domestic simplex we deduce that there are also simplices of type $\{1,3,4\}$ mapped onto opposites. Thus the nodes $1,2$ are shaded. If there exists a simplex of type $\{1,2,4\}$ or $\{1,2,3\}$ mapped onto an opposite, then considering the type $\sA_2(4)$ residue of the $\{1,2\}$ subsimplex we deduce that $\theta$ is non-domestic, and hence capped, a contradiction. Thus the nodes $3$ and $4$ are not shaded. 
\end{proof}

Theorem~\ref{thm:main*}(a) now follows from Propositions~\ref{prop:1.1}, \ref{prop:1.2}, \ref{prop:B}, \ref{prop:D1}, \ref{prop:D2}, and~\ref{prop:exceptional}.

\subsection{Applications}\label{sec:applications}

This section contains applications and corollaries of Theorem~\ref{thm:main*}(a).

\begin{cor}\label{cor:app1} Let $\theta$ be a an exceptional domestic automorphism of a thick irreducible spherical building~$\Delta$.
\begin{compactenum}[$(a)$]
\item If $\theta$ is an oppomorphism and $\Delta$ is simply laced, then $\theta$ is strongly exceptional domestic.
\item If $\theta$ is not an oppomorphism then $\theta$ is not strongly exceptional domestic. 
\end{compactenum}
\end{cor}

\begin{proof}
The first statement follows by noting that in Tables~\ref{table:1} and~\ref{table:2}, if $\theta$ is an oppomorphism and $\Delta$ is simply laced, then whenever all nodes are encircled they are all shaded (see the first, third, sixth rows of Table~\ref{table:1} and the first, second, and third rows of Table~\ref{table:2}). The second statement follows by inspecting the third and fourth rows of Table~\ref{table:1} and the first row of Table~\ref{table:2}. 
\end{proof}

The following lemma is in preparation for our next corollary to Theorem~\ref{thm:main*}(a).

\begin{lemma}\label{lem:orderproj}
Let $\theta$ be an involution of a thick spherical building, and suppose that the simplex $\sigma$ is mapped onto an opposite simplex. Then the induced automorphism $\theta_{\sigma}$ of $\Res(\sigma)$ is either the identity or it is an involution.
\end{lemma}

\begin{proof}
Let $\alpha$ be a simplex of $\mathrm{Res}(\sigma)$. If $\alpha^{\theta}=\proj_{\sigma^{\theta}}(\alpha)$ then $\alpha^{\theta_{\sigma}}=\alpha$ (because the projection maps $\mathrm{proj}_{\sigma}:\mathrm{Res}(\sigma^{\theta})\to\mathrm{Res}(\sigma)$ and $\mathrm{proj}_{\sigma^{\theta}}:\mathrm{Res}(\sigma)\to\mathrm{Res}(\sigma^{\theta})$ are mutually inverse bijections). If $\alpha^{\theta}=\mathrm{proj}_{\sigma^{\theta}}(\alpha)$ then $\alpha^{\theta_\sigma}=\alpha$. If $\alpha^{\theta}\neq \proj_{\sigma^{\theta}}(\alpha)$ then, since $\theta$ maps $\alpha^{\theta}$ onto $\alpha$, the projection $\mathrm{proj}_{\sigma}(\alpha^{\theta})$ is mapped onto $\mathrm{proj}_{\sigma^{\theta}}(\alpha)$. Thus $\theta_{\sigma}^2=1$.
\end{proof}

\begin{cor}\label{cor:involutions}
Every involution of a thick irreducible spherical building is capped. 
\end{cor}

\begin{proof} The result is of course true for large buildings of rank at least 3 (where all automorphisms are capped by \cite{PVM:17a}), and thus it remains to show that involutions of small buildings and of arbitrary generalised polygons are capped. Let us begin with the former.  We use the decorated opposition diagrams in Tables~\ref{table:1} and~\ref{table:2} to show that every uncapped automorphism has order strictly greater than~$2$. Consider type $\sA_n$, and let $\theta$ be uncapped. By Theorem~\ref{thm:main*}(a) there exists a non-domestic type $\{3,4,\ldots,n\}$ simplex~$\sigma$. Then $\theta_{\sigma}$ is a domestic duality of the Fano plane. However by \cite{PTM:15} the only domestic duality of the Fano plane is the unique exceptional domestic duality, and this has order~$8$. Thus, by Lemma~\ref{lem:orderproj} $\theta$ has order strictly greater than~$2$.

The arguments are similar for all other uncapped diagrams. The key fact is that in some residue one finds a domestic duality of the Fano plane. For example, in the first $\sE_6(2)$ diagram in Table~\ref{table:2} we have a non-domestic type $\{1,3,5,6\}$ simplex $\sigma$ (because, for example, the node $2$ is shaded), and $\theta_{\sigma}$ is a domestic duality of the Fano plane residue.

We now show that every involution of an arbitrary generalised $m$-gon, $m\geq 2$, is capped. Recall that a generalised $m$-gon~$\Delta$ is a bipartite graph with diameter $m$ and girth $2m$. A chamber is a pair of vertices connected by an edge. If $\{x,y\}$ is a chamber we write $x\sim y$ and call $x$ and $y$ adjacent. In particular, if $x\sim y$ then the vertices $x$ and $y$ have different types. Vertices $x$ and $y$ of $\Delta$ are opposite if and only if the distance between them is $m$, and this in turn is equivalent to the existence of a path $x=x_0\sim x_1\sim\cdots \sim x_{m}=y$ with $x_j\neq x_{j+2}$ for all $j=0,\ldots,m-2$. If the distance between vertices $x,y$ is $k<m$ then there is a unique geodesic from $x$ to $y$. In this case, writing $x=z_0\sim z_1\sim\cdots \sim z_k=y$ the vertex $z_1$ (respectively $z_{k-1}$) is the projection of $y$ onto~$x$ (respectively $x$ onto $y$). 
\smallskip

\noindent\textit{Claim 1:} Every involutary collineation of a thick generalised $2n$-gon $\Delta$, $n\geq 1$, is capped. 
\smallskip

\noindent\textit{Proof of Claim 1:} The case $n=1$ is trivial, and so suppose that $\theta$ is an uncapped involutary collineation of a generalised $2n$-gon with $n\geq 2$. Thus $\theta$ is domestic (on chambers), and maps at least one vertex of each type onto an opposite vertex. Let $\Delta'$ denote the fixed elements of $\theta$. Let $x_0$ be a type $1$ vertex mapped onto an opposite vertex $x_{2n}=x_0^{\theta}$, and consider any geodesic path $x_0\sim x_1\sim \cdots \sim x_{2n-1}\sim x_{2n}$. If $x_1^\theta\neq x_{2n-1}$ then the chamber $\{x_0,x_1\}$ is mapped onto an opposite chamber and $\theta$ is capped. Hence $x_1^\theta=x_{2n-1}$, and it follows that $x_i^\theta=x_{2n-i}$, for all $i\in\{0,1,2,\ldots,2n\}$. In particular $x_n^\theta=x_n$ is fixed. Consider another geodesic $x_0\sim y_1\sim\cdots \sim y_{2n-1}\sim x_{2n}$ with $y_1\neq x_1$. Then $y_n^{\theta}=y_n$. By considering the path from $x_n$ to $x_0$ to $y_n$ we see that $x_n$ and $y_n$ are opposite, and thus there is a pair of opposite vertices $x_n,y_n\in\Delta'$.

Similarly, by considering a type 2 vertex $x_0'$ that is mapped onto an opposite vertex we deduce the existence of a pair of opposite vertices $x_n',y_n'\in\Delta'$. Since the vertices $x_n',y_n'$ have different type to the vertices $x_n,y_n$ we conclude that for each type $j\in\{1,2\}$ there are pairs of opposite vertices of type~$j$ in $\Delta'$. It follows that $\Delta'$ is a sub-$2n$-gon (because the fixed structure of an collineation of a $2n$-gon is either empty, consists of pairwise opposite elements, is a tree of diameter at most $2n$, or is a sub-$2n$-gon, and the first three options are impossible from the above considerations). 

Now, the distance from $x_n'$ to $x_n$ is at most $2n-1$ (by types and diameter) and hence the unique geodesic from $x_n'$ to $x_n$ is fixed by~$\theta$. In particular the chamber $\{z,x_n\}$ is fixed, where $z\sim x_n$ is the projection of $x_n'$ onto~$x_n$. Note that $z\neq x_{n-1},x_{n+1}$ because $x_{n-1}^{\theta}=x_{n+1}$ is not fixed. We claim that every vertex $z_1\sim z$ is fixed. With $y_j$ as above, note that $z$ and $y_{n-1}$ are opposite (consider the path from $z$ to $x_0$ to $y_{n-1}$). Hence the distance from $z_1$ to $y_{n-1}$ is $2n-1$, and so there is a unique geodesic $z_1\sim z_2\sim\cdots z_{2n-1}=y_{n-1}$. If $z_1^{\theta}\neq z_1$ then $z_n$ and $z_n^{\theta}$ are opposite (consider the path from $z_n$ to $z_0$ to $z_n^{\theta}$). Similarly, since $y_{n-1}^{\theta}=y_{n+1}$ we have $y_{n-1}\neq y_{n-1}^{\theta}$ and so $z_{n+1}$ and $z_{n+1}^{\theta}$ are opposite. Hence the chamber $\{z_n,z_{n+1}\}$ is mapped onto an opposite chamber, a contradiction. 

It now follows from \cite[Proposition~1.8.1]{HVM:98} that the sub-$2n$-gon $\Delta'$ has the property that whenever $x\in \Delta'$ has the same type as $z$, then all neighbours of $x$ are fixed (and hence are in~$\Delta'$). But $x_n'$ has the same type as~$z$, contradicting the fact that the projection of $x_0'$ onto $x_n'$ is mapped onto the projection of $x_0'^\theta$ onto $x_n'$ and that these projections are distinct. %type as However repeating the entire argument with the roles of $x_0$ and $x_0'$ interchanged we see that all neighbours of  $\Delta'$ has this same property for both types of vertices. This forces $\Delta'=\Delta$ and so $\theta$ is the identity. 
This contradiction completes the proof of Claim~1.
\smallskip

\noindent\textit{Claim 2:} Every involutary duality of a thick generalised $(2n-1)$-gon $\Delta$, $n\geq 2$, is capped. 
\smallskip

\noindent\textit{Proof of Claim 2:} Let $\theta$ be a polarity of a generalised $(2n-1)$-gon and suppose that $\theta$ maps some element $x_0$ to an opposite element $x_{2n-1}$. Suppose that $\theta$ is not capped, i.e., $\theta$ does not map any chamber to an opposite chamber. Let $x_1\sim x_0$ be arbitrary. Consider the path $x_0\sim x_1\sim\cdots\sim x_{2n-1}$. In a similar way to the previous proof we deduce that $x_i^\theta=x_{2n-1-i}$ for all $i\in\{0,1,2,\ldots,2n-1\}$. Hence $x_n^\theta=x_{n-1}$. Consider a second path $x_0\sim y_1\sim\cdots \sim y_{2n-2}\sim x_{2n-1}$ with $y_1\neq x_1$. Then also $y_{n-1}^\theta=y_n$. Let $z_0\sim x_n$ be arbitrary but distinct from $x_{n-1}$ and $x_{n+1}$ (using thickness). There is a unique path $z_0\sim z_1\sim\cdots \sim z_{2n-2}=y_{n-1}$ from $z_0$ to $y_{n-1}$. By considering the path $z_{n-2}\sim \cdots \sim z_0\sim x_n\sim x_n^{\theta}\sim z_0^{\theta}\sim\cdots \sim z_{n-2}^{\theta}$ we see that $z_{n-2}$ is mapped onto an opposite vertex. Similarly, since $y_{n-1}^{\theta}=y_n$ we see that $z_{n-1}$ is mapped onto an opposite vertex (consider the path $z_{n-1}\sim\cdots \sim y_{n-1}\sim y_{n-1}^{\theta}\sim \cdots \sim z_{n-1}^{\theta}$). Hence the chamber $\{z_{n-2},z_{n-1}\}$ is mapped onto an opposite chamber, a contradiction. This completes the proof of Claim~$2$.
\smallskip

Finally, we note that no duality of a thick generalised $2n$-gon is domestic and no collineation of a thick generalised $(2n-1)$-gon is domestic (see \cite[Lemmas~3.1 and~3.2]{PTM:15}), completing the proof of the corollary.
\end{proof}

Corollary~\ref{cor:involutions} shows that every uncapped automorphism has order at least~$3$. Since every known example of an uncapped automorphism has order at least~$4$ (see the examples in Sections~\ref{sec:classical} and~\ref{sec:exceptional}, and also the rank~$2$ classification in~\cite{PTM:15}) we are led to make the following conjecture.

\begin{conjecture}
If $\theta$ is an automorphism of a thick irreducible spherical building, and if $\theta$ has order $3$, then $\theta$ is capped. 
\end{conjecture}

Note that if we remove the shading from the diagrams in Tables~\ref{table:1} and~\ref{table:2} then the diagrams we obtain are contained in \cite[Tables 1--5]{PVM:17a}. Thus Theorem~\ref{thm:main*}(a) has the following immediate corollary.

\begin{cor}
The (undecorated) opposition diagram of any automorphism of a thick irreducible spherical building is contained in  \emph{\cite[Tables 1--5]{PVM:17a}}.
\end{cor}

We now use Theorem~\ref{thm:main*}(a) to determine the partially ordered set $\mathcal{T}(\theta)$ for all automorphisms~$\theta$. We first note that, by the proposition below, it is sufficient to determine the maximal elements of $\mathcal{T}(\theta)$.

\begin{prop}\label{prop:redd}
Let $\mathcal{M}(\theta)$ be the set of maximal elements of $\mathcal{T}(\theta)$. Then 
$$
\mathcal{T}(\theta)=\{J\subseteq S\mid J^{\pi_{\theta} w_0}=J\text{ and }J\subseteq M\text{ for some $M\in\mathcal{M}(\theta)$}\}.
$$
\end{prop}

\begin{proof}
This follows immediately from the facts that if $\sigma$ is a non-domestic type $K$ simplex then (i) $K$ is preserved by $w_0\circ\pi_{\theta}$, and (ii) if $J\subseteq K$ is preserved under $w_0\circ \pi_{\theta}$ then the type $J$ subsimplex of $\sigma$ is also non-domestic (see \cite[Lemma~1.3]{PVM:17a}). 
\end{proof}

Thus it remains to compute the set $\mathcal{M}(\theta)$ of maximal elements of $\mathcal{T}(\theta)$. We do this in the corollary below. Recall that if $\theta$ is uncapped then the decorated opposition diagram of $\theta$ is $(\Gamma,\Type(\theta),K_{\theta},\pi_{\theta})$ where, in particular, $K_{\theta}$ is the set of shaded nodes. 

\begin{cor}\label{cor:maximal}
Let $\theta$ be an automorphism of a spherical building~$\Delta$. 
\begin{compactenum}[$(a)$]
\item If $\theta$ is capped then $\mathcal{M}(\theta)=\{\Type(\theta)\}$.
\item If $\theta$ is uncapped then $\mathcal{M}(\theta)=\{\Type(\theta)\backslash\{k\}\mid k\in K_{\theta}\}$.\end{compactenum}
\end{cor}

\begin{proof}
The first statement is obvious, so consider the second statement. Let $(\Gamma,J,K,\pi)$ be the decorated opposition diagram, and so $J=\Type(\theta)$. If $J=K$ then there are non-domestic simplices of each type $\Type(\theta)\backslash\{k\}$ with $k\in J$, and these are clearly the maximal types mapped to opposite (otherwise $\theta$ is capped). Suppose now that $J\backslash K$ consists of a single minimal $w_0\circ\pi$ invariant subset~$J'$ (thus $J'$ is either a singleton, or $J'$ consists of a pair, as in the second $\sD_{2n}(2)$ diagram in Table~\ref{table:1}). In this case the only $w_0\circ \pi$ stable strict subset of $J$ that is not contained in an element of $\{J\backslash\{k\}\mid k\in K\}$ is $J\backslash J'$, and since $J'$ is not shaded all simplicies of this type are domestic. Hence the result in this case. 

By Theorem~\ref{thm:main*}(a) the only remaining cases are the $6$ diagrams where $J\backslash K$ consists of precisely $2$ minimal $w_0\circ\pi$ invariant sets. Specifically, these examples are the $\sE_6(2)$ collineation diagram, the first $\sE_7(2)$ and $\sE_8(2)$ diagrams, the first two $\sF_4(2)$ diagrams (these are dual to one another), and the $\sF_4(2,4)$ diagram. In these cases the result is implied by the following claim.
\smallskip

\noindent\textit{Claim:} Suppose that the decorated opposition diagram of $\theta$ is one of the $6$ diagrams listed above. Then $\theta$ is $\{i,j\}$-domestic where $i$ and $j$ are the two shaded nodes.
\smallskip

\noindent\textit{Proof of Claim:} Consider the $\sE_6$ diagram. If there is a non-domestic type $\{2,4\}$ simplex then with $v$ the type $4$ vertex of this simplex the map $\theta_v$ acts on the $\sA_2\times\sA_2$ component of the residue swapping the components (by Proposition~\ref{prop:typemap}). It follows that $\theta$ is not domestic, a contradiction. Similar arguments apply for $\sE_7$ and $\sE_8$, using an $\sA_5$ and $\sE_6$ residue respectively. For the first $\sF_4(2)$ diagram, suppose there is a non-domestic type $\{1,2\}$ simplex $\sigma$. Then $\theta_{\sigma}$ is a domestic duality of $\sA_2(2)$, and hence is the exceptional domestic duality of the Fano plane. It follows that there is non-domestic type $\{1,2,3\}$ simplex, contradicting the node $4$ being unshaded. A dual argument applies to the second $\sF_4(2)$ diagram. The $\sF_4(2,4)$ diagram is similar. Hence the proof of the claim is complete, and the corollary follows.
\end{proof}

\begin{example}
Suppose that $\theta$ has the $\sE_6(2)$ collineation diagram in Table~\ref{table:2}. Then the partially ordered set $\mathcal{T}(\theta)$ is (using Proposition~\ref{prop:redd} and Corollary~\ref{cor:maximal}):
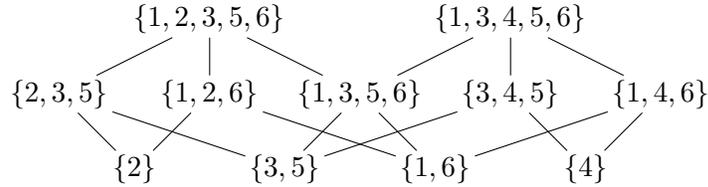
\begin{center}
\begin{tikzpicture}[scale=1,baseline=-0.5ex]
\node at (0,0.3) {};
\node [inner sep=0.8pt,outer sep=0.8pt] at (-3,-1) (1) {$\{2\}$};
\node [inner sep=0.8pt,outer sep=0.8pt] at (-1,-1) (2) {$\{3,5\}$};
\node [inner sep=0.8pt,outer sep=0.8pt] at (1,-1) (3) {$\{1,6\}$};
\node [inner sep=0.8pt,outer sep=0.8pt] at (3,-1) (4) {$\{4\}$};
\node [inner sep=0.8pt,outer sep=0.8pt] at (-4,0) (5) {$\{2,3,5\}$};
\node [inner sep=0.8pt,outer sep=0.8pt] at (-2,0) (6) {$\{1,2,6\}$};
\node [inner sep=0.8pt,outer sep=0.8pt] at (0,0) (7) {$\{1,3,5,6\}$};
\node [inner sep=0.8pt,outer sep=0.8pt] at (2,0) (8) {$\{3,4,5\}$};
\node [inner sep=0.8pt,outer sep=0.8pt] at (4,0) (9) {$\{1,4,6\}$};
\node [inner sep=0.8pt,outer sep=0.8pt] at (-2,1) (10) {$\{1,2,3,5,6\}$};
\node [inner sep=0.8pt,outer sep=0.8pt] at (2,1) (11) {$\{1,3,4,5,6\}$};
\draw (1)--(5);
\draw (1)--(6);
\draw (2)--(5);
\draw (2)--(7);
\draw (2)--(8);
\draw (3)--(6);
\draw (3)--(7);
\draw (3)--(9);
\draw (4)--(8);
\draw (4)--(9);
\draw (5)--(10);
\draw (6)--(10);
\draw (7)--(10);
\draw (7)--(11);
\draw (8)--(11);
\draw (9)--(11);
\end{tikzpicture}
\vspace{-0.3cm}
\end{center}
\medskip
\end{example}

As a final application we will compute the displacement of an arbitrary automorphism~$\theta$ in Corollary~\ref{cor:disp} below. Recall that, by definition, $\disp(\theta)=\max\{d(C,C^{\theta})\mid C\in\mathcal{C}\}$, where $\mathcal{C}$ is the set of chambers of $\Delta$, and $d(C,D)=\ell(\delta(C,D))$ for chambers $C,D\in\mathcal{C}$.

\begin{prop}\label{prop:disp} Let $\theta$ be any automorphism of a thick irreducible spherical building of type~$(W,S)$. Then
$$
\disp(\theta)=\diam(W)-\min \{\mathrm{diam}(W_{S\backslash J})\mid J\in\mathcal{M}(\theta)\}.
$$ 
\end{prop}

\begin{proof}
Let $N=\min \{\mathrm{diam}(W_{S\backslash J})\mid J\in\mathcal{M}(\theta)\}$. We note first that
\begin{align}\label{eq:ineq}
N=\min\{\diam(W_{S\backslash J})\mid \text{there exists a type $J$ simplex in $\Opp(\theta)$}\}
\end{align}
because the minimum is obviously attained at a maximal element of $\mathcal{T}(\theta)$.

Let $J\subseteq \Type(\theta)$ be any subset for which there exists a non-domestic type $J$ simplex. Then for all chambers $C$ containing this simplex we have $\delta(C,C^{\theta})\in W_{S\backslash J}w_0$ (see \cite[Lemma~2.5]{PVM:17a}) and thus 
$$
\disp(\theta)\geq \ell(\delta(C,C^{\theta}))\geq \ell(w_0)-\ell(w_{S\backslash J})=\mathrm{diam}(W)-\mathrm{diam}(W_{S\backslash J}).
$$
Since this inequality holds for all $J$ such that there exists a type $J$ simplex in $\Opp(\theta)$ the formula (\ref{eq:ineq}) gives $\disp(\theta)\geq \diam(W)-N$. 

On the other hand, let $C$ be any chamber with $\ell(\delta(C,C^{\theta}))$ maximal. By the arguments of \cite[Lemma~2.4 and Theorem~4.2]{AB:09} we have $\delta(C,C^{\theta})=w_{I}w_0$ for some $I\subseteq S$ with $I^{\pi_{\theta}}=I^{w_0}$. Hence the type $J=S\backslash I$ simplex of $C$ is mapped onto an opposite simplex. Thus
$$
\mathrm{disp}(\theta)=\ell(\delta(C,C^{\theta}))=\ell(w_0)-\ell(w_{S\backslash J})=\diam(W)-\diam(W_{S\backslash J})\leq \diam(W)-N,
$$
hence the result. 
\end{proof}

\begin{cor}\label{cor:disp}
Let $\theta$ be an automorphism of a thick irreducible spherical building and let $J=\Type(\theta)$. Then
$$
\disp(\theta)=\begin{cases}
\mathrm{diam}(W)-\mathrm{diam}(W_{S\backslash J})&\text{if $\theta$ is capped}\\
\mathrm{diam}(W)-\mathrm{diam}(W_{S\backslash J})-1&\text{if $\theta$ is uncapped.}
\end{cases}
$$ 
In particular, if $\theta$ is exceptional domestic then $\disp(\theta)=\diam(\Delta)-1$. 
\end{cor}

\begin{proof}
The case of capped automorphisms is \cite[Theorem~5]{PVM:17a}. In the case of an uncapped automorphism we note that by Corollary~\ref{cor:maximal} the maximal elements of $\mathcal{T}(\theta)$ are of the form $\Type(\theta)\backslash\{j\}$ for some $j\in\Type(\theta)$, and then the result follows from Proposition~\ref{prop:disp}.
\end{proof}

\begin{remark}\label{rem:disp}
Corollary~\ref{cor:disp} shows that the set of possible displacements is extremely restricted. For example, consider an $\sE_8$ building~$\Delta$, where a priori there are $\ell(w_0)=120$ potential displacements. However, by Corollary~\ref{cor:disp}, \cite[Theorem~3]{PVM:17a}, and Theorem~\ref{thm:main*}(a) the only possible displacements are:
\begin{align*}
0&=\diam(\sE_8)-\diam(\sE_8) &&\text{for the trivial (hence capped) automorphism}\\
57&=\diam(\sE_8)-\diam(\sE_7)&&\text{for capped automorphisms with $\Type(\theta)=\{8\}$}\\
90&=\diam(\sE_8)-\diam(\sD_6)&&\text{for capped automorphisms with $\Type(\theta)=\{1,8\}$}\\
107&=\diam(\sE_8)-\diam(\sD_4)-1&&\text{for uncapped automorphisms with $\Type(\theta)=\{1,6,7,8\}$}\\
108&= \diam(\sE_8)-\diam(\sD_4)&&\text{for capped automorphisms with $\Type(\theta)=\{1,6,7,8\}$}\\
119&=\diam(\sE_8)-1&&\text{for uncapped automorphisms with $\Type(\theta)=S$}\\
120&= \diam(\sE_8)&&\text{for non-domestic (hence capped) automorphisms}.
\end{align*}
In particular, note that for $\sE_8$ buildings the displacement determines the (decorated) opposition diagram of the automorphism. %For example, every automorphism with displacement $107$ is necessarily uncapped, with $\Type(\theta)=\{1,6,7,8\}$. 
This phenomenon is not true for all types; for example in $\sB_7(\mathbb{F})$ displacement $45$ is obtained by both capped automorphisms with $\Type(\theta)=\{1,2,3,4,5\}$ and capped automorphisms with $\Type(\theta)=\{2,4,6\}$.
\end{remark}

\section{Uncapped automorphisms for classical types}\label{sec:classical}

In this section we prove Theorem~\ref{thm:main*}(b) for classical types. Thus our aim is to construct uncapped automorphisms with each of the diagrams listed in Tables~\ref{table:1} and~\ref{table:2} for the buildings $\sA_n(2)$, $\sB_n(2)$, $\sB_n(2,4)$, and $\sD_n(2)$.

\subsection{The buildings $\sA_n(2)$}

In this section we work with the concrete model $\sA_n(2)=\mathsf{PG}(n,\FF_2)$ for the small building of type $\sA_n$. Thus an $i$-space of $\sA_n(2)$ means a subspace of $\mathbb{F}_2^n$ of  (projective) dimension~$i$, and this corresponds to a type $i+1$ vertex of the building. Let $\theta$ be a duality of $\sA_n(2)$. Recall that a point $p$ of $\sA_n(2)$ is called \textit{absolute} with respect to $\theta$ if $p\in p^{\theta}$ (that is, $p$ is not mapped to an opposite hyperplane). Dually, a hyperplane $\pi$ is absolute if $\pi^{\theta}\in \pi$ (that is, $\pi$ is not mapped to an opposite point).

\begin{lemma}\label{lem:dimension}
Let $\theta$ be a duality of a projective space. Suppose that $U$ is an $m$-space consisting of absolute points of $\theta$, and let $k=\dim(U\cap U^{\theta})$. Then $m-k$ is even. 
\end{lemma}

\begin{proof}
The hyperplanes through $\langle U^{\theta},U\rangle$ form a dual space of (projective) dimension~$k$, and the inverse image is a $k$-space contained in $U$. Choose a complementary $(m-k-1)$-space $H$ in $U$, and so $H$ intersects neither $U^{\theta}$ nor $U^{\theta^{-1}}$. Then for each $x\in H$ we have that $x^{\theta}\cap H$ is a hyperplane of $H$ through $x$, and hence is absolute. Thus $\theta$ is a symplectic polarity on $H$, and so $m-k$ is even (see Lemma~\ref{lem:An}). 
\end{proof}

\begin{thm}\label{thm:existenceAn(2)}
For each $n\geq 2$ there exists a unique duality $\theta$ of $\sA_n(2)$ (up to conjugation) with the property that the set of absolute points of $\theta$ is the union of two distinct hyperplanes. This duality is strongly exceptional domestic, with order $8$ if $n$ is even and $4$ if $n$ is odd. 
\end{thm}

\begin{proof}
We first demonstrate the existence of a duality whose absolute points form the union of two hyperplanes. Let $J_1$, $J_2$, and $J_3$ be the matrices
\begin{align*}
J_1&=\begin{bmatrix}
0&1\\
1&0\end{bmatrix},& J_2&=\begin{bmatrix}
0&0&1\\
1&0&1\\
1&1&0
\end{bmatrix}, & J_3&=\begin{bmatrix}
0&0&1&1\\
1&0&0&1\\
1&0&0&0\\
1&1&0&0
\end{bmatrix}
\end{align*}
and let $A$ be the $(n+1)\times (n+1)$ matrix in block diagonal form 
$$
A=\mathrm{diag}(J,J_1,J_1,\ldots,J_1)\quad\text{with $J=J_2$ for even $n$ and $J=J_3$ for odd $n$}.
$$
Let $\theta$ be the duality of $\sA_n(2)$ with matrix $A$. That is, $X^{\theta}=(AX)^{\perp}$ where $X$ is written as a column vector. Then $X$ is absolute if and only if $X\in (AX)^{\perp}$, and hence by direct calculation $X$ is absolute if and only if $X_0X_1=0$. The matrix for the collineation $\theta^2$ is given by $A^{-t}A$, and it follows by calculation that $\theta$ has order $8$ if $n$ is even, and order $4$ if $n$ is odd.

We now prove that there is at most one duality $\theta$ up to conjugation with the given property, and that such a duality is necessarily strongly exceptional domestic. We proceed by induction on $n$, the case $n=2$  being contained in \cite{PTM:15}. 

So let $\theta$ be a duality of $\sA_n(2)$ such that $\alpha_1\cup\alpha_2$ is the set of absolute points for $\theta$ with $\alpha_1\neq\alpha_2$ two hyperplanes of $\sA_n(2)$. Let $\beta$ be the hyperplane containing $\alpha_1\cap\alpha_2$ and different from both $\alpha_1$ and $\alpha_2$. Note that $\alpha_1\cup \alpha_2\cup\beta$ is the entire point set. Let $p_i=\alpha_i^\theta$, $i=1,2$ and $q=\beta^\theta$; then $L=\{p_1,p_2,q\}$ is a line. 

Note that $q$ is absolute (for if $q\in\beta$ we have $q^{\theta}\ni\beta^{\theta}=q$). Thus $q\in \alpha_1\cup \alpha_2$. In fact we claim that $q\in\alpha_1\cap \alpha_2$. For if not we have $\beta^{\theta}=q\notin\beta$ and so $\beta$ is not absolute, contradicting the fact that $\beta=q^{\theta^{-1}}$ is absolute (since $q$ is absolute).

Since $L=\{p_1,p_2,q\}$ is a line and $q\in \alpha_1\cap \alpha_2$ we either have $p_1,p_2\in\beta\backslash(\alpha_1\cup\alpha_2)$ or $p_1,p_2\in\alpha_1\cup\alpha_2$. We treat these two cases below. Before doing this, we observe that in the first case $n$ is necessarily even, and in the second case $n$ is necessarily odd. To see this, note that if $p_1,p_2\in\beta\backslash(\alpha_1\cup\alpha_2)$ then the point $p_1$ is non-absolute and the mapping $\rho_1:z\mapsto z^\theta\cap\alpha_1$, $z\in\alpha_1$, is a duality on $\alpha_1$ every point of which is absolute, forcing $n$ to be even (see Lemma~\ref{lem:An}). On the other hand, if $p_1,p_2\in\alpha_1\cup\alpha_2$ then we have $(\alpha_1\cap\alpha_2)^\theta=\<p_1,p_2\>\subseteq \alpha_1\cap\alpha_2$ and so Lemma~\ref{lem:dimension} implies $(n-2)-1=n-3$ is even, and so $n$ is odd. We also observe that since $\alpha_1$ and $\alpha_2$ are the only two hyperplanes all of whose points are absolute, every even power of $\theta$ preserves the set $\{\alpha_1,\alpha_2\}$, and hence also the set $\{p_1,p_2\}$. It follows that $p_i^\theta\in\{\alpha_1,\alpha_2\}$ for $i=1,2$. 
\smallskip

\noindent \textit{Case 1:} $p_1,p_2\in\beta\backslash(\alpha_1\cup\alpha_2)$. As noted above $n$ is even, and so we may assume $n\geq 4$. Let $\sigma=\{x,\xi\}$ be any non-domestic (point-hyperplane)-flag for $\theta$ (that is, a non-domestic type $\{1,n\}$-simplex of the building). We note that such simplices exist, and indeed they obviously all arise as follows: Since the absolute hyperplanes for $\theta$ are precisely the hyperplanes through one of the points $p_1$ or $p_2$,  if we select any point $x\in\beta\setminus(\alpha_1\cup\alpha_2)$ and any hyperplane $\xi$ through $x$ not containing $p_1$ or $p_2$, then $\sigma=\{x,\xi\}$ is non-domestic.

We claim that the mapping $\theta_\sigma:z\mapsto z^\theta\cap \xi\cap x^\theta$ for $z\in\xi\cap x^\theta$ has exactly two hyperplanes consisting entirely of absolute points. Note that $q\in\xi$ and also $q\in x^\theta$. Note also that, since $p_i^\theta$ contains the absolute point $q_i:=\<p_i,x\>\cap(\alpha_1\cap\alpha_2)$, also $x^\theta$ contains $q_i$, $i=1,2$. Since $\xi$ does not contain $p_i$, but it does contain $x$, it does not contain $q_i$, $i=1,2$. Consequently $x^\theta\cap\alpha_1\cap\alpha_2$ is not contained in $\xi$ and the claim follows. 

Thus for every non-domestic (point-hyperplane)-pair $\sigma=\{x,\xi\}$ the induced duality $\theta_{\sigma}$ on the $\sA_{n-2}(2)$ residue has precisely two hyperplanes of absolute points. Since $n-2$ is even this duality again satisfies the condition of Case 1, and so by induction $\theta$ is domestic. Since $\theta$ has non-domestic points necessarily $\theta$ is strongly exceptional domestic by Theorem~\ref{thm:Asmall}.

We now show that $\theta$ is unique, up to a projectivity (and under the assumptions of Case~1). Let $\rho_1$ be the symplectic polarity on $\alpha_1$ introduced in the paragraph before Case 1. Noting that $q^{\rho_1}=\alpha_1\cap\alpha_2$, we see that the data $\alpha_1,\alpha_2$ and $\rho_1$ are projectively unique. This determines $q$. All choices of $p_1$ outside $\alpha_1\cup\alpha_2$ are projectively equivalent, and then $p_2$ is the third point on the line determined by $p_1$ and $q$. We then know the image of an arbitrary point $x_1$ of $\alpha_1\setminus\alpha_2$, as $x_1^\theta=\<x^{\rho_1}, p_1\>$. This determines the images of all points of $\alpha_1$. Since $p_1^\theta=\alpha_1$, we know the images of a basis, which suffices to determine the whole duality. 
\smallskip

\noindent\textit{Case 2:} $p_1,p_2\in\alpha_1\cup\alpha_2$. As noted above, $n$ is odd. Take an arbitrary point $z\in\beta\setminus(\alpha_1\cup\alpha_2)$ and set $H:= z^\theta$. Then $\varphi:x\mapsto x^\theta\cap H$ is a duality in the $(n-1)$-dimensional projective space $H$ such that its absolute points form two hyperplanes $H\cap\alpha_i$, $i=1,2$. Hence by the previous case is domestic, and since $z$ was arbitrary amongst the non-domestic points for $\theta$ we conclude that~$\theta$ is domestic. Thus by Theorem~\ref{thm:Asmall} $\theta$ is strongly exceptional domestic. 

It remains to show that $\theta$ is unique up to conjugation with a projectivity. Let $D_i=H\cap \alpha_i$, $i=1,2$. Set $\{i,j\}=\{1,2\}$ and $D_i^{\varphi^{-1}}=p_i'$. Then $\{q,p_1',p_2'\}$ is a line in $H\cap\beta$ (since ${p_i'}^{\varphi}=D_i$ it suffices to see that $q^{\varphi}=\beta\cap H$, and this follows from the definition of $\varphi$ as $\beta=q^{\theta}$). It also follows that $D_i^{\theta^{-1}}=\<p_i',z\>$. Since $D_i\subseteq\alpha_i$, we conclude $\alpha_i^{\theta^{-1}}\in\< p_i',z\>$. But $\alpha_i^{\theta^{-1}}\in\{p_1,p_2\}$. We claim that $\alpha_i^{\theta^{-1}}=p_i$. Suppose not. Then $\alpha_i^{\theta^{-1}}=p_j$. Now from $z^\theta=H$ and $p_i^\theta=\alpha_j$ follows that $t_i^\theta=\<D_j,z\>$, with $\{t_i,p_i,z\}$ a line. But $p_j'^\theta$ is a hyperplane through $D_j$ distinct from $\alpha_j$ and $H$ (as $p_j\in H$ and is not absolute); hence $p_j'^\theta=\<D_j,z\>$ and so $t_i=p'_j$. Now $p_j'^{\theta^{-1}}=\<D_i,z\>$ and $p_i^{\theta^{-1}}=\alpha_i$. It follows that $z^{\theta^{-1}}=H$. Hence $z^{\theta^2}=z$, for all $z\in\beta\setminus (\alpha_1\cup\alpha_2)$. It follows that $p_i^{\theta^2}=p_i$, contradicting $p_i^{\theta^2}=\alpha_j^\theta=p_j$. Our claim follows. 

But now, just like in the proof of our previous claim, we have that $\{p_i,p_i',z\}$ is a line and $p_i'^\theta=\<D_j,z\>$. It follows that $p_i^{\theta^2}=p_j$ and so $z^{\theta^2}=z'$, with $\{z,z',q\}$ a line. 

Now, $\alpha_1,\alpha_2,H,z$ and $\varphi$ are unique up to conjugation with a projectivity. But then, given $z^\theta=H$, the duality $\theta$ is completely determined, since $q$ is determined and hence also $z'$ (with the  above notation). This determines the image $x^\theta$ of an arbitrary point in $H$ as $x^\theta=\<x^\varphi,z'\>$. Furthermore, we also have $z^\theta=H$, and so $\theta$ is determined.
%
%Let $H'$ be the hyperplane through $H\cap\beta$ distinct from $H$ and $\beta$.  Let $x\in\alpha_1\setminus\alpha_2$ be arbitrary. Then $x^{\theta^2}=(\<x^\varphi,z'\>^\theta$ is the intersection of the line joining $z'$ with the intersection of $\<x,p_2'\>\cap\alpha_2$. This point lies in the plane $\<p_1,x,z'\>$ and is different from $x$, hence it is the point $x'$ such that $\{x,x',p_1\}$ is a line. It now follows that $\theta^4$ fixes all points of $\alpha_1\cup\beta$, hence it is the identity and so the order of $\theta$ is 4.
\end{proof}

\subsection{The buildings $\sB_n(2)$, $\sB_n(2,4)$, and $\sD_n(2)$}

It will be more convenient for us to regard $\sB_n(2)\cong \sC_n(2)$ as a symplectic polar space. We begin by recalling the standard models of the $\sC_n(2)$, $\sD_n(2)$, and $\sB_{n-1}(2,4)$ buildings in the ambient projective space $\mathsf{PG}(2n-1,2)$. Let $V=\FF_2^{2n}$, and let $(\cdot,\cdot)$ be the (symplectic and symmetric) bilinear form on $V=\FF_2^{2n}$ given by
\begin{align}\label{eq:form}
(X,Y)=X_1Y_{2n}+X_2Y_{2n-1}+\cdots+X_{2n}Y_{1}.
\end{align}
The points of the polar space $\mathsf{C}_n(2)$ are the $0$-spaces of $\mathsf{PG}(2n-1,2)$, and points $p=\langle X\rangle$ and $q=\langle Y\rangle$ are collinear (including the case $p=q$) if and only if $(X,Y)=0$. A subspace $U$ of $V$ is \textit{totally isotropic} if $(X,Y)=0$ for all $X,Y\in U$. The totally isotropic subspaces of maximal dimension have projective dimension~$n-1$, and for each $0\leq k\leq n-1$ the $k$-spaces of the polar space $\mathsf{C}_n(2)$ are the totally isotropic subspaces of~$V$ with projective dimension~$k$. To obtain the building of $\mathsf{C}_n(2)$ as a labelled simplicial complex one takes the totally isotropic $(k-1)$-spaces to be the type $k$ vertices of the building for $1\leq k\leq n$, with incidence of vertices given by symmetrised containment of the corresponding spaces. The full collineation group of $\mathsf{C}_n(2)$ is the symplectic group $\mathsf{Sp}_{2n}(2)$ consisting of all matrices $g\in\mathsf{GL}_{2n}(2)$ satisfying $g^TJg=J$, where $J$ is the matrix of the symplectic form $(\cdot,\cdot)$ (see \cite[Corollary~5.9]{Tit:74}).

Let $F^{+}$ and $F^-$ be quadratic forms on $V$ with Witt indices $n$ and $n-1$ respectively. We will fix the specific choices
\begin{align*}
F^+(X)&=X_1X_{2n}+X_2X_{2n-1}+\cdots+X_{n}X_{n+1} \\
F^-(X)&=X_1X_{2n}+X_2X_{2n-1}+\cdots+X_{n}X_{n+1}+X_{n}^2+X_{n+1}^2. 
\end{align*}
For $\epsilon\in\{-,+\}$, a subspace $U\subseteq V$ is \textit{singular} with respect to $F^{\epsilon}$ if $F^{\epsilon}(X)=0$ for all $X\in U$. 
The maximal dimensional singular subspaces of $V$ with respect to $F^{\epsilon}$ have vector space dimension equal to the Witt index of $F^{\epsilon}$. The points of $\mathsf{D}_n(2)$, respectively the polar space $\mathsf{B}_{n-1}(2,4)$, are those points of $\mathsf{PG}(2n-1,2)$ that are singular with respect to $F^+$, respectively $F^-$. In both cases points $p=\langle X\rangle$ and $q=\langle Y\rangle$ are collinear (including the case $p=q$) if and only if $(X,Y)=0$, where $(\cdot,\cdot)$ is as in~(\ref{eq:form}).

Let $\mathsf{GO}^{\epsilon}_{2n}(2)$ be the group of all matrices of $\mathsf{GL}_{2n}(2)$ preserving the quadratic form~$F^{\epsilon}$, and let $\mathsf{O}_{2n}^{\epsilon}(2)$ be the corresponding index~$2$ simple subgroup of $\mathsf{GO}^{\epsilon}_{2n}(2)$ (c.f. \cite[\S2.4]{ATLAS}). Since $\mathsf{GO}_{2n}^{\epsilon}(2)$ preserves collinearity, the group $\mathsf{GO}_{2n}^{+}(2)$ acts on $\mathsf{D}_n(2)$ and the group $\mathsf{GO}_{2n}^-(2)$ acts on $\mathsf{B}_{n-1}(2,4)$. In fact the group $\mathsf{GO}_{2n}^-(2)$ is the full automorphism group of $\mathsf{B}_{n-1}(2,4)$ (see \cite{Tit:74}). In the case of $\mathsf{D}_n(2)$ the maximal singular subspaces are partitioned into two sets of equal cardinality by the action of $\mathsf{O}_{2n}^+(2)$, and an automorphism $\theta$ of $\mathsf{D}_n(2)$ mapping points to points is called a \textit{collineation} if this partition of maximal singular subspaces is preserved by $\theta$, and a \textit{duality} otherwise. Then $\mathsf{O}^+_{2n}(2)$ is the group of all collineations of $\mathsf{D}_n(2)$, and $\mathsf{GO}_{2n}^+(2)\backslash\mathsf{O}_{2n}^+(2)$ is the set of all dualities of $\mathsf{D}_n(2)$ (see \cite{Tit:74}). 

To obtain the building of $\mathsf{B}_{n-1}(2,4)$ as a labelled simplicial complex one takes the singular $(k-1)$-spaces to be the type $k$ vertices of the building for $1\leq k\leq n-1$, with incidence of vertices given by symmetrised containment of the corresponding spaces. The situation for $\mathsf{D}_n(2)$ is slightly different: For $1\leq k\leq n-2$ the singular $(k-1)$-spaces are taken to be the type~$k$ vertices of the building, and the singular $(n-1)$-spaces in one part of the partition mentioned above are taken to be the type $n-1$ vertices of the building, and those in the other part of the partition are taken to be the type $n$ vertices of the building. A type $n-1$ vertex is declared to be incident with a type $n$ vertex if the corresponding $(n-1)$-spaces meet in an $(n-2)$-space. For all other types incidence is given by symmetrised containment of the corresponding spaces. 

Note the index shifts that occur (for example an $\{n\}$-domestic collineation of a $\mathsf{C}_n(2)$ building is a collineation that is domestic on the totally isotropic $(n-1)$-spaces). A point $p$ of a polar space is an \textit{absolute point} of an automorphism $\theta$ if $p^{\theta}$ is collinear with~$p$ (including $p^{\theta}=p$). 

\begin{lemma}\label{lem:Cn(2)} 
Let $\theta$ be a collineation of $\mathsf{C}_n(2)$. 
\begin{compactenum}[$(a)$]
\item If $\theta$ fixes a subspace of $\mathsf{PG}(2n-1,2)$ of projective dimension $k\geq n$ then $\theta$ is $\{j\}$-domestic for each $2n-k\leq j\leq n$.
\item If the set of absolute points of $\theta$ strictly contains the union of two distinct hyperplanes of $\mathsf{PG}(2n-1,2)$ then $\theta$ is $\{1\}$-domestic. 
\end{compactenum}
\end{lemma}

\begin{proof} (a) By considering dimensions, each $(j-1)$-space of $\mathsf{PG}(2n-1,2)$ with $j\geq 2n-k$ intersects the subspace of fixed points. In particular, no totally isotropic $(j-1)$-space is mapped onto an opposite and so $\theta$ is $\{j\}$-domestic for all $2n-k\leq j\leq n$. 

(b) A point $X$ is an absolute point of $\theta\in\mathsf{Sp}_4(2)$ if and only if $(X,\theta X)=X^TJ\theta X=0$, where $J$ is the matrix of the symplectic form $(\cdot,\cdot)$. Thus the set of absolute points of $\theta$ is a quadric, and so if it strictly contains the union of two distinct hyperplanes then all points are absolute. 
\end{proof}

In the following proofs we use the standard notations $p\perp q$ if points $p$ and $q$ are collinear (including the case $p=q$), and $p^{\perp}$ for the set of all points collinear to~$p$.

\begin{lemma}\label{lem:fixed-subspace}
Let $\Delta=\sC_n(2)$ with $n\geq 2$ and let $\theta$ be a collineation. 
\begin{compactenum}[$(a)$]
\item If the fixed points of $\theta$ form a $(2n-3)$-space $W$, then the absolute points form a subspace containing $W$. 
\item If the fixed points of $\theta$ form a $(2n-2)$-space $W$, then every absolute point is fixed. \end{compactenum} 
\end{lemma}

\begin{proof}
(a) Let $p$ be a point not contained in $W$ and suppose $p$ is absolute. Let $q\in\<W,p\>\setminus W$. We claim that $q$ is absolute. Indeed, let $r:=\<p,q\>\cap W$. If $p\perp q$, then the plane $\pi=\<p,q,p^\theta\>$ contains the triangle $\{p,p^\theta,r\}$ of points collinear in $\sC_n(2)$ and so $
q\perp q^\theta$, as both points belong to~$\pi$.  If $p\notin q^\perp$, then $\pi$ contains the line $\<p,p^\perp\>$, which belongs to $\sC_n(2)$, but also contains the line $\<p,r\>$, which does not belong to $\sC_n(2)$. Also $\<p^\theta, r\>$ does not belong to $\sC_n(2)$, and it follows that the line $\<r,s\>$, where $\{p,p^\theta,s\}$ is the line of $\sC_n(2)$ through $p$ and $p^\theta$, belongs to $\sC_n(2)$. Hence also the line $\{s,q,q^\theta\}$ belongs to $\sC_n(2)$, which proves our claim. 

So, if there are no absolute points besides those in $W$, then $(i)$ holds. If some absolute point $p\notin W$ exists, then there are three possibilities. Either exactly one hyperplane through $W$ consists of absolute points (and then $(i)$ holds), or all three hyperplanes through $W$ consist of absolute points (and then, again, $(i)$ holds), or exactly two hyperplanes $H_1$ and $H_2$ through $W$ consist of absolute points. In this final case, let $H$ be the third hyperplane through $W$. Let $t,t_1,t_2$ be points such that $t^\perp = H$ and $t_i^\perp= H_i$, $i=1,2$. Then, since $\theta$ fixes $H$, we have $t\in W$. Since $t_i\in t_i^\perp= H_i$, $i=1,2$, we deduce $t_i\in W$, $i=1,2$. Hence $\theta$ induces collineations in $H,H_1,H_2$ having a hyperplane $W$ as fixed points. Consequently, these collineations are central involutions. Since all points of $W$ are fixed, all subspaces through $\{t,t_1,t_2\}$ are fixed. Hence the centres of the above collineations are $t,t_1,t_2$. Since the collineations in $H_i$, $i=1,2$, map points to a collinear point, the centers are $t_i$. But then the centre of the collineation in $H$ is $t$ and hence it also maps points to collinear points, a contradiction. This shows~(a).

(b) If the fixed points of $\theta$ form a $(2n-2)$-space $W$, then $\theta$ is a central elation in $\PG(2n-1,2)$, and the centre is necessarily $W^\perp$ since every point of $W$ is fixed, and hence every hyperplane through $W^\perp$ is fixed. No line through $W^\perp$ not contained in $W$ is a line of $\sC_n(2)$, whence~(b).
\end{proof}

\begin{lemma}\label{lem:CBase}
A collineation $\theta$ of the generalised quadrangle $\mathsf{C}_2(2)$ is exceptional domestic if and only if the set of absolute points of $\theta$ equals the union of two distinct hyperplanes in $\mathsf{PG}(3,2)$. 
\end{lemma}

\begin{proof}
It is known that $\sC_2(2)$ admits a unique exceptional domestic collineation (see \cite{TTM:12b}), and direct inspection shows that the set of absolute points of this collineation forms the union of two distinct hyperplanes in $\mathsf{PG}(3,2)$. It remains to show that no other collineation of $\sC_2(2)$ has such a structure of absolute points. This can be done, for example, using the character tables in the $\mathbb{ATLAS}$, see \cite[p.5]{ATLAS}. We omit the details.
\end{proof}

\begin{lemma}\label{lem:CBase3}
Let $\Delta=\sC_n(2)$ with $n\geq 3$ and let $\theta$ be a collineation. If the absolute points of $\theta$ lie on a union of two hyperplanes, and if the fixed points of $\theta$ form a $(2n-4)$-space $W$, then $\theta$ has decorated opposition diagram \quad
\begin{tikzpicture}[scale=0.5,baseline=-0.5ex]
\node at (0,0.3) {};
\node [inner sep=0.8pt,outer sep=0.8pt] at (-5,0) (-5) {$\bullet$};
\node [inner sep=0.8pt,outer sep=0.8pt] at (-4,0) (-4) {$\bullet$};
\node [inner sep=0.8pt,outer sep=0.8pt] at (-3,0) (-3) {$\bullet$};
\node [inner sep=0.8pt,outer sep=0.8pt] at (-2,0) (-2) {$\bullet$};
\node [inner sep=0.8pt,outer sep=0.8pt] at (-1,0) (-1) {$\bullet$};
\node [inner sep=0.8pt,outer sep=0.8pt] at (1,0) (1) {$\bullet$};
\node [inner sep=0.8pt,outer sep=0.8pt] at (2,0) (2) {$\bullet$};
\node [inner sep=0.8pt,outer sep=0.8pt] at (3,0) (3) {$\bullet$};
\node [inner sep=0.8pt,outer sep=0.8pt] at (4,0) (4) {$\bullet$};
\node [inner sep=0.8pt,outer sep=0.8pt] at (5,0) (5) {$\bullet$};
%\node at (2,-0.7) {$n-j$};
\draw [line width=0.5pt,line cap=round,rounded corners,fill=ggrey] (-5.north west)  rectangle (-5.south east);
\draw [line width=0.5pt,line cap=round,rounded corners,fill=ggrey] (-4.north west)  rectangle (-4.south east);
\draw [line width=0.5pt,line cap=round,rounded corners] (-3.north west)  rectangle (-3.south east);
\draw (-5,0)--(-0.5,0);
\draw (0.5,0)--(4,0);
\draw (4,0.07)--(5,0.07);
\draw (4,-0.07)--(5,-0.07);
\draw [style=dashed] (-1,0)--(1,0);
\node [inner sep=0.8pt,outer sep=0.8pt] at (-5,0) (-5) {$\bullet$};
\node [inner sep=0.8pt,outer sep=0.8pt] at (-4,0) (-4) {$\bullet$};
\node [inner sep=0.8pt,outer sep=0.8pt] at (-3,0) (-3) {$\bullet$};
\node [inner sep=0.8pt,outer sep=0.8pt] at (-2,0) (-2) {$\bullet$};
\node [inner sep=0.8pt,outer sep=0.8pt] at (-1,0) (-1) {$\bullet$};
\node [inner sep=0.8pt,outer sep=0.8pt] at (1,0) (1) {$\bullet$};
\node [inner sep=0.8pt,outer sep=0.8pt] at (2,0) (2) {$\bullet$};
\node [inner sep=0.8pt,outer sep=0.8pt] at (3,0) (3) {$\bullet$};
\node [inner sep=0.8pt,outer sep=0.8pt] at (4,0) (4) {$\bullet$};
\node [inner sep=0.8pt,outer sep=0.8pt] at (5,0) (5) {$\bullet$};
\end{tikzpicture}
\end{lemma}

\begin{proof} The hypothesis implies that every $3$-space contains a fixed point, and thus $\theta$ is $\{i\}$-domestic for all $4\leq i\leq n$. 

By the hypothesis on the structure of the absolute points of $\theta$ there exist points in $\mathrm{Opp}(\theta)$. Let $p$ be an arbitrary point in $\mathrm{Opp}(\theta)$. We will show below that the induced collineation $\theta_p$ of $\mathsf{C}_{n-1}(2)$ is $\{2\}$-domestic (in the inherited labelling). Hence $\theta$ is $\{1,2\}$-domestic. So if $\theta$ is capped then $\theta$ is $\{2\}$-domestic, however by \cite[Theorem~5.1]{TTM:12} every such collineation fixes a geometric hyperplane pointwise, contrary to our hypothesis that the fixed points form a $(2n-4)$-space. Thus $\theta$ is uncapped, and  then by Theorem~\ref{thm:main*}(a) the decorated opposition diagram of $\theta$ is forced to be as claimed. 

Therefore it only remains to show that $\theta_p$ is $\{2\}$-domestic (that is, point-domestic on $\sC_{n-1}(2)$). We fix some notation. Let $H_i$, $i=1,2$, be the two hyperplanes all points of which are absolute. Set $S=H_1\cap H_2$ and let $H$ be the hyperplane distinct from $H_i$, $i=1,2$, and containing $S$. Note that all points of $\mathrm{Opp}(\theta)$ are contained in $H$ (more precisely they form the set $H\setminus S$).

First we claim that any line in $\mathrm{Opp}(\theta)$ incident to $p$ must necessarily be contained in the hyperplane $H$. Suppose the such a line $L$ is not contained in $H$. Then $L=\{p,q_1,q_2\}$, with $q_i\in H_i$ and hence $q_i^\theta\perp q_i$, $i=1,2$. Since $p$ is not collinear to $p^\theta$, it must be collinear to $q_i^\theta$ for some $i\in\{1,2\}$. But then $q_i^\theta$ is collinear to all points of $L$, and so the line $L^\theta\ni q_i^\theta$ is not opposite the line $L$. Hence the claim.

Consider the subspace $\xi:= p^\perp\cap (p^\theta)^\perp$ of dimension $2n-3$. Then clearly $\xi$ contains the subspace $p^\perp\cap W$. We claim that $\dim(p^\perp\cap W)=2n-5$. Indeed, if not, then $W$ is a hyperplane of $\xi$. By Lemma~\ref{lem:fixed-subspace}(b) and our previous claim, all lines of $\sC_n(2)$ through $p$ are contained in $H$, implying $p^\perp=H$. But since $H$ is fixed by $\theta$ we deduce that $p\in W$, a contradiction. Our claim follows.

Hence $\dim(p^\perp\cap W)=2n-5$. It follows that $\dim(\xi\cap W)=2n-5$ as well, since $p^\perp\cap W=(p^\theta)^\perp\cap W$. Now let $q\in\xi\setminus W$. Suppose $q\notin H$. Then the line $\<p,q\>$ is not mapped to an opposite, as we showed above. Suppose $q\in S\setminus W$. Then $q^\theta\perp q$, and since $p^\theta\perp q$, we deduce that $q$ is collinear to $\<p,q\>^\theta$, implying that $\<p,q\>\notin\mathrm{Opp}(\theta)$.  Hence, if $\theta_p$ is not $\{2\}$-domestic, then $\xi\cap(H\setminus S)\neq\emptyset$. Under that conditon, if $\xi$ is not contained in $H$, then $\xi\cap H_i$ is a hyperplane of $\xi$, $i=1,2$, and this contradicts Lemma~\ref{lem:fixed-subspace}(a). 

Hence we deduce that if $\theta_p$ is not $\{2\}$-domestic, then $\xi\subseteq H$. In this case, since both $p$ and $p^\theta$ are in $H$, we have $p^\perp=\<p,\xi\>=H$ and $(p^{\theta})^{\perp}=\langle p^{\theta},\xi\rangle=H$. However $\perp$ is a symplectic polarity and so $p^\perp=H=(p^{\theta})^{\perp}$ forces $p=p^{\theta}$, a contradiction. The lemma is proved.
%$p\in H$, we have $p^\perp=\<p,\xi\>=H$. But also $p^{\theta}\in\Opp(\theta)$, and if $\theta_p$ is not $\{2\}$-domestic then $\theta_{p^{\theta}}$ is also not $\{2\}$-domestic, and hence $(p^{\theta})^{\perp}=\langle p^{\theta},\xi\rangle=H$. However $\perp$ is a symplectic polarity and so $p^\perp=H=(p^{\theta})^{\perp}$ forces $p=p^{\theta}$, a contradiction. The lemma is proved.
\end{proof}

\begin{thm}\label{thm:existenceBn(2)} Let $\theta$ be a collineation of $\mathsf{C}_n(2)$. Suppose that the set of absolute points of $\theta$ equals the union of two distinct hyperplanes of $\mathsf{PG}(2n-1,2)$. Then $\theta$ is domestic. Moreover, if $k$ is the projective dimension of the subspace of points of $\mathsf{PG}(2n-1,2)$ fixed by $\theta$, then
\begin{compactenum}[$(a)$]
\item if $k=n-2$ then $\theta$ is strongly exceptional domestic, and
\item if $k=n-1+j$ for some $0\leq j\leq n-3$ then $\theta$ is uncapped with decorated opposition diagram
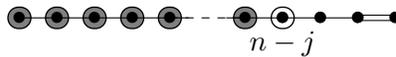
\begin{center}
\begin{tikzpicture}[scale=0.5,baseline=-0.5ex]
\node at (0,0.3) {};
\node [inner sep=0.8pt,outer sep=0.8pt] at (-5,0) (-5) {$\bullet$};
\node [inner sep=0.8pt,outer sep=0.8pt] at (-4,0) (-4) {$\bullet$};
\node [inner sep=0.8pt,outer sep=0.8pt] at (-3,0) (-3) {$\bullet$};
\node [inner sep=0.8pt,outer sep=0.8pt] at (-2,0) (-2) {$\bullet$};
\node [inner sep=0.8pt,outer sep=0.8pt] at (-1,0) (-1) {$\bullet$};
\node [inner sep=0.8pt,outer sep=0.8pt] at (1,0) (1) {$\bullet$};
\node [inner sep=0.8pt,outer sep=0.8pt] at (2,0) (2) {$\bullet$};
\node [inner sep=0.8pt,outer sep=0.8pt] at (3,0) (3) {$\bullet$};
\node [inner sep=0.8pt,outer sep=0.8pt] at (4,0) (4) {$\bullet$};
\node [inner sep=0.8pt,outer sep=0.8pt] at (5,0) (5) {$\bullet$};
\node at (2,-0.7) {$n-j$};
\draw [line width=0.5pt,line cap=round,rounded corners,fill=ggrey] (-5.north west)  rectangle (-5.south east);
\draw [line width=0.5pt,line cap=round,rounded corners,fill=ggrey] (-4.north west)  rectangle (-4.south east);
\draw [line width=0.5pt,line cap=round,rounded corners,fill=ggrey] (-3.north west)  rectangle (-3.south east);
\draw [line width=0.5pt,line cap=round,rounded corners,fill=ggrey] (-2.north west)  rectangle (-2.south east);
\draw [line width=0.5pt,line cap=round,rounded corners,fill=ggrey] (-1.north west)  rectangle (-1.south east);
\draw [line width=0.5pt,line cap=round,rounded corners,fill=ggrey] (1.north west)  rectangle (1.south east);
\draw [line width=0.5pt,line cap=round,rounded corners] (2.north west)  rectangle (2.south east);
\draw (-5,0)--(-0.5,0);
\draw (0.5,0)--(4,0);
\draw (4,0.07)--(5,0.07);
\draw (4,-0.07)--(5,-0.07);
\draw [style=dashed] (-1,0)--(1,0);
\node [inner sep=0.8pt,outer sep=0.8pt] at (-5,0) (-5) {$\bullet$};
\node [inner sep=0.8pt,outer sep=0.8pt] at (-4,0) (-4) {$\bullet$};
\node [inner sep=0.8pt,outer sep=0.8pt] at (-3,0) (-3) {$\bullet$};
\node [inner sep=0.8pt,outer sep=0.8pt] at (-2,0) (-2) {$\bullet$};
\node [inner sep=0.8pt,outer sep=0.8pt] at (-1,0) (-1) {$\bullet$};
\node [inner sep=0.8pt,outer sep=0.8pt] at (1,0) (1) {$\bullet$};
\node [inner sep=0.8pt,outer sep=0.8pt] at (2,0) (2) {$\bullet$};
\node [inner sep=0.8pt,outer sep=0.8pt] at (3,0) (3) {$\bullet$};
\node [inner sep=0.8pt,outer sep=0.8pt] at (4,0) (4) {$\bullet$};
\node [inner sep=0.8pt,outer sep=0.8pt] at (5,0) (5) {$\bullet$};
\end{tikzpicture}
\end{center}
\end{compactenum}
Moreover examples exist for each $n-2\leq k\leq 2n-4$.
\end{thm}

\begin{proof} Suppose that $\theta$ is a collineation of $\mathsf{C}_n(2)$ such that the set of absolute points of $\theta$ is the union of two distinct hyperplanes $H_1$ and $H_2$ of $\mathsf{PG}(2n-1,2)$. We show by induction on $n-j$ that $\theta$ is domestic, with Lemma~\ref{lem:CBase} providing the base case $n-j=3$.  

Let $p$ be any point not in $H_1\cup H_2$. Thus $p$ is mapped to an opposite point by $\theta$. Let $\mathrm{Res}(p)$ be the set of totally isotropic subspaces containing $p$. Thus $\mathrm{Res}(p)$ is a $\mathsf{C}_{n-1}(2)$ building, whose points are the lines through $p$, lines are the planes through $p$, and so forth. Let $\theta_p=\mathrm{proj}_{\mathrm{Res}(p)}\circ \theta$, regarded as a collineation of $\mathsf{C}_{n-1}(2)$. Since $p^{\perp}$ and $(p^{\theta})^{\perp}$ are hyperplanes of $\mathsf{PG}(2n-1,2)$ the spaces $H_i'=p^{\perp}\cap(p^{\theta})^{\perp}\cap H_i$ are $(2n-4)$-spaces for $i=1,2$ (as in the proof of Lemma~\ref{lem:CBase}). Let $q\in p^{\perp}\cap (p^{\theta})^{\perp}\cap (H_1\cup H_1)$, and let $L=\langle p,q\rangle$. Similar arguments as those in Lemma~\ref{lem:CBase} show that 
\begin{compactenum}[$(i)$]
\item if $q$ is fixed by $\theta$, then $L$ is fixed by $\theta_p$, and
\item if $q$ is mapped to a distinct collinear point by $\theta$ then $L$ is either fixed by $\theta_p$, or is mapped to a distinct coplanar line by $\theta_p$. 
\end{compactenum}
%To prove the first claim, if $q$ is fixed by $\theta$ then the line $\langle p,q\rangle\in \mathrm{Res}(p)$ and the image line $L^{\theta}=\langle p^{\theta},q\rangle$ intersect, and thus $L=L^{\theta_p}$. To prove the second claim, if $q$ is mapped to a distinct collinear point $q^{\theta}$ by $\theta$ then either $L^{\theta}=\langle q,q^{\theta}\rangle$ or $L^{\theta}\neq \langle q,q^{\theta}\rangle$. In the first case, projecting back onto $\mathrm{Res}(p)$ gives $L^{\theta_p}=L$. In the second case we have that $p$ is collinear with a point $r\in L^{\theta}\backslash\{p^{\theta}\}$, and then by projecting onto $\mathrm{Res}(p)$ we have $L^{\theta_p}=\langle p,r\rangle$. There are only two choices for $r$: Either $r=q^{\theta}$, in which case since $q^{\theta}$ is collinear with $p$ and $q$ it follows from the one-or-all property of polar spaces that every point of $L$ is collinear with every point of $L^{\theta_p}$ and so $L$ and $L^{\theta_p}$ are coplanar. Alternatively $r\neq q^{\theta}$, in which case since $q$ is collinear with both $q^{\theta}$ and $p^{\theta}$ it follows that $q$ is collinear with $r$, and again it easily follows that $L$ and $L^{\theta_p}$ are coplanar. Hence the claims (a) and (b) are proven. 
Thus for all non-domestic points $p$ the induced collineation $\theta_p$ of the $\mathsf{C}_{n-1}(2)$ building $\mathrm{Res}(p)$ has the property that the set of points mapped to collinear points (including fixed points) contains the union of two distinct hyperplanes in $\mathsf{PG}(2n-3,2)$. Thus by Lemma~\ref{lem:Cn(2)} and the induction hypothesis the collineation $\theta_p$ is domestic, and hence $\theta$ is domestic.

Now suppose that the absolute points of $\theta$ form a union of two hyperplanes, and that the fixed point set $F$ of $\theta$ is an $(n-2)$-space of $\mathsf{PG}(2n-1,2)$. We prove by induction on~$n$ that $\theta$ is strongly exceptional domestic, with Lemma~\ref{lem:CBase} providing the base case. The above argument shows that $\theta$ is necessarily domestic, and so it remains to show that there are non-domestic panels of each cotype $1,2,\ldots,n$. We claim that for $n\geq 3$ there exists a non-domestic point $p$ such that the hyperplane $p^{\perp}$ intersects $F$ in an $(n-3)$-space $F'$. To see this it suffices to show that there is a point $p$ with $p\notin H_1\cup H_2$ and $p\notin F^{\perp}$. The number of points in $H_1\cup H_2$ is $3\cdot 2^{2n-2}-1$ and the number of points in $F^{\perp}$ is $2^{n+1}-1$. Thus for $n\geq 3$ there is a point $p\notin H_1\cup H_2$ and $p\notin F^{\perp}$. By the induction hypothesis, there are panels of cotypes $2,3,\ldots,n$ of $\mathrm{Res}(p)$ mapped to an opposite panels by $\theta_p$, and thus there are panels of each cotype $2,3,\ldots,n$ of $\mathsf{C}_n(2)$ mapped to an opposite by $\theta$. It is then easy to see that there is also a non-domestic cotype $1$ panel (by a residue argument) and hence $\theta$ is strongly exceptional domestic. 
%In the residue of the type $n$-vertex of a cotype $2$ panel mapped to an opposite by $\theta$ we obtain a domestic duality of an $\mathsf{A}_{n-1}(2)$ building which is not a symplectic polarity, and thus this duality is strongly exceptional domestic and it follows that there is a cotype~$1$ panel of $\mathsf{C}_n(2)$ mapped to an opposite by $\theta$. Hence $\theta$ is strongly exceptional domestic. 

Now suppose that the absolute points of $\theta$ form a union of two hyperplanes, and that the fixed point set $F$ of $\theta$ is a $k$-space with $k=n-1+j$ for some $0\leq j\leq n-3$. An argument as in the previous paragraph shows that there is a non-domestic point $p$ such that $p^{\perp}$ intersects $F$ in an $(n-2+j)$-space. By induction, with Lemma~\ref{lem:CBase3} as the base case, the collineation $\theta_p$ of the $\sC_{n-1}(2)$ building $\mathrm{Res}(p)$ has diagram
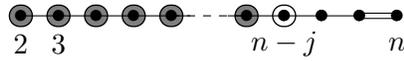
\begin{center}
\begin{tikzpicture}[scale=0.5,baseline=-0.5ex]
\node at (0,0.3) {};
\node [inner sep=0.8pt,outer sep=0.8pt] at (-5,0) (-5) {$\bullet$};
\node [inner sep=0.8pt,outer sep=0.8pt] at (-4,0) (-4) {$\bullet$};
\node [inner sep=0.8pt,outer sep=0.8pt] at (-3,0) (-3) {$\bullet$};
\node [inner sep=0.8pt,outer sep=0.8pt] at (-2,0) (-2) {$\bullet$};
\node [inner sep=0.8pt,outer sep=0.8pt] at (-1,0) (-1) {$\bullet$};
\node [inner sep=0.8pt,outer sep=0.8pt] at (1,0) (1) {$\bullet$};
\node [inner sep=0.8pt,outer sep=0.8pt] at (2,0) (2) {$\bullet$};
\node [inner sep=0.8pt,outer sep=0.8pt] at (3,0) (3) {$\bullet$};
\node [inner sep=0.8pt,outer sep=0.8pt] at (4,0) (4) {$\bullet$};
\node [inner sep=0.8pt,outer sep=0.8pt] at (5,0) (5) {$\bullet$};
\node at (2,-0.75) {$n-j$};
\node at (-5,-0.75) {$2$};
\node at (-4,-0.75) {$3$};
\node at (5,-0.75) {$n$};
\draw [line width=0.5pt,line cap=round,rounded corners,fill=ggrey] (-5.north west)  rectangle (-5.south east);
\draw [line width=0.5pt,line cap=round,rounded corners,fill=ggrey] (-4.north west)  rectangle (-4.south east);
\draw [line width=0.5pt,line cap=round,rounded corners,fill=ggrey] (-3.north west)  rectangle (-3.south east);
\draw [line width=0.5pt,line cap=round,rounded corners,fill=ggrey] (-2.north west)  rectangle (-2.south east);
\draw [line width=0.5pt,line cap=round,rounded corners,fill=ggrey] (-1.north west)  rectangle (-1.south east);
\draw [line width=0.5pt,line cap=round,rounded corners,fill=ggrey] (1.north west)  rectangle (1.south east);
\draw [line width=0.5pt,line cap=round,rounded corners] (2.north west)  rectangle (2.south east);
\draw (-5,0)--(-0.5,0);
\draw (0.5,0)--(4,0);
\draw (4,0.07)--(5,0.07);
\draw (4,-0.07)--(5,-0.07);
\draw [style=dashed] (-1,0)--(1,0);
\node [inner sep=0.8pt,outer sep=0.8pt] at (-5,0) (-5) {$\bullet$};
\node [inner sep=0.8pt,outer sep=0.8pt] at (-4,0) (-4) {$\bullet$};
\node [inner sep=0.8pt,outer sep=0.8pt] at (-3,0) (-3) {$\bullet$};
\node [inner sep=0.8pt,outer sep=0.8pt] at (-2,0) (-2) {$\bullet$};
\node [inner sep=0.8pt,outer sep=0.8pt] at (-1,0) (-1) {$\bullet$};
\node [inner sep=0.8pt,outer sep=0.8pt] at (1,0) (1) {$\bullet$};
\node [inner sep=0.8pt,outer sep=0.8pt] at (2,0) (2) {$\bullet$};
\node [inner sep=0.8pt,outer sep=0.8pt] at (3,0) (3) {$\bullet$};
\node [inner sep=0.8pt,outer sep=0.8pt] at (4,0) (4) {$\bullet$};
\node [inner sep=0.8pt,outer sep=0.8pt] at (5,0) (5) {$\bullet$};
\end{tikzpicture}
\end{center}
Moreover, for any other non-domestic point $p$ we have that either $\theta_p$ has the above diagram, or $\theta_p$ is domestic on type $n-1-j$ vertices. Thus no simplex $\mathsf{C}_n(2)$ of type $\{1,2,\ldots,n-j-1\}$ is mapped to an opposite by $\theta$, hence the result.

To conclude we prove existence of collineations with each diagram. Recursively define elements $g_n\in\mathsf{Sp}_{2n}(2)$, for $n\geq 2$, by
$$
g_2=\begin{bmatrix}
0&1&0&0\\
1&0&0&0\\
1&0&0&1\\
0&0&1&0
\end{bmatrix},\quad g_3=\begin{bmatrix}
0&0&1&0&0&0\\
0&1&0&0&0&0\\
1&0&0&0&0&0\\
0&1&0&0&0&1\\
1&0&0&0&1&0\\
0&0&0&1&0&0
\end{bmatrix}, \quad g_n=\begin{bmatrix}
1&0&0&0&0\\
1&1&0&0&0\\
0&0&g_{n-2}&0&0\\
0&0&0&1&0\\
0&0&0&1&1
\end{bmatrix}.
$$
Moreover, for each $j\geq 0$ define $g_n^{(j)}\in\mathsf{Sp}_{2n}(2)$ by
$$
g_n^{(j)}=\begin{bmatrix}
I_j&0&0\\
0&g_{n-j}&0\\
0&0&I_j
\end{bmatrix}.
$$
By direct calculation, the absolute points of $g_{2n}$ and $g_{2n}^{(j)}$ are given by $X_{2n-1}X_{2n}=0$ and the collinear points of $g_{2n+1}$ and $g^{(j)}_{2n+1}$ are given by $X_{n-1}(X_{n-2}+X_{n})=0$. Moreover, the fixed points of $g_n$ form an $(n-2)$-space of $\mathsf{PG}(2n-1,2)$, and the fixed points of $g_n^{(j)}$ form an $(n-2+j)$-space of $\mathsf{PG}(2n-1,2)$. Thus, by the arguments above, $g_n$ is a strongly exceptional domestic collineation of $\mathsf{C}_n(2)$ for each $n\geq 2$, and $g_n^{(j+1)}$ diagram as in~(b). 
\end{proof}

Similar theorems hold, with similar proofs, for the $\sB_n(2,4)$ and $\sD_n(2)$ buildings. We will only sketch the details below. Consider first the case $\sB_n(2,4)$. The following lemmas are similar to the $\sC_n(2)$ case.

\begin{lemma}\label{lem:BBase1}
A collineation $\theta$ of the generalised quadrangle $\mathsf{B}_2(2,4)$ is exceptional domestic if and only if the set of absolute points of $\theta$ is the set of points of $\sB_2(2,4)$ lying on the union of two distinct hyperplanes in $\mathsf{PG}(5,2)$. 
\end{lemma}

\begin{lemma}\label{lem:BBase3}
Let $\Delta=\sB_n(2,4)$ with $n\geq 3$ and let $\theta$ be a collineation. If the absolute points of $\theta$ lie on a union of two hyperplanes, and if the fixed points of $\theta$ are the isotropic points of a $(2n-3)$-space in $\mathsf{PG}(2n+1,2)$, then $\theta$ has decorated opposition diagram \quad
\begin{tikzpicture}[scale=0.5,baseline=-0.5ex]
\node at (0,0.3) {};
\node [inner sep=0.8pt,outer sep=0.8pt] at (-5,0) (-5) {$\bullet$};
\node [inner sep=0.8pt,outer sep=0.8pt] at (-4,0) (-4) {$\bullet$};
\node [inner sep=0.8pt,outer sep=0.8pt] at (-3,0) (-3) {$\bullet$};
\node [inner sep=0.8pt,outer sep=0.8pt] at (-2,0) (-2) {$\bullet$};
\node [inner sep=0.8pt,outer sep=0.8pt] at (-1,0) (-1) {$\bullet$};
\node [inner sep=0.8pt,outer sep=0.8pt] at (1,0) (1) {$\bullet$};
\node [inner sep=0.8pt,outer sep=0.8pt] at (2,0) (2) {$\bullet$};
\node [inner sep=0.8pt,outer sep=0.8pt] at (3,0) (3) {$\bullet$};
\node [inner sep=0.8pt,outer sep=0.8pt] at (4,0) (4) {$\bullet$};
\node [inner sep=0.8pt,outer sep=0.8pt] at (5,0) (5) {$\bullet$};
%\node at (2,-0.7) {$n-j$};
\draw [line width=0.5pt,line cap=round,rounded corners,fill=ggrey] (-5.north west)  rectangle (-5.south east);
\draw [line width=0.5pt,line cap=round,rounded corners,fill=ggrey] (-4.north west)  rectangle (-4.south east);
\draw [line width=0.5pt,line cap=round,rounded corners] (-3.north west)  rectangle (-3.south east);
\draw (-5,0)--(-0.5,0);
\draw (0.5,0)--(4,0);
\draw (4,0.07)--(5,0.07);
\draw (4,-0.07)--(5,-0.07);
\draw [style=dashed] (-1,0)--(1,0);
\node [inner sep=0.8pt,outer sep=0.8pt] at (-5,0) (-5) {$\bullet$};
\node [inner sep=0.8pt,outer sep=0.8pt] at (-4,0) (-4) {$\bullet$};
\node [inner sep=0.8pt,outer sep=0.8pt] at (-3,0) (-3) {$\bullet$};
\node [inner sep=0.8pt,outer sep=0.8pt] at (-2,0) (-2) {$\bullet$};
\node [inner sep=0.8pt,outer sep=0.8pt] at (-1,0) (-1) {$\bullet$};
\node [inner sep=0.8pt,outer sep=0.8pt] at (1,0) (1) {$\bullet$};
\node [inner sep=0.8pt,outer sep=0.8pt] at (2,0) (2) {$\bullet$};
\node [inner sep=0.8pt,outer sep=0.8pt] at (3,0) (3) {$\bullet$};
\node [inner sep=0.8pt,outer sep=0.8pt] at (4,0) (4) {$\bullet$};
\node [inner sep=0.8pt,outer sep=0.8pt] at (5,0) (5) {$\bullet$};
\end{tikzpicture}
\end{lemma}

\begin{thm}\label{thm:existenceBn(2,4)} Let $\theta$ be a collineation of $\mathsf{B}_{n}(2,4)$. Suppose that the set of absolute points of $\theta$ is the set of isotropic points lying on the union of two hyperplanes of $\mathsf{PG}(2n+1,2)$. Let $k$ be the projective dimension of the subspace of points of $\mathsf{PG}(2n+1,2)$ fixed by~$\theta$. Then $\theta$ is domestic, and
\begin{compactenum}[$(a)$]
\item if $k=n$ then $\theta$ is strongly exceptional domestic, and
\item if $k=n+1+j$ for some $0\leq j\leq n-3$ then $\theta$ is uncapped with decorated diagram
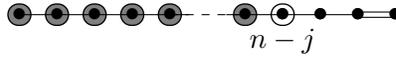
\begin{center}
\begin{tikzpicture}[scale=0.5,baseline=-0.5ex]
\node at (0,0.3) {};
\node [inner sep=0.8pt,outer sep=0.8pt] at (-5,0) (-5) {$\bullet$};
\node [inner sep=0.8pt,outer sep=0.8pt] at (-4,0) (-4) {$\bullet$};
\node [inner sep=0.8pt,outer sep=0.8pt] at (-3,0) (-3) {$\bullet$};
\node [inner sep=0.8pt,outer sep=0.8pt] at (-2,0) (-2) {$\bullet$};
\node [inner sep=0.8pt,outer sep=0.8pt] at (-1,0) (-1) {$\bullet$};
\node [inner sep=0.8pt,outer sep=0.8pt] at (1,0) (1) {$\bullet$};
\node [inner sep=0.8pt,outer sep=0.8pt] at (2,0) (2) {$\bullet$};
\node [inner sep=0.8pt,outer sep=0.8pt] at (3,0) (3) {$\bullet$};
\node [inner sep=0.8pt,outer sep=0.8pt] at (4,0) (4) {$\bullet$};
\node [inner sep=0.8pt,outer sep=0.8pt] at (5,0) (5) {$\bullet$};
\node at (2,-0.7) {$n-j$};
\draw [line width=0.5pt,line cap=round,rounded corners,fill=ggrey] (-5.north west)  rectangle (-5.south east);
\draw [line width=0.5pt,line cap=round,rounded corners,fill=ggrey] (-4.north west)  rectangle (-4.south east);
\draw [line width=0.5pt,line cap=round,rounded corners,fill=ggrey] (-3.north west)  rectangle (-3.south east);
\draw [line width=0.5pt,line cap=round,rounded corners,fill=ggrey] (-2.north west)  rectangle (-2.south east);
\draw [line width=0.5pt,line cap=round,rounded corners,fill=ggrey] (-1.north west)  rectangle (-1.south east);
\draw [line width=0.5pt,line cap=round,rounded corners,fill=ggrey] (1.north west)  rectangle (1.south east);
\draw [line width=0.5pt,line cap=round,rounded corners] (2.north west)  rectangle (2.south east);
\draw (-5,0)--(-0.5,0);
\draw (0.5,0)--(4,0);
\draw (4,0.07)--(5,0.07);
\draw (4,-0.07)--(5,-0.07);
\draw [style=dashed] (-1,0)--(1,0);
\node [inner sep=0.8pt,outer sep=0.8pt] at (-5,0) (-5) {$\bullet$};
\node [inner sep=0.8pt,outer sep=0.8pt] at (-4,0) (-4) {$\bullet$};
\node [inner sep=0.8pt,outer sep=0.8pt] at (-3,0) (-3) {$\bullet$};
\node [inner sep=0.8pt,outer sep=0.8pt] at (-2,0) (-2) {$\bullet$};
\node [inner sep=0.8pt,outer sep=0.8pt] at (-1,0) (-1) {$\bullet$};
\node [inner sep=0.8pt,outer sep=0.8pt] at (1,0) (1) {$\bullet$};
\node [inner sep=0.8pt,outer sep=0.8pt] at (2,0) (2) {$\bullet$};
\node [inner sep=0.8pt,outer sep=0.8pt] at (3,0) (3) {$\bullet$};
\node [inner sep=0.8pt,outer sep=0.8pt] at (4,0) (4) {$\bullet$};
\node [inner sep=0.8pt,outer sep=0.8pt] at (5,0) (5) {$\bullet$};
\end{tikzpicture}
\end{center}
\end{compactenum}
Moreover examples exist for each $n\leq k\leq 2n-2$. 
\end{thm}

\begin{proof}
The proofs are very similar to Theorem~\ref{thm:existenceBn(2)}, with the base cases given by Lemma~\ref{lem:BBase1} and~\ref{lem:BBase3}, and we omit the details. Thus it only remains to exhibit the existence of collineations of $\mathsf{B}_n(2,4)$ with the desired properties. To this end, define matrices $g_n$, $n\geq 3$ by
\begin{align*}
g_2=\begin{bmatrix}
0&1&0&0&0&0\\
0&0&0&0&0&1\\
0&0&1&0&0&0\\
0&0&0&1&0&0\\
1&0&0&0&0&0\\
0&0&0&0&1&0
\end{bmatrix},\quad 
g_3=\begin{bmatrix}
0&0&1&0&0&0&0&0\\
0&1&1&0&0&0&0&0\\
1&0&0&0&0&0&0&0\\
0&0&0&1&0&0&0&0\\
0&0&0&0&1&0&0&0\\
0&0&0&0&0&0&0&1\\
0&0&0&0&0&0&1&0\\
0&0&0&0&0&1&1&0
\end{bmatrix},\quad g_n=\begin{bmatrix}
1&0&0&0&0\\
1&1&0&0&0\\
0&0&g_{n-2}&0&0\\
0&0&0&1&0\\
0&0&0&1&1
\end{bmatrix}.
\end{align*}
Moreover, for each $j\geq 1$ define $g_n^{(j)}$ by 
$$
g_n^{(j)}=\begin{bmatrix}
I_j&0&0\\
0&g_{n-j}&0\\
0&0&I_j
\end{bmatrix}.
$$
Since $g_n,g_n^{(j)}\in \mathsf{GO}_{2n+2}^-(2)$ these matrices induce collineations of $\sB_n(2,4)$. It is straightforward to check that $g_n$ satisfies the conditions~(a) and $g_n^{(j+1)}$ satisfies the conditions~(b). 
\end{proof}
\newpage

Consider now the case $\sD_n(2)$. 

\begin{thm}\label{thm:existenceDn(2)} Let $\theta$ be an automorphism of $\mathsf{D}_n(2)$. Suppose that the set of absolute points of $\theta$ is the set of points of $\mathsf{D}_n(2)$ lying on the union of two hyperplanes of $\mathsf{PG}(2n-1,2)$. Let $k$ be the projective dimension of the subspace of points of $\mathsf{PG}(2n-1,2)$ fixed by $\theta$. Then $\theta$ is domestic, and
\begin{compactenum}[$(a)$]
\item  if $k=n-1$ and $\theta$ is an oppomorphism then $\theta$ is strongly exceptional domestic, and 
\item if $k=n-1+j$ for some $1\leq j\leq n-3$ and $\theta$ is a non-oppomorphism (for odd $j$) and an oppomorphism (for even $j$) then $\theta$ has diagram

\noindent\begin{minipage}{0.3\textwidth}
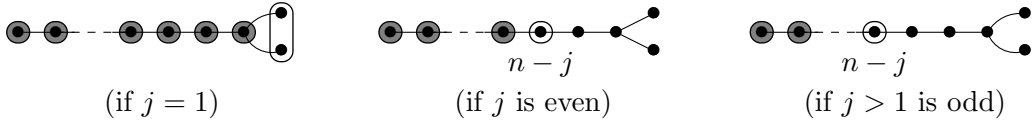
\begin{center}
\begin{tikzpicture}[scale=0.5,baseline=-0.5ex]
\node at (0,0.8) {};
\node at (0,-0.8) {};
\phantom{\node [below] at (1,-0.25) {$n-j$};}
%\node [inner sep=0.8pt,outer sep=0.8pt] at (-5,0) (-5) {$\bullet$};
%\node [inner sep=0.8pt,outer sep=0.8pt] at (-4,0) (-4) {$\bullet$};
%\node [inner sep=0.8pt,outer sep=0.8pt] at (-3,0) (-3) {$\bullet$};
\node [inner sep=0.8pt,outer sep=0.8pt] at (-2,0) (-2) {$\bullet$};
\node [inner sep=0.8pt,outer sep=0.8pt] at (-1,0) (-1) {$\bullet$};
%\node at (0,0) (0) {$\bullet$};
\node [inner sep=0.8pt,outer sep=0.8pt] at (1,0) (1) {$\bullet$};
\node [inner sep=0.8pt,outer sep=0.8pt] at (2,0) (2) {$\bullet$};
\node [inner sep=0.8pt,outer sep=0.8pt] at (3,0) (3) {$\bullet$};
\node [inner sep=0.8pt,outer sep=0.8pt] at (4,0) (4) {$\bullet$};
\node [inner sep=0.8pt,outer sep=0.8pt] at (5,0.5) (5a) {$\bullet$};
\node [inner sep=0.8pt,outer sep=0.8pt] at (5,-0.5) (5b) {$\bullet$};
%\draw [line width=0.5pt,line cap=round,rounded corners] (-4.north west)  rectangle (-4.south east);
\draw [line width=0.5pt,line cap=round,rounded corners,fill=ggrey] (-2.north west)  rectangle (-2.south east);
\draw [line width=0.5pt,line cap=round,rounded corners,fill=ggrey] (2.north west)  rectangle (2.south east);
\draw [line width=0.5pt,line cap=round,rounded corners,fill=ggrey] (4.north west)  rectangle (4.south east);
%\draw [line width=0.5pt,line cap=round,rounded corners] (-5.north west)  rectangle (-5.south east);
%\draw [line width=0.5pt,line cap=round,rounded corners] (-3.north west)  rectangle (-3.south east);
\draw [line width=0.5pt,line cap=round,rounded corners,fill=ggrey] (-1.north west)  rectangle (-1.south east);
\draw [line width=0.5pt,line cap=round,rounded corners,fill=ggrey] (1.north west)  rectangle (1.south east);
\draw [line width=0.5pt,line cap=round,rounded corners,fill=ggrey] (3.north west)  rectangle (3.south east);
\draw [line width=0.5pt,line cap=round,rounded corners] (5a.north west)  rectangle (5b.south east);
%\node [below] at (2,-0.25) {$2i$};
\draw (-2,0)--(-0.5,0);
\draw (0.5,0)--(4,0);
\draw (4,0) to [bend left] (5,0.5);
\draw (4,0) to [bend right=45] (5,-0.5);
\draw [style=dashed] (-1,0)--(1,0);
\node [inner sep=0.8pt,outer sep=0.8pt] at (-2,0) (-2) {$\bullet$};
\node [inner sep=0.8pt,outer sep=0.8pt] at (-1,0) (-1) {$\bullet$};
%\node at (0,0) (0) {$\bullet$};
\node [inner sep=0.8pt,outer sep=0.8pt] at (1,0) (1) {$\bullet$};
\node [inner sep=0.8pt,outer sep=0.8pt] at (2,0) (2) {$\bullet$};
\node [inner sep=0.8pt,outer sep=0.8pt] at (3,0) (3) {$\bullet$};
\node [inner sep=0.8pt,outer sep=0.8pt] at (4,0) (4) {$\bullet$};
\node [inner sep=0.8pt,outer sep=0.8pt] at (5,0.5) (5a) {$\bullet$};
\node [inner sep=0.8pt,outer sep=0.8pt] at (5,-0.5) (5b) {$\bullet$};
\end{tikzpicture}\newline
\emph{(if $j=1$)}
\end{center}
\end{minipage}
\begin{minipage}{0.3\textwidth}
\begin{center}
\begin{tikzpicture}[scale=0.5,baseline=-0.5ex]
\node at (0,0.8) {};
%\node [inner sep=0.8pt,outer sep=0.8pt] at (-5,0) (-5) {$\bullet$};
%\node [inner sep=0.8pt,outer sep=0.8pt] at (-4,0) (-4) {$\bullet$};
%\node [inner sep=0.8pt,outer sep=0.8pt] at (-3,0) (-3) {$\bullet$};
\node [inner sep=0.8pt,outer sep=0.8pt] at (-2,0) (-2) {$\bullet$};
\node [inner sep=0.8pt,outer sep=0.8pt] at (-1,0) (-1) {$\bullet$};
%\node at (0,0) (0) {$\bullet$};
\node [inner sep=0.8pt,outer sep=0.8pt] at (1,0) (1) {$\bullet$};
\node [inner sep=0.8pt,outer sep=0.8pt] at (2,0) (2) {$\bullet$};
\node [inner sep=0.8pt,outer sep=0.8pt] at (3,0) (3) {$\bullet$};
\node [inner sep=0.8pt,outer sep=0.8pt] at (4,0) (4) {$\bullet$};
\node [inner sep=0.8pt,outer sep=0.8pt] at (5,0.5) (5a) {$\bullet$};
\node [inner sep=0.8pt,outer sep=0.8pt] at (5,-0.5) (5b) {$\bullet$};
%\draw [line width=0.5pt,line cap=round,rounded corners] (-5.north west)  rectangle (2.south east);
%\draw [line width=0.5pt,line cap=round,rounded corners] (-5.north west)  rectangle (-5.south east);
%\draw [line width=0.5pt,line cap=round,rounded corners] (-4.north west)  rectangle (-4.south east);
%\draw [line width=0.5pt,line cap=round,rounded corners] (-3.north west)  rectangle (-3.south east);
\draw [line width=0.5pt,line cap=round,rounded corners,fill=ggrey] (-2.north west)  rectangle (-2.south east);
\draw [line width=0.5pt,line cap=round,rounded corners,fill=ggrey] (-1.north west)  rectangle (-1.south east);
\draw [line width=0.5pt,line cap=round,rounded corners,fill=ggrey] (1.north west)  rectangle (1.south east);
\draw [line width=0.5pt,line cap=round,rounded corners] (2.north west)  rectangle (2.south east);
\node [below] at (2,-0.25) {$n-j$};
\draw (-2,0)--(-0.5,0);
\draw (0.5,0)--(4,0);
\draw (4,0) to (5,0.5);
\draw (4,0) to   (5,-0.5);
\draw [style=dashed] (-1,0)--(1,0);
\node [inner sep=0.8pt,outer sep=0.8pt] at (-2,0) (-2) {$\bullet$};
\node [inner sep=0.8pt,outer sep=0.8pt] at (-1,0) (-1) {$\bullet$};
%\node at (0,0) (0) {$\bullet$};
\node [inner sep=0.8pt,outer sep=0.8pt] at (1,0) (1) {$\bullet$};
\node [inner sep=0.8pt,outer sep=0.8pt] at (2,0) (2) {$\bullet$};
\node [inner sep=0.8pt,outer sep=0.8pt] at (3,0) (3) {$\bullet$};
\node [inner sep=0.8pt,outer sep=0.8pt] at (4,0) (4) {$\bullet$};
\node [inner sep=0.8pt,outer sep=0.8pt] at (5,0.5) (5a) {$\bullet$};
\node [inner sep=0.8pt,outer sep=0.8pt] at (5,-0.5) (5b) {$\bullet$};
\end{tikzpicture}\newline
\emph{(if $j$ is even)}
\end{center}
\end{minipage}
\begin{minipage}{0.3\textwidth}
\begin{center}
\begin{tikzpicture}[scale=0.5,baseline=-0.5ex]
\node at (0,0.8) {};
%\node [inner sep=0.8pt,outer sep=0.8pt] at (-5,0) (-5) {$\bullet$};
%\node [inner sep=0.8pt,outer sep=0.8pt] at (-4,0) (-4) {$\bullet$};
%\node [inner sep=0.8pt,outer sep=0.8pt] at (-3,0) (-3) {$\bullet$};
\node [inner sep=0.8pt,outer sep=0.8pt] at (-2,0) (-2) {$\bullet$};
\node [inner sep=0.8pt,outer sep=0.8pt] at (-1,0) (-1) {$\bullet$};
%\node at (0,0) (0) {$\bullet$};
\node [inner sep=0.8pt,outer sep=0.8pt] at (1,0) (1) {$\bullet$};
\node [inner sep=0.8pt,outer sep=0.8pt] at (2,0) (2) {$\bullet$};
\node [inner sep=0.8pt,outer sep=0.8pt] at (3,0) (3) {$\bullet$};
\node [inner sep=0.8pt,outer sep=0.8pt] at (4,0) (4) {$\bullet$};
\node [inner sep=0.8pt,outer sep=0.8pt] at (5,0.5) (5a) {$\bullet$};
\node [inner sep=0.8pt,outer sep=0.8pt] at (5,-0.5) (5b) {$\bullet$};
%\draw [line width=0.5pt,line cap=round,rounded corners] (-5.north west)  rectangle (2.south east);
%\draw [line width=0.5pt,line cap=round,rounded corners] (-5.north west)  rectangle (-5.south east);
%\draw [line width=0.5pt,line cap=round,rounded corners] (-4.north west)  rectangle (-4.south east);
%\draw [line width=0.5pt,line cap=round,rounded corners] (-3.north west)  rectangle (-3.south east);
\draw [line width=0.5pt,line cap=round,rounded corners,fill=ggrey] (-2.north west)  rectangle (-2.south east);
\draw [line width=0.5pt,line cap=round,rounded corners,fill=ggrey] (-1.north west)  rectangle (-1.south east);
\draw [line width=0.5pt,line cap=round,rounded corners] (1.north west)  rectangle (1.south east);
%\draw [line width=0.5pt,line cap=round,rounded corners] (2.north west)  rectangle (2.south east);
\node [below] at (1,-0.25) {$n-j$};
\draw (-2,0)--(-0.5,0);
\draw (0.5,0)--(4,0);
\draw (4,0) to [bend left] (5,0.5);
\draw (4,0) to [bend right=45] (5,-0.5);
\draw [style=dashed] (-1,0)--(1,0);
\node [inner sep=0.8pt,outer sep=0.8pt] at (-2,0) (-2) {$\bullet$};
\node [inner sep=0.8pt,outer sep=0.8pt] at (-1,0) (-1) {$\bullet$};
%\node at (0,0) (0) {$\bullet$};
\node [inner sep=0.8pt,outer sep=0.8pt] at (1,0) (1) {$\bullet$};
\node [inner sep=0.8pt,outer sep=0.8pt] at (2,0) (2) {$\bullet$};
\node [inner sep=0.8pt,outer sep=0.8pt] at (3,0) (3) {$\bullet$};
\node [inner sep=0.8pt,outer sep=0.8pt] at (4,0) (4) {$\bullet$};
\node [inner sep=0.8pt,outer sep=0.8pt] at (5,0.5) (5a) {$\bullet$};
\node [inner sep=0.8pt,outer sep=0.8pt] at (5,-0.5) (5b) {$\bullet$};
\end{tikzpicture}\newline
\emph{(if $j>1$ is odd)}
\end{center}
\end{minipage}
\end{compactenum}

\smallskip
\noindent Moreover examples exist for all $n-1\leq k\leq 2n-4$.
\end{thm}

\begin{proof}
The proofs of statements (a) and (b) are again analogous to those in Theorem~\ref{thm:existenceBn(2)}, with an appropriate start to the induction. We omit the details. 

To prove existence, note that the matrices $g_{n-1}$, $n\geq 3$, from the proof of Theorem~\ref{thm:existenceBn(2,4)} are also elements of $\mathsf{GO}_{2n}^+(2)$. Let $h_3=g_2$ and $h_4=g_3$. Then $h_3$ induces a duality of $\sD_3(2)$ and $h_4$ induces a collineation of $\sD_4(2)$. Let $h_n=g_{n-1}$, and for each $1\leq j\leq n-3$ let $h_n^{(j)}=g_{n-1}^{(j)}$. It is easy to check that $h_n$ satisfies conditions~(a), and $h_n^{(j)}$ satisfies conditions~(b).
\end{proof}

%
%\begin{conjecture} \james{A more precise conjecture says:} Let $\Delta$ be a building of type $C_n(2)$. If $\theta$ is a collineation such that the set of fixed points and points mapped to collinear points lies on the union of two hyperplanes then $\theta$ is domestic and has order~$4$. Moreover, the set of fixed points of such a $\theta$ is necessarily an $(n-1+k)$--dimensional space for some $k=0,1,\ldots,n-2$, and for each of these values of $k$ there is a unique example of such a $\theta$. The example with $k=0$ is strongly exceptional domestic, and the fixed $n-1$ space is isotropic. The example with $k=1$ is exceptional domestic but not strongly exceptional domestic, and the fixed $n$-space is not isotropic. For $k\geq 2$ the example is $\{i\}$-domestic for all $i> n-k+1$, yet maps vertices of each type $i\leq n-k+1$ to opposites. 
%\end{conjecture}
%

\section{Uncapped automorphisms for exceptional types}\label{sec:exceptional}

In this section we prove Theorem~\ref{thm:main*}(b) for the small buildings of exceptional type. Moreover we completely classify the domestic automorphisms of the buildings $\sF_4(2)$, $\sF_4(2,4)$, and $\sE_6(2)$. We begin, in Section~\ref{sec:detect}, by developing a (computationally feasible) method of detecting whether a given automorphism is domestic. In Section~\ref{sec:minimal} we briefly describe the implementation of the minimal faithful permutation representations of the relevant $\mathbb{ATLAS}$ groups into $\mathsf{MAGMA}$, and then in Section~\ref{sec:EF} we give the classification of domestic automorphisms of the buildings $\sF_4(2)$, $\sF_4(2,4)$, and $\sE_6(2)$ making use of these permutation representations. We provide examples of uncapped automorphisms in $\sE_7(2)$, and give conjectures for $\sE_8(2)$ in Section~\ref{sec:E7E8}.

Throughout this section we will use standard notation for Chevalley groups and twisted Chevalley groups~$G$, and we refer to Carter~\cite{Car:89} for details. In particular, the symbols $B$, $H$, $N$, $U$, $W$, $S$, $R$, $x_{\alpha}(a)$, $n_{\alpha}(a)$, etc, have their usual meanings. However we note that in the twisted case we use these symbols for the objects in the twisted group (rather than the untwisted group). Then the quadruple $(B,N,W,S)$ forms a Tits system in $G$, and thus $(\Delta,\delta)$ is a building of type $(W,S)$ where $\Delta=G/B$ and $\delta(gB,hB)=w$ if and only if $g^{-1}h\in BwB$. In the case of graph automorphisms of a simply laced Dynkin diagram we assume that the Chevalley generators are chosen so that \cite[Proposition~12.2.3]{Car:89} holds (in particular $x_{\alpha}(a)^{\sigma}=x_{\sigma(\alpha)}(\pm a)$). 

\subsection{Detecting domesticity}\label{sec:detect}

The following lemma shows that under certain hypotheses, to verify domesticity it is sufficient to show that no chamber opposite a given chamber is mapped onto an opposite. As we see in the remark after the lemma, the hypotheses cannot be removed. 
\newpage

\begin{lemma}\label{lem:red2}
Let $\theta$ be an automorphism of a thick spherical building~$\Delta$, and let $L=\disp(\theta)$. Let $C$ be any chamber. Suppose that either
\begin{compactenum}[$(i)$]
\item each panel of $\Delta$ has at least $4$ chambers, or
\item $\theta$ is an involution, or
\item $\theta$ induces opposition and $L=\ell(w_0)$. 
\end{compactenum}
Then there exists a chamber $D$ with $\delta(C,D)=w_0$ and $\ell(\delta(D,D^{\theta}))=L$. 
\end{lemma}

\begin{proof}
Let $E$ be a chamber with $\ell(\delta(E,E^{\theta}))=L$, and write $v=\delta(E,E^{\theta})$. Let $w=\delta(C,E)$, and suppose that $w\neq w_0$. Then there exists $s\in S$ with $\ell(ws)>\ell(w)$. We show that there is a chamber $D$ with $\delta(E,D)=s$ such that $\ell(\delta(D,D^{\theta}))=L$. Consider each case. 

\begin{compactenum}[$(1)$]
\item $\ell(sv)<\ell(v)$. Then either:
\begin{compactenum}[$(a)$]
\item $\ell(svs^{\theta})=\ell(v)$, in which case we choose the unique $D$ with $\delta(E,D)=s$ such that $\delta(D,E^{\theta})=sv$. Since $\delta(E^{\theta},D^{\theta})=s^{\theta}$ and $\ell(svs^{\theta})>\ell(sv)$ we have $\delta(D,D^{\theta})=svs^{\theta}$ and so $\ell(\delta(D,D^{\theta}))=L$.
\item $\ell(svs^{\theta})<\ell(v)$, in which case necessarily $\ell(vs^{\theta})<\ell(v)$, and it follows that there exists a reduced expression for $v$ starting with $s$ and ending with $s^{\theta}$. Thus there exists a minimal length gallery 
$E=E_0\sim_{s_1}E_1\sim_{s_2}\cdots\sim_{s_{\ell-1}}E_{\ell-1}\sim_{s_{\ell}}E_{\ell}=E^{\theta}$
with $s_1=s$ and $s_{\ell}=s^{\theta}$. 
\begin{compactenum}[$(i)$]
\item If every panel of $\Delta$ has at least $4$ chambers then there exists a chamber $D$ with $\delta(E,D)=s$ such that $D\notin \{E_1,E_{\ell-1}^{\theta^{-1}}\}$. Then there is a gallery $D\sim_{s_1}E_1\sim_{s_2}\cdots\sim_{s_{\ell-1}}E_{\ell-1}\sim_{s_{\ell}}D^{\theta}$, and hence $\delta(D,D^{\theta})=v$ has length~$L$. 
\item If $\theta$ is an involution then $\theta$ maps every minimal length gallery from $E$ to $E^{\theta}$ to a minimal length gallery from $E^{\theta}$ to $E$, and it follows by considering types of first and last steps that $E_1^{\theta}=E_{\ell-1}$. Thus for any $D$ with $\delta(E,D)=s$ and $D\neq E_1$ we again have $\delta(D,D^{\theta})=v$. 
\item If $\theta$ induces opposition and $L=\ell(w_0)$ then $v=w_0$, and $svs^{\theta}=sw_0s^{\theta}=w_0s^{\theta}s^{\theta}=w_0$, and so case (1)(b) cannot occur. 
\end{compactenum}
\end{compactenum}
\item $\ell(sv)>\ell(v)$. Then either:
\begin{compactenum}[$(a)$]
\item $\ell(svs^{\theta})>\ell(v)$, in which case every chamber $D$ with $\delta(E,D)=s$ has $\delta(D,D^{\theta})=svs^{\theta}$, contradicting $\ell(v)=\disp(\theta)$. Thus this case cannot occur. 
\item $\ell(svs^{\theta})=\ell(v)$, in which case we choose $D$ to be any chamber with $\delta(E,D)=s$. Then $\delta(D,E^{\theta})=sv$ (since $\ell(sv)>\ell(v)$), and thus $\delta(D,D^{\theta})=sv$ or $\delta(D,D^{\theta})=svs^{\theta}$. The first case is impossible by the definition of displacement, and thus $\delta(D,D^{\theta})=svs^{\theta}$ has length~$L$.
\end{compactenum}
\end{compactenum}
Hence the result.
\end{proof}

\begin{remark}\label{rem:counter}
The following examples illustrate that the conclusion of Lemma~\ref{lem:red2} may fail if the hypotheses of the lemma are not satisfied.  
\begin{compactenum}[$(1)$]
\item The collineation $\theta$ of the Fano plane given by the upper triangular $3\times 3$ matrix with all upper triangular entries equal to $1$ maps no chamber opposite the base chamber $C=(\langle e_1\rangle,\langle e_1+e_2\rangle)$ to an opposite chamber. However this collineation has displacement $\ell(w_0)=3$, since no nontrivial collineation of a projective plane is domestic.
\item The exceptional domestic collineation of the generalised quadrangle $\mathsf{GQ}(2)=\sC_2(2)$ is given by $\theta=x_1(1)x_2(1)$ in Chevalley generators. The chambers opposite the base chamber $B$ of $G/B$ are mapped to distances $s_1s_2$ or $s_2s_1$, however $\theta$ has displacement $3$ (by both $s_1s_2s_1$ and $s_2s_1s_2$). 
\end{compactenum}
\end{remark} 

%
%If $\theta$ is capped then the set $\{\delta(C,C^{\theta})\mid C\in\Delta\}$ contains a unique element of maximal length (this may fail if $\theta$ is not capped, as the $\mathsf{GQ}(2)$ example above illustrates). Denoting this element by $w_{\theta}$, Lemma~\ref{lem:red2} immediately implies the following.
%
%
%\begin{cor} Under the hypothesis of Lemma~\ref{lem:red2}, if $\theta$ is capped then there exists a chamber $D$ with $\delta(C,D)=w_0$ and $\delta(D,D^{\theta})=w_{\theta}$.
%\end{cor}
%
\subsection{Minimal faithful permutation representations}\label{sec:minimal}

Let $\mathcal{G}$ be the following set of $\mathbb{ATLAS}$ groups:
$$\mathcal{G}=\{\sF_4(2),\sF_4(2).2,{^2}\sE_6(2^2),{^2}\sE_6(2^2).2,\sE_6(2),\sE_6(2).2\}.
$$
These groups are, respectively, the collineation group of $\sF_4(2)$, the full automorphism group of $\sF_4(2)$ (including dualities), the ``inner'' automorphism group of $\sF_4(2,4)$, the full automorphism group of $\sF_4(2,4)$, the collineation group of $\sE_6(2)$, and the full automorphism group of $\sE_6(2)$. In the following section we will need an explicit set of conjugacy class representatives for the groups in $\mathcal{G}$. With the exception of perhaps $\sF_4(2)$, these groups appear to be too large for the standard conjugacy class algorithms in $\mathsf{MAGMA}$ (or $\mathsf{GAP}$) when input as matrix groups using the standard adjoint representation (for example $\sE_6(2).2$ has order $429683151044011150540800$, and in any case it is not an entirely trivial task to construct such extensions as a matrix group). However the available algorithms in both $\mathsf{MAGMA}$ and $\mathsf{GAP}$ for permutation groups turn out to be considerably more efficient, and therefore we require faithful permutation representations of the groups in $\mathcal{G}$.

 The degrees $\deg(G)$ of the minimal faithful permutation representations of the groups in $\mathcal{G}$ are well known (see for example \cite{Vas:96,Vas:97,Vas:98}): $\deg(\sF_4(2))=69615$, $\deg(\sF_4(2).2)=139230$, $\deg({^2}\sE_6(2^2))=\deg({^2}\sE_6(2^2).2)=3968055$, $\deg(\sE_6(2))=139503$, and $\deg(\sE_6(2).2)=279006$. In each case the permutation representation can naturally be realised by the action of $G$ on certain maximal parabolic coset spaces (equivalently, on certain vertices of the building). For example, for $G=\sE_6(2).2$ we consider the action on $G/P_1\cup G/P_6$ (the set of type $1$ and type $6$ vertices of the $\sE_6(2)$ building), and for $G={^2}\sE_6(2^2).2$ we consider the action on ${^2}\sE_6(2^2)/P_1$ (the set of type $1$ vertices of the $\sF_4(2,4)$ building), where $P_i$ denotes the maximal parabolic subgroup of type $S\backslash \{s_i\}$.

To our knowledge, at the time of writing these minimal faithful permutation representations were not available in either $\mathsf{GAP}$ or $\mathsf{MAGMA}$. Therefore we have implemented these permutation representations using the above action on vertices of the building and the The Groups of Lie Type package~\cite{CMT:04}. The resulting permutation representations are available on the first author's webpage, where we also provide lists of conjugacy class representatives and code relevant to the computations in the following sections. We would like to thank Bill Unger from the $\mathsf{MAGMA}$ team at Sydney University for helping us generate the conjugacy class representatives from the permutation representations.

\subsection{Domestic automorphism of small buildings of types $\sF_4$ and $\sE_6$}\label{sec:EF}

In this section we classify the domestic automorphisms of the buildings $\sF_4(2)$, $\sF_4(2,4)$, and $\sE_6(2)$. This requires two main steps. We first exhibit a list of $n$ examples of pairwise non-conjugate domestic automorphisms for each building (for some integer $n$). Next, using an explicit set of conjugacy class representatives, we show that all but $n$ of these representatives map some chamber to an opposite and are hence non-domestic. Thus we conclude that our list of $n$ examples is complete.

We make frequent use of both commutator relations, and the formula 
\begin{align}
\label{eq:n}
n_{\alpha}(a)=x_{\alpha}(a)x_{-\alpha}(-a^{-1})x_{\alpha}(a).
\end{align}
 We will also use the following observation: For the buildings $\sE_n(2)$, $n=6,7,8$, the displacement of an automorphism $\theta$ determines the (decorated) opposition diagram of $\theta$ (c.f. Remark~\ref{rem:disp}). For the buildings $\sF_4(2)$ and $\sF_4(2,4)$ the (capped) automorphisms with types $\{1\}$ and $\{4\}$ are not distinguished by displacement, and furthermore in $\sF_4(2)$ the three uncapped diagrams all have displacement~$23$.

Before beginning we outline of a useful technique. Suppose that $\theta\in G$ induces an automorphism of $\Delta=G/B$ such that the hypothesis of Lemma~\ref{lem:red2} holds. Then there exists $gB\in Bw_0B/B$ such that $\disp(\theta)=\ell(\delta(gB,\theta gB))$. Each $gB\in Bw_0B/B$ can be written as $gB=uw_0B$ with $u\in U$, and $\delta(gB,\theta gB)$ is the unique $w\in W$ such that
\begin{align}\label{eq:calcw}
w_0^{-1}u^{-1}\theta uw_0\in BwB.
\end{align}
Thus to determine $\disp(\theta)$ it is sufficient to analyse the terms $w_0^{-1}u^{-1}\theta uw_0$ with $u\in U$. However $|U|=|\mathbb{F}|^{\ell(w_0)}$, and so even for relatively small buildings it is not computationally feasible practical to check each $u\in U$ (for example, in $\sE_6(2)$ we have $|U|=2^{36}$). 

The following idea often provides considerable efficiency. Note that each $u\in U$ can be written as $\prod_{\alpha\in R^+}x_{\alpha}(a_{\alpha})$ with $a_{\alpha}\in\mathbb{F}$ and the product taken in any order (see \cite[Lemma~17]{St:16}; of course the $a_{\alpha}$ depend on the order chosen). Writing 
$A=\{\alpha\in R^+\mid x_{\alpha}(a)\theta\neq \theta x_{\alpha}(a)\text{ for all $a\in\mathbb{F}$}\}
$ 
we can write $u=u_A'u_{A}$ where $u_{A}$ is a product over terms $\alpha\in A$, and $u_A'$ is a product over the remaining positive roots. Then $u_A'$ commutes with $\theta$, and so 
\begin{align}\label{eq:calcw2}
w_0^{-1}u^{-1}\theta uw_0=w_0^{-1}u_A^{-1}\theta u_Aw_0.
\end{align}
There are $|\mathbb{F}|^{|A|}$ such elements, and so the technique works best if a conjugacy class representative for $\theta$ is chosen with the property that it commutes with as many elements $x_{\alpha}(a)$, $\alpha\in R^+$, as possible.

The residue of the type $J$ simplex of the chamber $gB$ is the coset $gP_{S\backslash J}$, and this residue is non-domestic for $\theta$ if and only if $
g^{-1}\theta g\in P_{S\backslash J}w_0P_{S\backslash J}$, and thus if and only if 
\begin{align}\label{eq:para}
\text{$g^{-1}\theta g\in BwB$ for some $w\in w_0W_{S\backslash J}$}
\end{align}
In the following we write $g_1\sim g_2$ to mean that $g_1$ and $g_2$ are conjugate in~$G$.

\begin{thm}\label{thm:F4} Let $G=\sF_4(2)$, and let $\Delta=G/B$ be the associated building. Let $\varphi=(2342)$ and~$\varphi'=(1232)$ be the highest root and highest short root (respectively) of the $\sF_4$ root system. There are precisely $6$ conjugacy classes of domestic collineations of $\Delta$, as follows:
$$
\noindent\begin{tabular}{|l|l|l|l|l|}
\hline
$\theta$&\emph{capped}&\emph{diagram}&\emph{fixed type $1/4$ vertices}&$\mathbb{ATLAS}$\\
\hline\hline
$\theta_1=x_{\varphi}(1)$&\emph{yes}&\begin{tikzpicture}[scale=0.5,baseline=-0.5ex]
\node at (0,0.3) {};
\node [inner sep=0.8pt,outer sep=0.8pt] at (-1.5,0) (1) {$\bullet$};
\node [inner sep=0.8pt,outer sep=0.8pt] at (-0.5,0) (2) {$\bullet$};
\node [inner sep=0.8pt,outer sep=0.8pt] at (0.5,0) (3) {$\bullet$};
\node [inner sep=0.8pt,outer sep=0.8pt] at (1.5,0) (4) {$\bullet$};
\draw (-1.5,0)--(-0.5,0);
\draw (0.5,0)--(1.5,0);
\draw (-0.5,0.07)--(0.5,0.07);
\draw (-0.5,-0.07)--(0.5,-0.07);
\draw [line width=0.5pt,line cap=round,rounded corners] (1.north west)  rectangle (1.south east);
\end{tikzpicture}&$2287/5103$&$2B$\\
\hline
$\theta_2=x_{\varphi'}(1)$&\emph{yes}&\begin{tikzpicture}[scale=0.5,baseline=-0.5ex]
\node at (0,0.3) {};
\node [inner sep=0.8pt,outer sep=0.8pt] at (-1.5,0) (1) {$\bullet$};
\node [inner sep=0.8pt,outer sep=0.8pt] at (-0.5,0) (2) {$\bullet$};
\node [inner sep=0.8pt,outer sep=0.8pt] at (0.5,0) (3) {$\bullet$};
\node [inner sep=0.8pt,outer sep=0.8pt] at (1.5,0) (4) {$\bullet$};
\draw (-1.5,0)--(-0.5,0);
\draw (0.5,0)--(1.5,0);
\draw (-0.5,0.07)--(0.5,0.07);
\draw (-0.5,-0.07)--(0.5,-0.07);
\phantom{\draw [line width=0.5pt,line cap=round,rounded corners] (1.north west)  rectangle (1.south east);}
\draw [line width=0.5pt,line cap=round,rounded corners] (4.north west)  rectangle (4.south east);
\end{tikzpicture}&$5103/2287$&$2A$\\
\hline
$\theta_3=x_{\varphi}(1)x_{\varphi'}(1)$&\emph{yes}&\begin{tikzpicture}[scale=0.5,baseline=-0.5ex]
\node at (0,0.3) {};
\node [inner sep=0.8pt,outer sep=0.8pt] at (-1.5,0) (1) {$\bullet$};
\node [inner sep=0.8pt,outer sep=0.8pt] at (-0.5,0) (2) {$\bullet$};
\node [inner sep=0.8pt,outer sep=0.8pt] at (0.5,0) (3) {$\bullet$};
\node [inner sep=0.8pt,outer sep=0.8pt] at (1.5,0) (4) {$\bullet$};
\draw (-1.5,0)--(-0.5,0);
\draw (0.5,0)--(1.5,0);
\draw (-0.5,0.07)--(0.5,0.07);
\draw (-0.5,-0.07)--(0.5,-0.07);
\draw [line width=0.5pt,line cap=round,rounded corners] (1.north west)  rectangle (1.south east);
\draw [line width=0.5pt,line cap=round,rounded corners] (4.north west)  rectangle (4.south east);
\end{tikzpicture}&$1263/1263$&$2C$\\
\hline
$\theta_4=x_1(1)x_2(1)$&\emph{no}&\begin{tikzpicture}[scale=0.5]
\node at (0,0.3) {};
\node [inner sep=0.8pt,outer sep=0.8pt] at (-1.5,0) (1) {$\bullet$};
\node [inner sep=0.8pt,outer sep=0.8pt] at (-0.5,0) (2) {$\bullet$};
\node [inner sep=0.8pt,outer sep=0.8pt] at (0.5,0) (3) {$\bullet$};
\node [inner sep=0.8pt,outer sep=0.8pt] at (1.5,0) (4) {$\bullet$};
\draw [line width=0.5pt,line cap=round,rounded corners,fill=ggrey] (1.north west)  rectangle (1.south east);
\draw [line width=0.5pt,line cap=round,rounded corners,fill=ggrey] (2.north west)  rectangle (2.south east);
\draw [line width=0.5pt,line cap=round,rounded corners] (3.north west)  rectangle (3.south east);
\draw [line width=0.5pt,line cap=round,rounded corners] (4.north west)  rectangle (4.south east);
\draw (-1.5,0)--(-0.5,0);
\draw (0.5,0)--(1.5,0);
\draw (-0.5,0.07)--(0.5,0.07);
\draw (-0.5,-0.07)--(0.5,-0.07);
\node [inner sep=0.8pt,outer sep=0.8pt] at (-1.5,0) (1) {$\bullet$};
\node [inner sep=0.8pt,outer sep=0.8pt] at (-0.5,0) (2) {$\bullet$};
\node [inner sep=0.8pt,outer sep=0.8pt] at (0.5,0) (3) {$\bullet$};
\node [inner sep=0.8pt,outer sep=0.8pt] at (1.5,0) (4) {$\bullet$};
\end{tikzpicture}&$127/399$&$4D$\\
\hline
$\theta_5=x_4(1)x_3(1)$&\emph{no}&\begin{tikzpicture}[scale=0.5]
\node at (0,0.3) {};
\node [inner sep=0.8pt,outer sep=0.8pt] at (-1.5,0) (1) {$\bullet$};
\node [inner sep=0.8pt,outer sep=0.8pt] at (-0.5,0) (2) {$\bullet$};
\node [inner sep=0.8pt,outer sep=0.8pt] at (0.5,0) (3) {$\bullet$};
\node [inner sep=0.8pt,outer sep=0.8pt] at (1.5,0) (4) {$\bullet$};
\draw [line width=0.5pt,line cap=round,rounded corners] (1.north west)  rectangle (1.south east);
\draw [line width=0.5pt,line cap=round,rounded corners] (2.north west)  rectangle (2.south east);
\draw [line width=0.5pt,line cap=round,rounded corners,fill=ggrey] (3.north west)  rectangle (3.south east);
\draw [line width=0.5pt,line cap=round,rounded corners,fill=ggrey] (4.north west)  rectangle (4.south east);
\draw (-1.5,0)--(-0.5,0);
\draw (0.5,0)--(1.5,0);
\draw (-0.5,0.07)--(0.5,0.07);
\draw (-0.5,-0.07)--(0.5,-0.07);
\node [inner sep=0.8pt,outer sep=0.8pt] at (-1.5,0) (1) {$\bullet$};
\node [inner sep=0.8pt,outer sep=0.8pt] at (-0.5,0) (2) {$\bullet$};
\node [inner sep=0.8pt,outer sep=0.8pt] at (0.5,0) (3) {$\bullet$};
\node [inner sep=0.8pt,outer sep=0.8pt] at (1.5,0) (4) {$\bullet$};
\end{tikzpicture}&$399/127$&$4C$\\
\hline 
$\theta_6=x_2(1)x_3(1)$&\emph{no}&\begin{tikzpicture}[scale=0.5]
\node at (0,0.3) {};
\node [inner sep=0.8pt,outer sep=0.8pt] at (-1.5,0) (1) {$\bullet$};
\node [inner sep=0.8pt,outer sep=0.8pt] at (-0.5,0) (2) {$\bullet$};
\node [inner sep=0.8pt,outer sep=0.8pt] at (0.5,0) (3) {$\bullet$};
\node [inner sep=0.8pt,outer sep=0.8pt] at (1.5,0) (4) {$\bullet$};
\draw [line width=0.5pt,line cap=round,rounded corners,fill=ggrey] (1.north west)  rectangle (1.south east);
\draw [line width=0.5pt,line cap=round,rounded corners,fill=ggrey] (2.north west)  rectangle (2.south east);
\draw [line width=0.5pt,line cap=round,rounded corners,fill=ggrey] (3.north west)  rectangle (3.south east);
\draw [line width=0.5pt,line cap=round,rounded corners,fill=ggrey] (4.north west)  rectangle (4.south east);
\draw (-1.5,0)--(-0.5,0);
\draw (0.5,0)--(1.5,0);
\draw (-0.5,0.07)--(0.5,0.07);
\draw (-0.5,-0.07)--(0.5,-0.07);
\node [inner sep=0.8pt,outer sep=0.8pt] at (-1.5,0) (1) {$\bullet$};
\node [inner sep=0.8pt,outer sep=0.8pt] at (-0.5,0) (2) {$\bullet$};
\node [inner sep=0.8pt,outer sep=0.8pt] at (0.5,0) (3) {$\bullet$};
\node [inner sep=0.8pt,outer sep=0.8pt] at (1.5,0) (4) {$\bullet$};
\end{tikzpicture}&$151/151$&$4E$\\
\hline
\end{tabular}
$$
Moreover, $\theta_{3+i}^2\sim\theta_i$ for $i=1,2,3$, and $\theta_2=\sigma(\theta_1)$, $\theta_3=\sigma(\theta_3)$, $\theta_5=\sigma(\theta_4)$, and $\theta_6=\sigma(\theta_6)$. 
\end{thm}

\begin{proof}
We first show that the automorphisms have the claimed diagrams. Note that $\theta_1$, $\theta_2$, and $\theta_3$ are involutions, and hence the hypothesis of Lemma~\ref{lem:red2} applies. Consider $\theta_1$. Following the strategy of~(\ref{eq:calcw}) we notice that $\theta_1=x_{\varphi}(1)$ is central in $U$ (by the commutator formulae), and hence, for all $u\in U$, using~(\ref{eq:n}) we have
$$
w_0^{-1}u^{-1}\theta_1 uw_0=w_0^{-1}x_{\varphi}(1)w_0=x_{-\varphi}(1)=x_{\varphi}(1)n_{\varphi}(1)x_{\varphi}(1)\in Bs_{\varphi}B.
$$
Thus $\delta(gB,\theta_1 gB)=s_{\varphi}$ for all $gB\in Bw_0B/B$, and so $\disp(\theta)=\ell(s_{\varphi})=15$ (using Lemma~\ref{lem:red2}). Moreover, note that $s_{\varphi}=w_0w_{\{2,3,4\}}$ (for example, by computing inversion sets), and so there exists a non-domestic type $1$ vertex. All type $2$ or $3$ vertices are domestic, for if, for example, there is a non-domestic type $2$ vertex then there is $g\in G$ with $\delta(gB,\theta gB)\in w_0W_{\{1,3,4\}}$ and hence $\disp(\theta)\geq 24-4>15$. If there exists a non-domestic type $4$ vertex then by \cite[Lemma~4.5]{PVM:17a} there exists a non-domestic type $\{1,4\}$ simplex, which again contradicts the displacement calculation. Thus the diagram for $\theta_1$ is as claimed, and since $\theta_2=\sigma(\theta_1)$ (with $\sigma$ the graph automorphism) the result for $\theta_2$ also follows. 

Consider $\theta_3$. Since $x_{\varphi'}(1)$ is also central in $U$ (this special feature of characteristic~$2$ follows from the commutator relations) we see that $\theta_3$ is central in~$U$.  Thus, using commutator relations and~(\ref{eq:n}) we have
\begin{align*}
w_0^{-1}u^{-1}\theta_3 uw_0&=x_{-\varphi'}(1)x_{-\varphi}(1)\\
&=x_{-\varphi'}(1)x_{\varphi}(1)n_{\varphi}(1)x_{\varphi}(1)\\
&=x_{\varphi}(1)x_{(1110)}(1)x_{-(0122)}(1)x_{-\varphi'}(1)n_{\varphi}(1)x_{\varphi}(1)\\
&\in Bx_{-(0122)}(1)x_{-\varphi'}(1)s_{\varphi}B\\
&= Bs_{\varphi}x_{-(0122)}(1)x_{(1110)}(1)B\\
&=Bs_{\varphi}s_{(0122)}B.
\end{align*}
We have $s_{\varphi}s_{(0122)}=w_0w_{\{2,3\}}$ (for example, by computing the inversion sets), and hence there exists a non-domestic type $\{1,4\}$ simplex (see (\ref{eq:para})). By Lemma~\ref{lem:red2} the above calculation also shows that $\disp(\theta)=\ell(w_0w_{\{2,3\}})=20$, and the diagram of $\theta_3$ follows.

Consider $\theta_4$. We first show that $\theta_4$ is domestic. We will work with the conjugate 
$$
\theta_4'=x_{(1220)}(1)x_{1122}(1)=w^{-1}\theta_4 w\quad\text{where}\quad w=s_{(0110)}s_{(1242)}
$$
because this representative commutes with more elements $x_{\alpha}(1)$ with $\alpha\in R^+$, making~(\ref{eq:calcw}) more effective. Indeed $\theta_4'$ commutes with all $x_{\alpha}(1)$ with $\alpha\in R^+\backslash A$, where 
$$
A=\{(0100),(0001),(0110),(0011),(0120),(1220),(0122),(1122)\}.
$$
Then, as in~(\ref{eq:calcw2}), we have
$
w_0^{-1}u^{-1}\theta_4'uw_0=w_0^{-1}u_A^{-1}\theta_4'u_Aw_0.
$
There are $2^8$ distinct elements $u_A$, and using the Groups of Lie Type package in $\mathsf{MAGMA}$ we can easily verify that $w_0^{-1}u_A^{-1}\theta_4'u_Aw_0\notin Bw_0B$ for all $u_A$ (see the first author's webpage for the code). This implies that $\theta_4'$ is domestic, for if $\theta_4'$ were not domestic then the third hypothesis of Lemma~\ref{lem:red2} holds and hence there exists an element $u_A$ with $w_0^{-1}u_A^{-1}\theta_4'u_Aw_0\in Bw_0B$. 

One may see that $\theta_4'$ maps panels of cotypes $1$ and $2$ to opposites by simply exhibiting such panels (the Groups of Lie Type package is helpful here). Checking that there are no cotype $3$ or $4$ panels mapped to opposite panels is more complicated, and we have resorted to exhaustively verifying this by computation. However some efficiencies must be found to make the search feasible. Firstly, it is sufficient to check that there are no non-domestic type $\{1,2\}$ simplices (by a simple residue argument). Writing $P=P_{\{3,4\}}$, the (residues of the) type $\{1,2\}$ simplices of $\Delta$ are the cosets $gP$, $g\in G$. Let $T\subseteq W$ denote a transversal of minimal length representatives for cosets in $W/W_{\{3,4\}}$. A complete set of representatives for $P$ cosets in $G$ (and hence type $\{1,2\}$ simplices in $\Delta$) is 
$$
\{u_w(a)w \mid w\in T,\,a\in\mathbb{F}_2^{\ell(w)}\}\quad\text{where}\quad u_w(a)=x_{\beta_1}(a_1)\cdots x_{\beta_k}(a_k),
$$
where $R(w)=\{\beta_1,\ldots,\beta_k\}$ is the inversion set of $w$. Thus, using~(\ref{eq:para}), it is sufficient to check that $\delta(g,\theta_4'g)\notin w_0W_{\{3,4\}}$ for all $g=u_w(a)w$ with $w\in T$. However there are $4385745$ such elements $g$ (the cardinality of $G/P$) and this would be computationally expensive. Considerable efficiency can be gained by using the fact that the product $u_w(a)$ can be taken in any order (again, see \cite[Lemma~17]{St:16}). Thus, applying the technique~(\ref{eq:calcw2}), we only need to consider terms $u_w'(a)=x_{\gamma_1}(a_1)\cdots x_{\gamma_{\ell}}(a_{\ell})$ with $\{\gamma_1,\ldots,\gamma_{\ell}\}=R(w)\cap A$. This drastically reduces the number of cases needing checking. In fact it turns out that there are only $3885$ elements to check, and these are very quickly checked by the computer.

Since $\theta_5=\sigma(\theta_4)$ the result for $\theta_5$ follows. 

Consider $\theta_6$. Again we use a different conjugate
$
\theta_6\sim \theta_6'=x_{(1110)}(1)x_{(0122)}(1).
$
This element commutes with all $x_{\alpha}(1)$ with $\alpha\in R^+\backslash A$, where
$$
A=\{(0001),(0011),(0122),(0111),(0121),(1120),(1220),(1110),(1100),(1000)\}.
$$
A similar argument to before, this time checking $2^{10}$ cases, verifies that $\theta_6'$ (and hence $\theta_6)$ is domestic. It is then straightforward to provide panels of each cotype mapped onto opposites, and hence $\theta_6$ has the claimed diagram.

There are $95$ conjugacy classes in the group $\sF_4(2)$ (computed using the permutation representation), and for $88$ of these classes a quick search finds non-domestic chambers. The $7$ remaining classes must therefore be domestic, because the $6$ examples given above are clearly non-conjugate (they have distinct decorated opposition diagrams), and the identity is also trivially domestic. 

The number of fixed type $1$ vertices for each example is easily computed using the permutation representation, and the number of fixed type $4$ vertices is obtained by considering the dual. Finally the $\mathbb{ATLAS}$ classes can be determined by the orders and fixed structures. 
\end{proof}

Since no duality of a thick $\sF_4$ building is domestic the classification of domestic automorphisms of $\sF_4(2)$ is complete (see \cite[Lemma~4.1]{PVM:17a}). We also note that Lemma~\ref{lem:F41234} follows from the above classification.

We now consider the building $\sF_4(2,4)$. The full automorphism group of this building is ${^2}\sE_6(2^2).2$ (that is, ${^2}\sE_6(2^2)$ extended by the diagram automorphism $\sigma$ of $\sE_6$; see \cite[Section~10.4]{Tit:74} and \cite[page~191]{ATLAS}). We write $x_{\alpha}(a)$ for the Chevalley generators in the twisted group ${^2}\sE_6(2^2)$. Thus $a\in\mathbb{F}_2$ (respectively $a\in\mathbb{F}_4$) if $\alpha$ is a long root (respectively short root) of the twisted root system.

\begin{thm}\label{thm:F424}
Let $G= {^2\sE}_6(2^2)$, and let $\Delta=G/B$ be the associated building of type $\sF_4(2,4)$. Let $\varphi$ (respectively~$\varphi'$) be the highest root (respectively highest short root) of the $\sF_4$ root system. There are precisely $4$ classes of nontrivial domestic collineations, as follows (where $\sigma$ is the graph automorphism of $\sE_6$):
$$
\noindent\begin{tabular}{|l|l|l|l|l|}
\hline
$\theta$&\emph{capped}&\emph{diagram}&\emph{fixed points}&$\mathbb{ATLAS}$\\
\hline\hline
$\theta_1=x_{\varphi}(1)$&\emph{yes}&\begin{tikzpicture}[scale=0.5,baseline=-0.5ex]
\node at (0,0.3) {};
\node [inner sep=0.8pt,outer sep=0.8pt] at (-1.5,0) (1) {$\bullet$};
\node  [inner sep=0.8pt,outer sep=0.8pt] at (-0.5,0) (2) {$\bullet$};
\node [inner sep=0.8pt,outer sep=0.8pt] at (0.5,0) (3) {$\bullet$};
\node [inner sep=0.8pt,outer sep=0.8pt] at (1.5,0) (4) {$\bullet$};
\draw (-1.5,0)--(-0.5,0);
\draw (0.5,0)--(1.5,0);
\draw (-0.5,0.07)--(0.5,0.07);
\draw (-0.5,-0.07)--(0.5,-0.07);
\draw (-0.15,0.3) -- (0.08,0) -- (-0.15,-0.3);%arrow
\draw [line width=0.5pt,line cap=round,rounded corners] (1.north west)  rectangle (1.south east);
\end{tikzpicture}&$46135$&$2A$\\
\hline
$\theta_2=x_{\varphi'}(1)$&\emph{yes}&\begin{tikzpicture}[scale=0.5,baseline=-0.5ex]
\node at (0,0.3) {};
\node [inner sep=0.8pt,outer sep=0.8pt] at (-1.5,0) (1) {$\bullet$};
\node [inner sep=0.8pt,outer sep=0.8pt] at (-0.5,0) (2) {$\bullet$};
\node [inner sep=0.8pt,outer sep=0.8pt] at (0.5,0) (3) {$\bullet$};
\node [inner sep=0.8pt,outer sep=0.8pt] at (1.5,0) (4) {$\bullet$};
\draw (-1.5,0)--(-0.5,0);
\draw (0.5,0)--(1.5,0);
\draw (-0.5,0.07)--(0.5,0.07);
\draw (-0.5,-0.07)--(0.5,-0.07);
\draw (-0.15,0.3) -- (0.08,0) -- (-0.15,-0.3);%arrow
\draw [line width=0.5pt,line cap=round,rounded corners] (1.north west)  rectangle (1.south east);
\draw [line width=0.5pt,line cap=round,rounded corners] (4.north west)  rectangle (4.south east);
\end{tikzpicture}&$20279$&$2B$\\
\hline
$\theta_3=\sigma$&\emph{yes}&\begin{tikzpicture}[scale=0.5,baseline=-0.5ex]
\node [inner sep=0.8pt,outer sep=0.8pt] at (0,0.3) {};
\node [inner sep=0.8pt,outer sep=0.8pt] at (-1.5,0) (1) {$\bullet$};
\node [inner sep=0.8pt,outer sep=0.8pt] at (-0.5,0) (2) {$\bullet$};
\node [inner sep=0.8pt,outer sep=0.8pt] at (0.5,0) (3) {$\bullet$};
\node [inner sep=0.8pt,outer sep=0.8pt] at (1.5,0) (4) {$\bullet$};
\draw (-1.5,0)--(-0.5,0);
\draw (0.5,0)--(1.5,0);
\draw (-0.5,0.07)--(0.5,0.07);
\draw (-0.5,-0.07)--(0.5,-0.07);
\draw (-0.15,0.3) -- (0.08,0) -- (-0.15,-0.3);%arrow
\phantom{\draw [line width=0.5pt,line cap=round,rounded corners,fill=ggrey] (1.north west)  rectangle (1.south east);}
\draw [line width=0.5pt,line cap=round,rounded corners] (4.north west)  rectangle (4.south east);
\end{tikzpicture}&$69615$&$2E$\\
\hline
$\theta_4=x_1(1)x_2(1)$&\emph{no}&\begin{tikzpicture}[scale=0.5]
\node at (0,0.3) {};
\node [inner sep=0.8pt,outer sep=0.8pt] at (-1.5,0) (1) {$\bullet$};
\node [inner sep=0.8pt,outer sep=0.8pt] at (-0.5,0) (2) {$\bullet$};
\node [inner sep=0.8pt,outer sep=0.8pt] at (0.5,0) (3) {$\bullet$};
\node [inner sep=0.8pt,outer sep=0.8pt] at (1.5,0) (4) {$\bullet$};
\draw [line width=0.5pt,line cap=round,rounded corners,fill=ggrey] (1.north west)  rectangle (1.south east);
\draw [line width=0.5pt,line cap=round,rounded corners,fill=ggrey] (2.north west)  rectangle (2.south east);
\draw [line width=0.5pt,line cap=round,rounded corners] (3.north west)  rectangle (3.south east);
\draw [line width=0.5pt,line cap=round,rounded corners] (4.north west)  rectangle (4.south east);
\draw (-1.5,0)--(-0.5,0);
\draw (0.5,0)--(1.5,0);
\draw (-0.5,0.07)--(0.5,0.07);
\draw (-0.5,-0.07)--(0.5,-0.07);
\draw (-0.15,0.3) -- (0.08,0) -- (-0.15,-0.3);%arrow
\node [inner sep=0.8pt,outer sep=0.8pt] at (-1.5,0) (1) {$\bullet$};
\node [inner sep=0.8pt,outer sep=0.8pt] at (-0.5,0) (2) {$\bullet$};
\node [inner sep=0.8pt,outer sep=0.8pt] at (0.5,0) (3) {$\bullet$};
\node [inner sep=0.8pt,outer sep=0.8pt] at (1.5,0) (4) {$\bullet$};
\end{tikzpicture}&$855$&$4A$\\
\hline
\end{tabular}
$$
Here $x_{\alpha}(a)$ denote the Chevalley generators in the twisted group. Moreover, $\theta_4^2\sim\theta_1$. 
\end{thm}

\begin{proof}
The analysis for $\theta_1$ is similar to the analysis of $\theta_1$ for $\sF_4(2)$. Specifically, this element commutes with all terms $x_{\alpha}(a)$, and the result easily follows.

Consider $\theta_2$. This element commutes with all terms $x_{\alpha}(a)$ with $\alpha\in R^+$ except for $x_{(0010)}(a)$, $x_{(0110)}(a)$ and $x_{(1110)}(a)$ with $a\in\{\xi,\xi^2\}$ (where $\xi$ is a generator of $\FF_4^*$). By commutator relations, if $a\in\{\xi,\xi^2\}$ we have
\begin{align*}
x_{(0010)}(-a)\theta_2x_{(0010)}(a)&=\theta_2x_{\varphi-\alpha_1-\alpha_2}(1)\\
x_{(0110)}(-a)\theta_2x_{(0110)}(a)&=\theta_2x_{\varphi-\alpha_1}(1)\\
x_{(1110)}(-a)\theta_2x_{(1110)}(a)&=\theta_2x_{\varphi}(1),
\end{align*}
and it follows that for all $u\in U$ we have
\begin{align*}
w_0^{-1}u^{-1}\theta_2 uw_0=x_{-\varphi'}(1)x_{-\varphi+\alpha_1+\alpha_2}(a_1)x_{-\varphi+\alpha_1}(a_2)x_{-\varphi}(a_3)\quad\text{with}\quad a_1,a_2,a_3\in\{0,1\}.
\end{align*}
Considering each of the $8$ possibilities for the triple $(a_1,a_2,a_3)\in\FF_2^3$ we see that the maximum length of $w=\delta(uw_0B,\theta_2uw_0B)$ is $20$ with $w=s_{\varphi}s_{(0122)}$, and the result follows.

Consider $\theta_4$. This element is conjugate to $\theta_4'=x_{(1220)}(1)x_{(1122)}(1)$, and then an analysis very similar to the case of $\theta_4$ for $\sF_4(2)$ applies. In particular, with $A$ as in the $\sF_4(2)$ case, we need to check each of the elements $\delta(u_Aw_0B,\theta_4'u_Aw_0B)$. This time there are $2048=4^3\times 2^5$ elements $u_A$ to check (since there are $3$ roots in $A$ whose root subgroup is isomorphic to $\mathbb{F}_4$ and the remaining $5$ root subgroups are isomorphic to $\mathbb{F}_2$). A quick check with the computer shows that the maximum length of $\delta(u_Aw_0B,\theta_4'u_Aw_0B)$ is $23$, and hence $\theta_4'\sim \theta_4$ is domestic. Then necessarily $\theta_4$ maps no panels of cotypes $3$ or $4$ to opposite (by a simple residue argument), and then since $\disp(\theta_4)=23$ it is forced that there are panels of cotypes both $1$ and $2$ mapped onto opposites. 

Consider $\theta_3=\sigma$. This element acts on the untwisted group $\sE_6(4)$ as a symplectic polarity, and thus is $\{i\}$-domestic for $i\in\{2,3,4,5\}$ (see \cite{HVM:12}). It follows that $\sigma$ is $\{i\}$-domestic for $i\in\{1,2,3\}$ on the building $\sF_4(2,4)$, hence the result.

Thus the diagrams of the four automorphisms are as claimed. Next, as in the $\sF_4(2)$ example, we use the permutation representation of ${^2}\sE_6(2^2).2$ to compute a complete list of conjugacy class representatives of this group. It turns out that there are $189$ conjugacy classes, and for $184$ of these classes one can exhibit a chamber mapped onto an opposite chamber. Thus there are at most $4$ classes of nontrivial domestic collineations, and since the examples exhibited above are pairwise non-conjugate (by decorated opposition diagrams) the list is complete.

Finally, the calculation of the numbers of fixed points is immediate from the permutation representation, and the $\mathbb{ATLAS}$ classes can be determined by the orders and fixed structures.
\end{proof}

\begin{thm}
Let $G=\sE_6(2).2$, and let $\Delta=\sE_6(2)/B$ be the associated building of type $\sE_6(2)$. There are precisely $3$ classes of domestic dualities (up to conjugation in the full automorphism group), as follows:
$$
\noindent\begin{tabular}{|l|l|l|l|}
\hline
$\theta$&\emph{capped}&\emph{diagram}&\emph{order}\\
\hline\hline
$\theta_1=\sigma$&\emph{yes}&\begin{tikzpicture}[scale=0.5,baseline=-1.5ex]
\node at (0,0.3) {};
\node at (0,-1.2) {};
\node [inner sep=0.8pt,outer sep=0.8pt] at (-2,0) (1) {$\bullet$};
\node [inner sep=0.8pt,outer sep=0.8pt] at (-1,0) (3) {$\bullet$};
\node [inner sep=0.8pt,outer sep=0.8pt] at (0,0) (4) {$\bullet$};
\node [inner sep=0.8pt,outer sep=0.8pt] at (1,0) (5) {$\bullet$};
\node [inner sep=0.8pt,outer sep=0.8pt] at (2,0) (6) {$\bullet$};
\node [inner sep=0.8pt,outer sep=0.8pt] at (0,-1) (2) {$\bullet$};
\draw (-2,0)--(2,0);
\draw (0,0)--(0,-1);
\draw [line width=0.5pt,line cap=round,rounded corners] (1.north west)  rectangle (1.south east);
\draw [line width=0.5pt,line cap=round,rounded corners] (6.north west)  rectangle (6.south east);
\end{tikzpicture} &$2$\\
\hline
$\theta_2=x_1(1)\sigma$&\emph{no}&\begin{tikzpicture}[scale=0.5,baseline=-2ex]
\node at (0,0.3) {};
\node at (0,-1.3) {};
\node [inner sep=0.8pt,outer sep=0.8pt] at (-2,0) (1) {$\bullet$};
\node [inner sep=0.8pt,outer sep=0.8pt] at (-1,0) (3) {$\bullet$};
\node [inner sep=0.8pt,outer sep=0.8pt] at (0,0) (4) {$\bullet$};
\node [inner sep=0.8pt,outer sep=0.8pt] at (1,0) (5) {$\bullet$};
\node [inner sep=0.8pt,outer sep=0.8pt] at (2,0) (6) {$\bullet$};
\node [inner sep=0.8pt,outer sep=0.8pt] at (0,-1) (2) {$\bullet$};
\draw [line width=0.5pt,line cap=round,rounded corners,fill=ggrey] (1.north west)  rectangle (1.south east);
\draw [line width=0.5pt,line cap=round,rounded corners,fill=ggrey] (2.north west)  rectangle (2.south east);
\draw [line width=0.5pt,line cap=round,rounded corners,fill=ggrey] (3.north west)  rectangle (3.south east);
\draw [line width=0.5pt,line cap=round,rounded corners,fill=ggrey] (4.north west)  rectangle (4.south east);
\draw [line width=0.5pt,line cap=round,rounded corners,fill=ggrey] (5.north west)  rectangle (5.south east);
\draw [line width=0.5pt,line cap=round,rounded corners,fill=ggrey] (6.north west)  rectangle (6.south east);
\draw (-2,0)--(2,0);
\draw (0,0)--(0,-1);
\node [inner sep=0.8pt,outer sep=0.8pt] at (-2,0) (1) {$\bullet$};
\node [inner sep=0.8pt,outer sep=0.8pt] at (-1,0) (3) {$\bullet$};
\node [inner sep=0.8pt,outer sep=0.8pt] at (0,0) (4) {$\bullet$};
\node [inner sep=0.8pt,outer sep=0.8pt] at (1,0) (5) {$\bullet$};
\node [inner sep=0.8pt,outer sep=0.8pt] at (2,0) (6) {$\bullet$};
\node [inner sep=0.8pt,outer sep=0.8pt] at (0,-1) (2) {$\bullet$};
\end{tikzpicture}&$4$\\
\hline
$\theta_3=x_1(1)x_3(1)\sigma$&\emph{no}&\begin{tikzpicture}[scale=0.5,baseline=-2ex]
\node at (0,0.3) {};
\node at (0,-1.3) {};
\node [inner sep=0.8pt,outer sep=0.8pt] at (-2,0) (1) {$\bullet$};
\node [inner sep=0.8pt,outer sep=0.8pt] at (-1,0) (3) {$\bullet$};
\node [inner sep=0.8pt,outer sep=0.8pt] at (0,0) (4) {$\bullet$};
\node [inner sep=0.8pt,outer sep=0.8pt] at (1,0) (5) {$\bullet$};
\node [inner sep=0.8pt,outer sep=0.8pt] at (2,0) (6) {$\bullet$};
\node [inner sep=0.8pt,outer sep=0.8pt] at (0,-1) (2) {$\bullet$};
\draw [line width=0.5pt,line cap=round,rounded corners,fill=ggrey] (1.north west)  rectangle (1.south east);
\draw [line width=0.5pt,line cap=round,rounded corners,fill=ggrey] (2.north west)  rectangle (2.south east);
\draw [line width=0.5pt,line cap=round,rounded corners,fill=ggrey] (3.north west)  rectangle (3.south east);
\draw [line width=0.5pt,line cap=round,rounded corners,fill=ggrey] (4.north west)  rectangle (4.south east);
\draw [line width=0.5pt,line cap=round,rounded corners,fill=ggrey] (5.north west)  rectangle (5.south east);
\draw [line width=0.5pt,line cap=round,rounded corners,fill=ggrey] (6.north west)  rectangle (6.south east);
\draw (-2,0)--(2,0);
\draw (0,0)--(0,-1);
\node [inner sep=0.8pt,outer sep=0.8pt] at (-2,0) (1) {$\bullet$};
\node [inner sep=0.8pt,outer sep=0.8pt] at (-1,0) (3) {$\bullet$};
\node [inner sep=0.8pt,outer sep=0.8pt] at (0,0) (4) {$\bullet$};
\node [inner sep=0.8pt,outer sep=0.8pt] at (1,0) (5) {$\bullet$};
\node [inner sep=0.8pt,outer sep=0.8pt] at (2,0) (6) {$\bullet$};
\node [inner sep=0.8pt,outer sep=0.8pt] at (0,-1) (2) {$\bullet$};
\end{tikzpicture}&$8$\\
\hline
\end{tabular}
$$
%The square of $\theta_2$ is a domestic collineation. The square of $\theta_3$ is not domestic, however $\theta_3^4$ is domestic. The element $\theta_2=x_1(1)\sigma$ is conjugate to the $\sigma$-centralised element $[x_1(1)x_6(1)][x_3(1)x_5(1)]\sigma$, and it appears that $\theta_3$ is not conjugate to any $\sigma$-centralised element. 
\end{thm}

\begin{proof}
As noted in Theorem~\ref{thm:F424}, the element $\theta_1=\sigma$ acts as a symplectic polarity on $\sE_6(2)$, and thus has the diagram claimed (see \cite{HVM:12}). For the remaining cases $\theta_2$ and $\theta_3$ we note that it is easy to find vertices of each type mapped onto opposite vertices. Thus it remains to show that these dualities are domestic. The working here is slightly more complicated than the case of collineations of the $\sF_4$ buildings. Writing $\theta=\tilde{\theta}\sigma$ with $\tilde{\theta}\in G$, we need to show that $w_0^{-1}u^{-1}\tilde{\theta}u^{\sigma}w_0\notin Bw_0B$ for all $u\in U$ (here we are applying Lemma~\ref{lem:red2}). 

Consider $\theta_2$. We use the conjugate $\theta_2'=x_{\beta}(1)\sigma$ with $\beta=(111221)$. 
It turns out, by commutator relations, that if $u\in U$ is arbitrary then $u^{-1}x_{\beta}(1)u^{\sigma}$ can be written in the following form (where we use $\mathsf{MAGMA}$'s built-in lexicographic order on the positive roots $\alpha_1,\ldots,\alpha_{32}$):
\begin{align*}
&x_1(a_1)x_7(a_2)x_{12}(a_3)x_{18}(a_4)x_{23}(0)x_{17}(a_5)x_{22}(a_6)x_{27}(0)x_{26}(a_7)x_{30}(0)x_{29}(a_8)x_{32}(a_9)\\
&\quad x_{33}(a_9+1)x_{34}(a_{10})x_{35}(a_{11})x_{36}(a_{12})x_3(a_{13})x_9(a_{14})x_{13}(a_{15})x_{15}(0)x_{19}(0)x_{21}(a_4)x_{25}(a_6)\\
&\quad x_{24}(0)x_{28}(a_7)x_{31}(a_{16})x_4(0)x_{10}(a_{14})x_8(0)x_{14}(a_{15})x_{16}(a_3)x_{20}(a_5)x_5(a_{13})x_{11}(a_2)x_2(0)x_6(a_1),
\end{align*}
where $a_1,\ldots,a_{16}\in\mathbb{F}_2$. The point is that there are only $2^{16}$ such terms, rather than $2^{36}=|U|$ terms. It is then a quick check on the computer to verify that $\theta_2$ is domestic (and hence strongly exceptional domestic by Corollary~\ref{cor:app1}). 
 
The analysis of $\theta_3$ is slightly more challenging. Using the conjugate $\theta_3'=x_{\beta}(1)x_{\beta'}(1)\sigma$ with $\beta=(010111)$ and $\beta'=(001111)$ we see that $u^{-1}x_{\beta}(1)x_{\beta'}(1)u^{\sigma}$ can be written in a similar way to the $\theta_2$ case above, this time with $2^{22}$ degrees of freedom. The verification is $\theta_3$ is domestic is then a long search with the computer. The details are on the first author's webpage. 

To verify that our list of domestic examples is complete we again use explicit conjugacy class representatives computed from the minimal faithful permutation representation, as in the previous theorems. See the first author's webpage for the relevant code. Note that the character table of $\sE_6(2)$ is not printed in $\mathbb{ATLAS}$, and therefore it is not possible to provide the $\mathbb{ATLAS}$ conjugacy class names. 
\end{proof}

\begin{thm}
Let $G=\sE_6(2)$, and let $\Delta=G/B$ be the associated building of type $\sE_6(2)$. There are precisely $3$ classes of domestic collineations, as follows:
$$
\noindent\begin{tabular}{|l|l|l|l|l|}
\hline
$\theta$&\emph{capped}&\emph{diagram}&\emph{fixed points}&\emph{order}\\
\hline\hline
$\theta_1=x_1(1)$&\emph{yes}&\begin{tikzpicture}[scale=0.5,baseline=-0.5ex]
\node at (0,0.3) {};
\node at (0,0.7) {};
\node at (0,-0.7) {};
\node [inner sep=0.8pt,outer sep=0.8pt] at (-2,0) (2) {$\bullet$};
\node [inner sep=0.8pt,outer sep=0.8pt] at (-1,0) (4) {$\bullet$};
\node [inner sep=0.8pt,outer sep=0.8pt] at (0,-0.5) (5) {$\bullet$};
\node [inner sep=0.8pt,outer sep=0.8pt] at (0,0.5) (3) {$\bullet$};
\node [inner sep=0.8pt,outer sep=0.8pt] at (1,-0.5) (6) {$\bullet$};
\node [inner sep=0.8pt,outer sep=0.8pt] at (1,0.5) (1) {$\bullet$};
\draw (-2,0)--(-1,0);
\draw (-1,0) to [bend left=45] (0,0.5);
\draw (-1,0) to [bend right=45] (0,-0.5);
\draw (0,0.5)--(1,0.5);
\draw (0,-0.5)--(1,-0.5);
\draw [line width=0.5pt,line cap=round,rounded corners] (2.north west)  rectangle (2.south east);
\end{tikzpicture}&$10479$&$2$\\
\hline
$\theta_2=x_1(1)x_2(1)$&\emph{yes}&\begin{tikzpicture}[scale=0.5,baseline=-0.5ex]
\node at (0,0.3) {};
\node at (0,0.7) {};
\node at (0,-0.7) {};
\node [inner sep=0.8pt,outer sep=0.8pt] at (-2,0) (2) {$\bullet$};
\node [inner sep=0.8pt,outer sep=0.8pt] at (-1,0) (4) {$\bullet$};
\node [inner sep=0.8pt,outer sep=0.8pt] at (0,-0.5) (5) {$\bullet$};
\node [inner sep=0.8pt,outer sep=0.8pt] at (0,0.5) (3) {$\bullet$};
\node [inner sep=0.8pt,outer sep=0.8pt] at (1,-0.5) (6) {$\bullet$};
\node [inner sep=0.8pt,outer sep=0.8pt] at (1,0.5) (1) {$\bullet$};
\draw (-2,0)--(-1,0);
\draw (-1,0) to [bend left=45] (0,0.5);
\draw (-1,0) to [bend right=45] (0,-0.5);
\draw (0,0.5)--(1,0.5);
\draw (0,-0.5)--(1,-0.5);
\draw [line width=0.5pt,line cap=round,rounded corners] (2.north west)  rectangle (2.south east);
\draw [line width=0.5pt,line cap=round,rounded corners] (1.north west)  rectangle (6.south east);
\end{tikzpicture}&$2543$&$2$\\
\hline
$\theta_3=x_1(1)x_3(1)$&\emph{no}&\begin{tikzpicture}[scale=0.5,baseline=-0.5ex]
\node at (0,0.3) {};
\node [inner sep=0.8pt,outer sep=0.8pt] at (-2,0) (2) {$\bullet$};
\node [inner sep=0.8pt,outer sep=0.8pt] at (-1,0) (4) {$\bullet$};
\node [inner sep=0.8pt,outer sep=0.8pt] at (0,-0.5) (5) {$\bullet$};
\node [inner sep=0.8pt,outer sep=0.8pt] at (0,0.5) (3) {$\bullet$};
\node [inner sep=0.8pt,outer sep=0.8pt] at (1,-0.5) (6) {$\bullet$};
\node [inner sep=0.8pt,outer sep=0.8pt] at (1,0.5) (1) {$\bullet$};
\draw [line width=0.5pt,line cap=round,rounded corners,fill=ggrey] (2.north west)  rectangle (2.south east);
\draw [line width=0.5pt,line cap=round,rounded corners,fill=ggrey] (4.north west)  rectangle (4.south east);
\draw [line width=0.5pt,line cap=round,rounded corners] (3.north west)  rectangle (5.south east);
\draw [line width=0.5pt,line cap=round,rounded corners] (1.north west)  rectangle (6.south east);
\draw (-2,0)--(-1,0);
\draw (-1,0) to [bend left=45] (0,0.5);
\draw (-1,0) to [bend right=45] (0,-0.5);
\draw (0,0.5)--(1,0.5);
\draw (0,-0.5)--(1,-0.5);
\node at (0,-0.5) {$\bullet$};
\node at (0,0.5) {$\bullet$};
\node at (1,-0.5) {$\bullet$};
\node at (1,0.5) {$\bullet$};
\node [inner sep=0.8pt,outer sep=0.8pt] at (-2,0) (2) {$\bullet$};
\node [inner sep=0.8pt,outer sep=0.8pt] at (-1,0) (4) {$\bullet$};
\node [inner sep=0.8pt,outer sep=0.8pt] at (0,-0.5) (5) {$\bullet$};
\node [inner sep=0.8pt,outer sep=0.8pt] at (0,0.5) (3) {$\bullet$};
\node [inner sep=0.8pt,outer sep=0.8pt] at (1,-0.5) (6) {$\bullet$};
\node [inner sep=0.8pt,outer sep=0.8pt] at (1,0.5) (1) {$\bullet$};
\end{tikzpicture}&$847$&$4$\\
\hline
\end{tabular}
$$
%Moreover, $\theta_3^2\sim \theta_1$. 
\end{thm}

\begin{proof} 
To analyse $\theta_1$ we work with the conjugate $\theta_1\sim x_{\varphi}(1)$, where $\varphi$ is the highest root. Then an analysis very similar to the $\sF_4(2)$ case shows that $\theta_1$ has the diagram claimed. 

The analysis for $\theta_2$ can be done by hand. We work with the conjugate $\theta_2'=x_{\varphi}(1)x_{\varphi'}(1)$ where $\varphi$ is the highest root and $\varphi'=(101111)$ is the highest root of the $\sA_5$ subsystem. Let $u\in U$. By commutator relations and a simple induction we see that $u^{-1}\theta_2'u$ is a product of terms $x_{\alpha}(a)$ with $\alpha\geq \varphi'$ (with $\geq$ being the natural dominance order). In particular, each such $\alpha$ is in $R^+\backslash\sD_5$, where $\sD_5$ is the subsystem generated by $\alpha_2,\ldots,\alpha_6$. Let $v=w_0w_{\sD_5}$, where $w_{\sD_5}$ is the longest element of the parabolic subgroup $\langle s_2,\ldots,s_6\rangle$. Then $R^+\backslash\sD_5=\{\alpha\in R^+\mid v^{-1}\alpha\in-R^+\}$. It follows that $v^{-1}(w_0^{-1}u^{-1}\theta_2'uw_0)v\in B$ for all $u\in U$, and therefore
$$
w_0^{-1}u^{-1}\theta_2'uw_0\in vBv^{-1}\subseteq BvB\cdot Bv^{-1}B.
$$
Hence $w_0^{-1}u^{-1}\theta_2'uw_0\in BwB$ for some $w$ with $\ell(w)\leq 2\ell(v)=2(\ell(w_0)-\ell(w_{\sD_5}))=32$ (in fact we necessarily have strict inequality here by double coset combinatorics). Thus $\disp(\theta)\leq 32$, and it then follows from the classification of diagrams (and hence of possible displacements) that $\disp(\theta)\leq 30$. On the other hand, a quick calculation shows that $w_0^{-1}\theta_2'w_0\in Bs_{\varphi}s_{\varphi'}B$, and by computing inversion sets we have $s_{\varphi}s_{\varphi'}=w_0w_{\sA_3}$ (where $\sA_3$ is the subsystem generated by $\alpha_3,\alpha_4,\alpha_5$). Thus $\theta_2'$ maps the type $\{1,2,6\}$ simplex of the chamber $w_0B$ to an opposite simplex, hence the result. 
%
%For $\theta_2$ we work with the conjugate $\theta_2\sim\theta_2'=x_{\varphi}(1)x_{\varphi'}(1)$ where $\varphi$ is the highest root, and $\varphi'=(101111)$ is the highest root of the $\sA_5$ subsystem. Moreover, note that since $\theta_2$ is an involution we may apply Lemma~\ref{lem:red2}. The element $\theta_2'$ commutes with all $x_{\alpha}(a)$ with $\alpha\in R^+\backslash A$ where 
%$
%A=\{(010000),(010100),(011100),(010110),(011110),(011210)\}.
%$
%Thus, using the technique~\ref{eq:calcw2}, there are only $2^6$ terms $w_0^{-1}u^{-1}\theta_2 uw_0$ to consider, and a quick search shows that $\disp(\theta)=30$, from which the result follows. 
%

The working for $\theta_3$ is more involved. Here Lemma~\ref{lem:red2} cannot be applied, and it is not practical to directly check every chamber for domesticity (there are $3126356394525$ of them). Instead we argue in a similar fashion as we did for the collineation $\theta_4$ in Theorem~\ref{thm:F4}. First replace $\theta_3$ by the conjugate $\theta_3\sim \theta_3'=x_{(111210)}(1)x_{(011111)}(1)$. Then $\theta_3'$ commutes with all $x_{\alpha}(a)$ with $\alpha\in R^+\backslash A$ where 
$$
A=\{\alpha_1,\alpha_3,\alpha_4,\alpha_6,(000110),(000011),(101100),(101110),(001111),(0111111),(111210)\}.
$$
By a residue argument it is sufficient to show that there are no non-domestic type $\{2,4\}$ simplices (see the claim in the proof of Corollary~\ref{cor:maximal}). Again one cannot feasibly check all type $\{2,4\}$ simplices (there are $7089243525$ of them). However, as in Theorem~\ref{thm:F4}, with $T$ a transversal of minimal length representatives for the cosets in $W/W_{\{1,3,5,6\}}$, it is sufficient to check that $\delta(g,\theta_3'g)\notin w_0W_{\{1,3,5,6\}}$ for all $g=u_w'(a)w$ with $w\in T$ and $u_w'(a)=x_{\gamma_1}(a_1)\cdots x_{\gamma_{\ell}}(a_{\ell})$ with $\{\gamma_1,\ldots,\gamma_{\ell}\}=R(w)\cap A$. It turns out that there are only $64158$ such elements $g$, and they are readily checked by computer in under an hour. 
\end{proof}

\subsection{Automorphisms of small buildings of types $\sE_7$ and $\sE_8$}\label{sec:E7E8}

Consider the $\sE_7$ root system~$R$. Fix the ordering $\alpha_1,\ldots,\alpha_{63}$ of the positive roots according to increasing height, using the natural lexicographic order for roots of the same height (for example, $(1122100)<(1112110)$). Note that this is the inbuilt order in $\mathsf{MAGMA}$. With this order, the roots $\alpha_{44}=(1112111)$, $\alpha_{45}=(0112211)$, and $\alpha_{46}=(1122210)$ play an special role below.

\begin{thm}\label{thm:E72}
Let $\theta_1=x_{44}(1)x_{46}(1)$ and $\theta_2=x_{44}(1)x_{45}(1)x_{46}(1)$ in $\sE_7(2)$. Then $\theta_1$ and $\theta_2$ are uncapped with respective decorated opposition diagrams
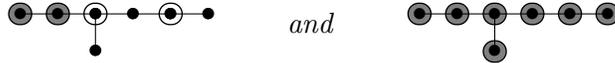
\begin{center}
\begin{tikzpicture}[scale=0.5,baseline=-1.5ex]
\node [inner sep=0.8pt,outer sep=0.8pt] at (-2,0) (1) {$\bullet$};
\node [inner sep=0.8pt,outer sep=0.8pt] at (-1,0) (3) {$\bullet$};
\node [inner sep=0.8pt,outer sep=0.8pt] at (0,0) (4) {$\bullet$};
\node [inner sep=0.8pt,outer sep=0.8pt] at (1,0) (5) {$\bullet$};
\node [inner sep=0.8pt,outer sep=0.8pt] at (2,0) (6) {$\bullet$};
\node [inner sep=0.8pt,outer sep=0.8pt] at (3,0) (7) {$\bullet$};
\node [inner sep=0.8pt,outer sep=0.8pt] at (0,-1) (2) {$\bullet$};
\draw [line width=0.5pt,line cap=round,rounded corners,fill=ggrey] (1.north west)  rectangle (1.south east);
\draw [line width=0.5pt,line cap=round,rounded corners,fill=ggrey] (3.north west)  rectangle (3.south east);
\draw [line width=0.5pt,line cap=round,rounded corners] (4.north west)  rectangle (4.south east);
\draw [line width=0.5pt,line cap=round,rounded corners] (6.north west)  rectangle (6.south east);
\node at (0,0) {$\bullet$};
\node at (2,0) {$\bullet$};
\node at (0,-1.3) {};
\node at (0,0.3) {};
\node [inner sep=0.8pt,outer sep=0.8pt] at (-2,0) (1) {$\bullet$};
\node [inner sep=0.8pt,outer sep=0.8pt] at (-1,0) (3) {$\bullet$};
\node [inner sep=0.8pt,outer sep=0.8pt] at (0,0) (4) {$\bullet$};
\node [inner sep=0.8pt,outer sep=0.8pt] at (1,0) (5) {$\bullet$};
\node [inner sep=0.8pt,outer sep=0.8pt] at (2,0) (6) {$\bullet$};
\node [inner sep=0.8pt,outer sep=0.8pt] at (3,0) (7) {$\bullet$};
\node [inner sep=0.8pt,outer sep=0.8pt] at (0,-1) (2) {$\bullet$};
\draw (-2,0)--(3,0);
\draw (0,0)--(0,-1);
\end{tikzpicture}
\qquad \text{and}\qquad\,
\begin{tikzpicture}[scale=0.5,baseline=-1.5ex]
\node [inner sep=0.8pt,outer sep=0.8pt] at (-2,0) (1) {$\bullet$};
\node [inner sep=0.8pt,outer sep=0.8pt] at (-1,0) (3) {$\bullet$};
\node [inner sep=0.8pt,outer sep=0.8pt] at (0,0) (4) {$\bullet$};
\node [inner sep=0.8pt,outer sep=0.8pt] at (1,0) (5) {$\bullet$};
\node [inner sep=0.8pt,outer sep=0.8pt] at (2,0) (6) {$\bullet$};
\node [inner sep=0.8pt,outer sep=0.8pt] at (3,0) (7) {$\bullet$};
\node [inner sep=0.8pt,outer sep=0.8pt] at (0,-1) (2) {$\bullet$};
\draw [line width=0.5pt,line cap=round,rounded corners,fill=ggrey] (1.north west)  rectangle (1.south east);
\draw [line width=0.5pt,line cap=round,rounded corners,fill=ggrey] (3.north west)  rectangle (3.south east);
\draw [line width=0.5pt,line cap=round,rounded corners,fill=ggrey] (4.north west)  rectangle (4.south east);
\draw [line width=0.5pt,line cap=round,rounded corners,fill=ggrey] (6.north west)  rectangle (6.south east);
\draw [line width=0.5pt,line cap=round,rounded corners,fill=ggrey] (2.north west)  rectangle (2.south east);
\draw [line width=0.5pt,line cap=round,rounded corners,fill=ggrey] (5.north west)  rectangle (5.south east);
\draw [line width=0.5pt,line cap=round,rounded corners,fill=ggrey] (7.north west)  rectangle (7.south east);
\node [inner sep=0.8pt,outer sep=0.8pt] at (-2,0) (1) {$\bullet$};
\node [inner sep=0.8pt,outer sep=0.8pt] at (-1,0) (3) {$\bullet$};
\node [inner sep=0.8pt,outer sep=0.8pt] at (0,0) (4) {$\bullet$};
\node [inner sep=0.8pt,outer sep=0.8pt] at (1,0) (5) {$\bullet$};
\node [inner sep=0.8pt,outer sep=0.8pt] at (2,0) (6) {$\bullet$};
\node [inner sep=0.8pt,outer sep=0.8pt] at (3,0) (7) {$\bullet$};
\node [inner sep=0.8pt,outer sep=0.8pt] at (0,-1) (2) {$\bullet$};
\draw (-2,0)--(3,0);
\draw (0,0)--(0,-1);
\end{tikzpicture} 
\end{center}
Moreover $\theta_1^2=\theta_2^2=x_{\varphi}(1)$ where $\varphi$ is the highest root, and hence $\theta_1$ and $\theta_2$ have order~$4$.
\end{thm}

\begin{proof}
Consider $\theta_2$ first. We show that $\theta_2$ is domestic using Lemma~\ref{lem:red2}. Applying~(\ref{eq:calcw2}) verbatim requires us to  check $2^{26}$ elements. The following modification of the theme is more efficient. It follows from commutator relations that
\begin{align*}
w_0^{-1}u^{-1}\theta_2 uw_0&=\prod_{\beta\in B}x_{-\beta}(a_{\beta}),
\end{align*}
where $B=\{\beta\in R^+\mid \beta\geq  \alpha_{44}\text{ or }\beta\geq  \alpha_{45}\text{ or }\beta\geq  \alpha_{46}\}$ (where here $\alpha\geq \beta$ if and only if $\alpha-\beta$ is a nonnegative combination of simple roots). There are $20$ roots in $B$. Moreover $a_{44}=a_{45}=a_{46}=1$ (by commutator relations), and so there remain only $2^{17}$ elements to consider. It is then readily checked by computer that $\theta_2$ is domestic, and we easily find vertices of each type mapped onto opposite vertices. Finally, commutator relations show that $\theta_2^2=x_{\varphi}(1)$. 

For $\theta_1$ we do a similar search to the above to show that $\theta_1$ is domestic. The remaining difficultly is showing that $\theta_1$ is $\{1,3\}$-domestic. Arguing as we did for $\theta_4$ in Theorem~\ref{thm:F4} it turns out that there are $1141419$ elements to check, and this can be done in an overnight run on the computer.
\end{proof}

Thus the proof of Theorem~\ref{thm:main*}(b) is complete. Our computational techniques are not efficient enough to handle the two diagrams for $\sE_8(2)$ due to the formidable size of the group. Thus for these diagrams we provide conjectural examples. For each of these conjectures we have randomly selected $10^5$ chambers and verified that restricted to this subset of the chamber set the structure of the automorphism is as claimed. 

Fix the ordering $\alpha_1,\ldots,\alpha_{120}$ of the positive roots of $\sE_8$ according to increasing height, using the natural lexicographic order for roots of the same height. Then the roots $\alpha_{88}=(11232221)$, $\alpha_{89}=(12243210)$ and $\alpha_{90}=(12233211)$ play a special role below. 

\begin{conjecture}\label{conj:2}
Let $\theta_1=x_{88}(1)x_{90}(1)$ and $\theta_2=x_{88}(1)x_{89}(1)x_{90}(1)$ in $\sE_8(2)$. Then $\theta_1$ and $\theta_2$ are uncapped with respective decorated opposition diagrams 
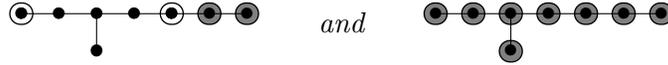
\begin{center}
\begin{tikzpicture}[scale=0.5,baseline=-1.5ex]
\node at (0,0.3) {};
\node at (0,-1.3) {};
\node [inner sep=0.8pt,outer sep=0.8pt] at (-2,0) (1) {$\bullet$};
\node [inner sep=0.8pt,outer sep=0.8pt] at (-1,0) (3) {$\bullet$};
\node [inner sep=0.8pt,outer sep=0.8pt] at (0,0) (4) {$\bullet$};
\node [inner sep=0.8pt,outer sep=0.8pt] at (1,0) (5) {$\bullet$};
\node [inner sep=0.8pt,outer sep=0.8pt] at (2,0) (6) {$\bullet$};
\node [inner sep=0.8pt,outer sep=0.8pt] at (3,0) (7) {$\bullet$};
\node [inner sep=0.8pt,outer sep=0.8pt] at (4,0) (8) {$\bullet$};
\node [inner sep=0.8pt,outer sep=0.8pt] at (0,-1) (2) {$\bullet$};
\draw [line width=0.5pt,line cap=round,rounded corners] (1.north west)  rectangle (1.south east);
\draw [line width=0.5pt,line cap=round,rounded corners] (6.north west)  rectangle (6.south east);
\draw [line width=0.5pt,line cap=round,rounded corners,fill=ggrey] (7.north west)  rectangle (7.south east);
\draw [line width=0.5pt,line cap=round,rounded corners,fill=ggrey] (8.north west)  rectangle (8.south east);
\draw (-2,0)--(4,0);
\draw (0,0)--(0,-1);
\node at (-2,0) {$\bullet$};
\node at (2,0) {$\bullet$};
\node [inner sep=0.8pt,outer sep=0.8pt] at (-2,0) (1) {$\bullet$};
\node [inner sep=0.8pt,outer sep=0.8pt] at (-1,0) (3) {$\bullet$};
\node [inner sep=0.8pt,outer sep=0.8pt] at (0,0) (4) {$\bullet$};
\node [inner sep=0.8pt,outer sep=0.8pt] at (1,0) (5) {$\bullet$};
\node [inner sep=0.8pt,outer sep=0.8pt] at (2,0) (6) {$\bullet$};
\node [inner sep=0.8pt,outer sep=0.8pt] at (3,0) (7) {$\bullet$};
\node [inner sep=0.8pt,outer sep=0.8pt] at (4,0) (8) {$\bullet$};
\node [inner sep=0.8pt,outer sep=0.8pt] at (0,-1) (2) {$\bullet$};
\end{tikzpicture}\qquad\text{and}\qquad
\begin{tikzpicture}[scale=0.5,baseline=-1.5ex]
\node [inner sep=0.8pt,outer sep=0.8pt] at (-2,0) (1) {$\bullet$};
\node [inner sep=0.8pt,outer sep=0.8pt] at (-1,0) (3) {$\bullet$};
\node [inner sep=0.8pt,outer sep=0.8pt] at (0,0) (4) {$\bullet$};
\node [inner sep=0.8pt,outer sep=0.8pt] at (1,0) (5) {$\bullet$};
\node [inner sep=0.8pt,outer sep=0.8pt] at (2,0) (6) {$\bullet$};
\node [inner sep=0.8pt,outer sep=0.8pt] at (3,0) (7) {$\bullet$};
\node [inner sep=0.8pt,outer sep=0.8pt] at (4,0) (8) {$\bullet$};
\node [inner sep=0.8pt,outer sep=0.8pt] at (0,-1) (2) {$\bullet$};
\draw [line width=0.5pt,line cap=round,rounded corners,fill=ggrey] (8.north west)  rectangle (8.south east);
\draw [line width=0.5pt,line cap=round,rounded corners,fill=ggrey] (7.north west)  rectangle (7.south east);
\draw [line width=0.5pt,line cap=round,rounded corners,fill=ggrey] (6.north west)  rectangle (6.south east);
\draw [line width=0.5pt,line cap=round,rounded corners,fill=ggrey] (5.north west)  rectangle (5.south east);
\draw [line width=0.5pt,line cap=round,rounded corners,fill=ggrey] (4.north west)  rectangle (4.south east);
\draw [line width=0.5pt,line cap=round,rounded corners,fill=ggrey] (3.north west)  rectangle (3.south east);
\draw [line width=0.5pt,line cap=round,rounded corners,fill=ggrey] (2.north west)  rectangle (2.south east);
\draw [line width=0.5pt,line cap=round,rounded corners,fill=ggrey] (1.north west)  rectangle (1.south east);
\draw (-2,0)--(4,0);
\draw (0,0)--(0,-1);
\node [inner sep=0.8pt,outer sep=0.8pt] at (-2,0) (1) {$\bullet$};
\node [inner sep=0.8pt,outer sep=0.8pt] at (-1,0) (3) {$\bullet$};
\node [inner sep=0.8pt,outer sep=0.8pt] at (0,0) (4) {$\bullet$};
\node [inner sep=0.8pt,outer sep=0.8pt] at (1,0) (5) {$\bullet$};
\node [inner sep=0.8pt,outer sep=0.8pt] at (2,0) (6) {$\bullet$};
\node [inner sep=0.8pt,outer sep=0.8pt] at (3,0) (7) {$\bullet$};
\node [inner sep=0.8pt,outer sep=0.8pt] at (4,0) (8) {$\bullet$};
\node [inner sep=0.8pt,outer sep=0.8pt] at (0,-1) (2) {$\bullet$};
\node at (0,0.3) {};
\node at (0,-1.3) {};
\end{tikzpicture} 
\end{center}
\end{conjecture}

We note that $\theta_1^2=\theta_2^2=x_{\varphi}(1)$ where $\varphi$ is the highest root, and hence $\theta_1$ and $\theta_2$ have order~$4$. It is not difficult to verify that $\Type(\theta_1)=\{1,6,7,8\}$ and $\Type(\theta_2)=\{1,2,3,4,5,6,7,8\}$. Thus the difficulty in the above conjecture is to show that $\theta_1$ is $\{7,8\}$-domestic, and that $\theta_2$ is domestic. In principle the approach taken for $\sE_7(2)$ is applicable, however in practice the enormous size of the group $\sE_8(2)$ makes the search impractical. For example, applying the technique of Theorem~\ref{thm:E72} to $\theta_2$ amounts to checking $2^{30}=1073741824$ elements. Each of these checks requires a sequence of commutator relations in the group $\sE_8(2)$, and while $\mathsf{MAGMA}$ has remarkably efficient algorithms implemented for this, the number of cases renders this computational approach unfeasible.

\begin{remark}
The examples of uncapped automorphisms that we have constructed thus far fix a chamber of the building. This is clear for the examples in exceptional types because the representatives are either in the Borel subgroup $B$, or are a composition of an element of $B$ with a standard graph automorphism. For the examples constructed in classical types we note that all examples have either order $4$ or $8$. It follows that they lie in a Sylow $2$-group of the automorphism group, and hence are conjugate to an element of $B$ (or $\langle B,\sigma\rangle$ in the case of an order~$2$ graph automorphism). However there do exist uncapped automorphisms that do not fix a chamber. For example, in $\sC_3(2)=\mathsf{Sp}_{6}(2)$ the element
$$
\theta=x_2(1)x_3(1)n_2=E_{11}+E_{23}+E_{24}+E_{25}+E_{32}+E_{33}+E_{45}+E_{54}+E_{55}+E_{66}
$$ 
is exceptional domestic (in fact strongly exceptional domestic), with order $6$. Thus $\theta$ does not lie in any conjugate of~$B$, and hence $\theta$ fixes no chamber of~$\sC_3(2)$. In fact the fixed structure of $\theta$ consists of three points $p_1$, $p_2$, $p_3$, a line $L$, and three planes $\pi_1$, $\pi_2$, and $\pi_3$ such that $\pi_1$, $\pi_2$ and $\pi_3$ intersect in $L$, $p_i\in\pi_i$ for $i=1,2,3$, and $p_i\notin L$ for $i=1,2,3$. 
\end{remark}

\bibliographystyle{plain}

%\bibliography{/Users/jamesp/Dropbox/Research/bibtex/Parkinson}
%\bibliography{/Users/jwparkinson/Dropbox/Research/bibtex/Parkinson}
%\bibliography{Parkinson}

\begin{thebibliography}{10}

\bibitem{AB:08}
P.~{A}bramenko and K.~Brown.
\newblock {\em Buildings: Theory and Applications}, volume 248.
\newblock Graduate Texts in Mathematics, Springer, 2008.

\bibitem{AB:09}
P.~{A}bramenko and K.~Brown.
\newblock Automorphisms of non-spherical buildings have unbounded displacement.
\newblock {\em Innov. Incidence Geom.}, 10:1--13, 2009.

\bibitem{MAGMA}
W.~{B}osma, J.~{C}annon, and C.~{P}layoust.
\newblock The magma algebra system {I}: The user language.
\newblock {\em J. Symbolic Comput.}, 24:235--265, 1997.

\bibitem{Car:89}
R.~{C}arter.
\newblock {\em {S}imple {G}roups of {L}ie {T}ype}.
\newblock Wiley Classics Library. John Wiley \& Sons, New York, 1989.

\bibitem{CMT:04}
A.~M. {C}ohen, S.~H. {M}urray, and D.~E. {T}aylor.
\newblock Computing in groups of {L}ie type.
\newblock {\em Mathematics of Computation}, 73(247):1477--1498, 2004.

\bibitem{ATLAS}
J.~H. {C}onway, R.~T. {C}urtis, S.~P. {N}orton, R.~A. {P}arker, and R.~A.
  {W}ilson.
\newblock {\em Atlas of Finite Groups}.
\newblock Clarendon Press, Oxford, 1985.

\bibitem{DPM:13}
A.~{D}evillers, J.~{P}arkinson, and H.~{V}an {M}aldeghem.
\newblock Automorphisms and opposition in twin buildings.
\newblock {\em J. Aust. Math. Soc.}, 94(2):189--201, 2013.

\bibitem{Lee:00}
B.~{L}eeb.
\newblock {\em A characterization of irreducible symmetric spaces and
  {E}uclidean buildings of higher rank by their asymptotic geometry}, volume
  326 of {\em Bonner Mathematische Schriften}.
\newblock Universit\"{a}t Bonn, 2000.

\bibitem{PTM:15}
J.~{P}arkinson, B.~{T}emmermans, and H.~{V}an {M}aldeghem.
\newblock The combinatorics of automorphisms and opposition in generalised
  polygons.
\newblock {\em Ann. Combin.}, 19(3):567--619, 2015.

\bibitem{PVM:17a}
J.~{P}arkinson and H.~{V}an {M}aldeghem.
\newblock Opposition diagrams for automorphisms of large spherical buildings.
\newblock {\em Journal of Combinatorial Theory, Series A}, 162:118--166, 2019.

\bibitem{St:16}
R.~{S}teinberg.
\newblock {\em Lectures on {C}hevalley Groups}, volume~66 of {\em University
  Lecture Series}.
\newblock AMS, 2016.

\bibitem{TTM:11}
B.~{T}emmermans, J.A. Thas, and H.~Van Maldeghem.
\newblock Domesticity in projective spaces.
\newblock {\em Innov. Incid. Geom.}, 12:141--149, 2011.

\bibitem{TTM:12}
B.~{T}emmermans, J.A. Thas, and H.~Van Maldeghem.
\newblock Collineations of polar spaces with restricted displacements.
\newblock {\em Des. Codes Cryptogr.}, 64:61--80, 2012.

\bibitem{TTM:12b}
B.~{T}emmermans, J.A. Thas, and H.~Van Maldeghem.
\newblock Domesticity in generalized quadrangles.
\newblock {\em Ann. Combin.}, 16:905--916, 2012.

\bibitem{Tit:74}
J.~{T}its.
\newblock {\em Buildings of spherical type and finite {BN}-pairs}, volume 386.
\newblock Lecture Notes in Mathematics, Springer-Verlag, 1974.

\bibitem{HVM:98}
H.~{V}an Maldeghem.
\newblock {\em Generalized Polygons}, volume~93 of {\em Monographs in
  Mathematics}.
\newblock Birkh\"{a}user, Basel, Boston, Berlin, 1998.

\bibitem{HVM:12}
H.~{V}an Maldeghem.
\newblock Symplectic polarities of buildings of type $\mathsf{E}_6$.
\newblock {\em Des. Codes Cryptogr.}, 65:115--125, 2012.

\bibitem{HVM:13}
H.~{V}an Maldeghem.
\newblock Characterizations of trialities of type $\mathsf{I}_{\mathrm{id}}$ in
  buildings of type $\mathsf{D_4}$.
\newblock In N.S.N. {S}astry, editor, {\em {G}roups of {E}xceptional {T}ype,
  {C}oxeter {G}roups, and {R}elated {G}eometries}, volume~82 of {\em {S}pringer
  {P}roc. {M}ath. {S}tat.}, pages 205--216. Springer, 2014.

\bibitem{Vas:96}
A.V. {V}asilyev.
\newblock Minimal permutation representations for finite simple exceptional
  groups of types ${G}_2$ and ${F}_4$.
\newblock {\em Algebra and Logic}, 35(6):371--383, 1996.

\bibitem{Vas:97}
A.V. {V}asilyev.
\newblock Minimal permutation representations for finite simple exceptional
  groups of types ${E}_6$, ${E}_7$, and ${E}_8$.
\newblock {\em Algebra and Logic}, 36(5):302--310, 1997.

\bibitem{Vas:98}
A.V. {V}asilyev.
\newblock Minimal permutation representations for finite simple exceptional
  twisted groups.
\newblock {\em Algebra and Logic}, 37(1):9--20, 1998.

\end{thebibliography}
\end{document}